\documentclass[final]{elsarticle}
\usepackage{amssymb}
\usepackage{algorithm, algorithmic} 
\usepackage{subfigure}
\usepackage{multicol, multirow}
\usepackage{ltablex}
\usepackage{float}

\usepackage{mathtools, amsthm, graphics, array}
\usepackage{diagbox}

\newtheorem{problem}{Prob.}
\newtheorem{theorem}{Theorem}
\newtheorem{property}{Property}
\newtheorem{lemma}{Lemma}

\usepackage[table]{xcolor}
\usepackage{tikz}
\usetikzlibrary{shapes.geometric, calc, backgrounds}

\journal{Elsevier}

\usepackage[colorlinks, linkcolor=blue, anchorcolor = blue, citecolor = blue]{hyperref}

\begin{document}

\begin{frontmatter}



\title{CPAFT: A Consistent Parallel Advancing Front Technique for 
Unstructured Triangular/Tetrahedral Mesh Generation}


\author[inst1]{Chengdi Ma}
\ead{mcd2020@stu.pku.edu.cn}
\author[inst2]{Jizu Huang}
\ead{huangjz@lsec.cc.ac.cn}
\author[inst1]{Hao Luo}
\ead{lhsms@pku.edu.cn}
\author[inst1,inst3]{Chao Yang\corref{cor1}}
\ead{chao\_yang@pku.edu.cn}
\cortext[cor1]{Corresponding author}

\affiliation[inst1]{organization={School of Mathematical Sciences},
            addressline={Peking University}, 
            city={Beijing},
            postcode={100871}, 
            country={P.R. China}}

\affiliation[inst2]{organization={Institute of Computational Mathematics and Scientific/Engineering Computing, Academy of
Mathematics and Systems Science},
            addressline={Chinese Academy of Sciences}, 
            city={Beijing},
            postcode={100190}, 
            country={P.R. China}}

\affiliation[inst3]{organization={PKU-Changsha Institute for Computing and Digital Economy},
            addressline={Peking University}, 
            city={Hunan},
            postcode={410006}, 
            country={P.R. China}}

\begin{abstract}
Compared with the remarkable progress made in parallel numerical solvers of partial differential equations, the development of algorithms for generating unstructured triangular/tetrahedral meshes has been relatively sluggish. In this paper, we propose a novel, consistent parallel advancing front technique (CPAFT) by combining the advancing front technique, the domain decomposition method based on space-filling curves, the distributed forest-of-overlapping-trees approach, and the consistent parallel maximal independent set algorithm. The newly proposed CPAFT algorithm can mathematically ensure that the generated unstructured triangular/tetrahedral meshes are independent of the number of processors and the implementation of domain decomposition. Several numerical tests are conducted to validate the parallel consistency and outstanding parallel efficiency of the proposed algorithm, which scales effectively up to two thousand processors. This is, as far as we know, the first parallel unstructured triangular/tetrahedral mesh generator with scalability to O(1,000) CPU processors.
\end{abstract}


\end{frontmatter}


\section{Introduction} 
Unstructured meshes have been playing an indispensable role in the processing of complex geometrical models, and are frequently employed in a large variety of applications including solid mechanics \citep{bailey1995finite, boscheri2022cell, demirdvzic1995numerical}, fluid dynamics \citep{bermudez1998upwind, di2009two, smolarkiewicz2016simulation}, and electromagnetics \citep{assous1992particle, ismagilov2015second}. In scientific and engineering computing, unstructured meshes are extensively utilized as the input for numerical simulations, particularly for solving partial differential equations (PDEs). To accurately capture the complex geometry and physical phenomena, it is often imperative to efficiently generate a large unstructured mesh with a substantial number of triangular/tetrahedral elements, while adhering to stringent quality standards. The computational complexity of state-of-the-art mesh generation algorithms often poses severe challenges in the generation of a large-scale, high-quality mesh on a single computing machine, rendering the mesh generation process inefficient and arduous. Despite that the computational simulations (i.e., PDEs solver) can be performed efficiently and in high parallelism \citep{gorobets2022heterogeneous, mullowney2021preparing, rasquin2011electronic, sahni2009scalable} with high-performance computing (HPC) hardware becoming increasingly powerful, the generation of unstructured meshes remains a significant bottleneck  \citep{lintermann2014massively, slotnick2014cfd}. 

During the past few decades, various  algorithms for efficient unstructured triangular/tetrahedral mesh generation have been proposed, of which the most popular ones are the Delaunay meshing method \citep{chernikov2009generalized, chew1997guaranteed, hang2015tetgen, shewchuk2002delaunay}, particle-based method \citep{fu2019isotropic, zheleznyakova2013molecular}, and advancing front technique (AFT) \citep{jin1993generation, lohner1996progress, lohner1988generation, schoberl1997netgen}. The Delaunay meshing method involves constructing a Delaunay tessellation by inserting vertices and edges into a coarse geometric representation, and then obtains unstructured triangular/tetrahedral meshes with high quality, which is mathematically guaranteed to avoid sliver elements \citep{chew1997guaranteed, chernikov2006parallel}. However, this method usually exhibits a main limitation on preserving the boundary integrity for non-convex domains, and the incorporation of non-local operations could inherently introduce round-off errors \citep{mavriplis1995advancing}. 
In the particle-based method, pair-wise forces are introduced to define the total energy or interpolation error, and the mesh generation process relies on the minimization of the aforementioned total energy or interpolation error to obtain the desired meshes, which could introduce a substantial computational overhead \citep{ji2020consistent}. When the AFT is used, the unstructured mesh is gradually generated from geometric boundaries through inward advancement, which is able to ensure the integrity of boundaries. The intersection judgments in AFT are only performed with neighboring faces (edges), thereby reducing the chances of failure induced by round-off errors \citep{mavriplis1995advancing}. Additionally, AFT does not require solving a minimization problem, thus it incurs a lower computational cost as compared to the particle-based method \citep{fu2019isotropic, zheleznyakova2013molecular}.
Notably, all these algorithms are sequential ones that can only be performed on a single machine. Therefore the generation of large-scale unstructured meshes using them could be substantially impeded by the memory and efficiency limitations of the hardware. 

In recent years, a series of efforts have been made in designing parallel unstructured triangular/tetrahedral mesh generation algorithms for large-scale unstructured meshes generation. Typical works include parallel Delaunay meshing methods  \citep{chrisochoides2003parallel, linardakis2006delaunay, said1999distributed, galtier1996prepartitioning, lo2012parallel}, parallel particle-based methods \citep{ji2020consistent, ji2021feature}, and the parallel AFTs \citep{ito2007parallel, lohner2001parallel, lohner2014recent, zhou2022saft}. However, it is still a key challenge to enable high degree of parallelism while retaining the mesh quality and the parallel consistency in the design of parallel mesh generation algorithms \citep{tsolakis2021parallel}.
For example, the well-known parallel Delaunay meshing methods, such as the parallel Delaunay domain decoupling method \citep{linardakis2006delaunay} and the parallel projective Delaunay meshing method \citep{galtier1996prepartitioning}, are able to overcome the memory and efficiency limitations and can scale to O(10) to O(100) processors by decomposing the geometry into continuous subdomains and meshing each subdomain separately. However, in order to construct a global Delaunay tessellation consistent with the sequential one, the subsequent merge of isolated patches into one coherent piece may lead to a large number of non-local operations, which could seriously hamper the level of parallelism \citep{ji2020consistent, lo2012parallel, ji2021feature}. The parallel particle-based mesh generation method can achieve parallelism utilizing hundreds of processors in large-scale scenarios by introducing an approximate pair-wise force that only relies on the local information \citep{ji2020consistent, ji2021feature}. Despite the fact that the parallel particle-based method ensures stringent mesh quality and parallel consistency without significantly deteriorating scalability, employing traditional particle sampling as the initial condition for complex geometries could result in substantial communication overhead \citep{ji2020consistent, ji2021feature}.

In comparison, the AFT framework enjoys two major advantages. 
One is that AFT can have higher degree of parallelism as compared to the Delaunay meshing methods due to its ability of iterative mesh generation based solely on neighboring faces (edges) \citep{mavriplis1995advancing, lohner2014recent, tsolakis2021parallel}. 
And the other advantage is that AFT does not require solving expensive physics-motivated governing equations for initial particle sampling as the parallel particle-based methods require, therefore the communication overhead is relatively low \citep{ji2020consistent}.
However, in contrast to the sequential AFT, it is required for parallel AFTs to deal with mesh intersections during parallel advancement, which is the major difficulty in the design of parallel AFTs  \citep{mavriplis1995advancing}. To address this issue, the shotgun AFT method utilizes fine-grained fragments that are sufficiently spaced apart to accommodate growth, with a side-effect that randomness is introduced into the framework \citep{zhou2022saft}, thereby leading to the lack of parallel consistency. Another strategy is the parallel AFT method based on the decomposition by coarse geometry \citep{ito2007parallel, lohner2014recent} or coarse octree \citep{lohner2001parallel}, which can restrict advancements to predetermined subdomains. While this strategy helps mitigate the mesh intersection problem, the generated meshes are influenced by the number of subdomains and different decomposition \citep{ito2007parallel, lohner2001parallel, lohner2014recent}. In essence, a major drawback of state-of-the-art parallel AFTs is that the meshes generated could exhibit variations depending on the number of processors utilized or the prior decompositions applied, i.e., they lack the parallel consistency \citep{ito2007parallel, lohner2001parallel, lohner2014recent, zhou2022saft}.

In this work, we propose a new algorithm called consistent parallel advancing front technique (CPAFT) for parallel  unstructured triangular/tetrahedral mesh generation, which can address the mesh intersection problem with parallel consistency, thus can obtain the same unstructured mesh as the sequential AFT does. We first introduce a non-overlapping domain decomposition method based on space-filling curves (SFC) \citep{aluru1997parallel, bader2012space, borrell2018parallel, sagan2012space} into the AFT framework, which can distribute the task of mesh generation to multiple processors and construct geometric invariant global index. Then, we extend each subdomain with $\delta$ layers of background meshes to a larger subdomain and obtain an overlapping domain decomposition. We combine the overlapping domain decomposition and the forest-of-octrees (or quadtree) approach \citep{isaac2015recursive} to propose a distributed forest-of-overlapping-trees approach, which can efficiently address the intersection judgments between newly generated elements and existing elements, especially for cases involving elements belonging to different subdomains. To avoid mutual mesh intersections between newly generated elements, we design a consistent parallel maximal independent set (MIS) algorithm by extending the original MIS algorithm \citep{luby1985simple} with the overlapping trees. With the SFC-based geometric invariant global index and the consistent parallel MIS algorithm, we demonstrate that the newly proposed CPAFT algorithm can obtain the same meshes as the meshes generated by the sequential AFT. Numerical experiments validate the parallel consistency and demonstrate that the CPAFT algorithm can scale well to over 2,000 processors. To the best of our knowledge, this is the first parallel unstructured triangular/tetrahedral mesh generator that can handle 3D complex geometries and scales to over O(1,000) processors. 

The remainder of the paper is organized as follows. In Section~\ref{section: preliminaries}, the unstructured triangular/tetrahedral meshes generation problem and the sequential AFT with mesh quality enhancement are introduced. Section~\ref{section: CPAFT} presents the SFC-based domain decomposition, the distributed forest-of-overlapping-trees approach, the consistent parallel MIS algorithm, and the workflow of the CPAFT algorithm. Results of a series of numerical experiments are presented in Section~\ref{sec: validation}. And the paper is concluded in Section~\ref{sec: conclusions}.


\section{The advancing front technique (AFT)}
\label{section: preliminaries}
In this paper, we consider the unstructured mesh generation problem for triangular/tetrahedral meshes. 
Let us assume that $\Omega\subset \mathbb{R}^{d}$ is a single connected domain with boundary $\partial\Omega\subset \mathbb{R}^{d-1}$. For multiple connected domains, the unstructured mesh can be generated separately. Initially, the boundary $\partial\Omega$ is divided into reasonably small face (edge) meshes, which provides an approximation of the geometric boundary. Let $\mathcal{F}_{0}$ denote the set of the face (edge) meshes. 
Let $\Omega_h$ denote the polyhedral domain formed by the boundary set $\mathcal{F}_{0}$, which serves as an approximation to $\Omega$. The problem of generating unstructured triangular/tetrahedral meshes for domain $\Omega_h$ is defined as follows.
\begin{problem}
For a given boundary set $\mathcal{F}_{0}$, find a partition $\mathbb{E}_h$ of domain $\Omega_h$ such that $\overline{\Omega_{h}} = \cup_{e_{k}\in \mathbb{E}_h} \overline{e_k}$, $e_{k_1}\cap e_{k_2}=\emptyset$ for $k_1\neq k_2$, and $\mathcal{F}_{0} =\cup_{e_{k}\in \mathbb{E}_h} \overline{e_k}\cap \partial\Omega_h $.
\label{prob}
\end{problem}

As a popular sequential algorithm for solving Prob.~\ref{prob}, the AFT \citep{jin1993generation, lohner1996progress, lohner1988generation, schoberl1997netgen} can be viewed as an iterative method. At each iteration, the it generates elements from the current front set $\mathcal{F}$, which is initially set as $\mathcal{F}_0$ and updated at each iteration. The AFT terminates until the entire domain is completely covered by elements. A notable advantage of the AFT is that points and elements are generated simultaneously, allowing for easy control of the shape and size of the elements during the mesh generation process.

\begin{figure}[H]
\centering
    \begin{tikzpicture} [xscale=2.0, yscale = 2.0, every node/.style={scale=1}]
        \coordinate (A1) at (1.25, 2.5-0.75);
        \coordinate (A2) at (1.55, 2.5-0.65);
        \coordinate (A3) at (1.75, 2.5-0.35);
        \coordinate (A4) at (1.85, 2.5-0.05);
        \coordinate (A5) at (2.15, 2.5-0.00);
        \coordinate (A6) at (2.45, 2.5-0.15);
        \coordinate (A7) at (2.60, 2.5-0.40);
        \coordinate (A8) at (2.65, 2.5-0.75);
        \coordinate (A9) at (2.50, 2.5-1.15);
        \coordinate (A10) at (2.52, 2.5-1.55);
        \coordinate (A11) at (2.35, 2.5-1.85);
        \coordinate (A12) at (2.25, 2.5-2.15);
        \coordinate (A13) at (2.10, 2.5-2.35);
        \coordinate (A14) at (1.80, 2.5-2.45);
        \coordinate (A15) at (1.45, 2.5-2.45);
        \coordinate (A16) at (1.15, 2.5-2.35);
        \coordinate (A17) at (0.85, 2.5-2.45);
        \coordinate (A18) at (0.55, 2.5-2.35);
        \coordinate (A19) at (0.35, 2.5-2.05);
        \coordinate (A20) at (0.05, 2.5-1.90);
        \coordinate (A21) at (0.00, 2.5-1.60);
        \coordinate (A22) at (0.15, 2.5-1.35);
        \coordinate (A23) at (0.35, 2.5-1.15);
        \coordinate (A24) at (0.50, 2.5-0.85);
        \coordinate (A25) at (0.75, 2.5-0.95);
        \coordinate (A26) at (0.95, 2.5-0.65);
        \draw[fill=gray, fill opacity=0.1] (A1)--(A2)--(A3)--(A4)--(A5)--(A6)--(A7)--(A8)--(A9)--(A10)--(A11)--(A12)--(A13)--(A14)--(A15)--(A16)--(A17)--(A18)--(A19)--(A20)--(A21)--(A22)--(A23)--(A24)--(A25)--(A26)-- cycle;
        \foreach \point in {A1, A2, A3, A4, A5, A6, A7, A8, A9, A10, A11, A12, A13, A14, A15, A16, A17, A18, A19, A20, A21, A22, A23, A24, A25, A26} {
            \filldraw[fill=red] (\point) circle (0.5pt);
        }
        \draw[thick, red, line width=0.45pt] (A1)--(A2)--(A3)--(A4)--(A5)--(A6)--(A7)--(A8)--(A9)--(A10)--(A11)--(A12)--(A13)--(A14)--(A15)--(A16)--(A17)--(A18)--(A19)--(A20)--(A21)--(A22)--(A23)--(A24)--(A25)--(A26)--(A1);

        \coordinate (B1) at (1.05, 2.5-1.20);
        \coordinate (B2) at (1.30, 2.5-1.26);
        \coordinate (B3) at (1.45, 2.5-1.43);
        \coordinate (B4) at (1.48, 2.5-1.68);
        \coordinate (B5) at (1.37, 2.5-1.89);
        \coordinate (B6) at (1.16, 2.5-2.00);
        \coordinate (B7) at (0.90, 2.5-1.95);
        \coordinate (B8) at (0.73, 2.5-1.79);
        \coordinate (B9) at (0.69, 2.5-1.54);
        \coordinate (B10) at (0.82, 2.5-1.30);
        \draw[fill=white, fill opacity=1.0] (B1)--(B2)--(B3)--(B4)--(B5)--(B6)--(B7)--(B8)--(B9)--(B10)--cycle;
        \foreach \point in {B1, B2, B3, B4, B5, B6, B7, B8, B9, B10} {
            \filldraw[fill=red] (\point) circle (0.5pt);
        }
        \draw[thick, red, line width=0.45pt] (B1)--(B2)--(B3)--(B4)--(B5)--(B6)--(B7)--(B8)--(B9)--(B10)--(B1);
        \node at (2.00, 1.65) {$\Omega_{h}$};
        \node at (0.20, 1.50) {$\mathcal{F}_{0}$};

        \draw[->] (2.7,0.9) -- (3.8,0.9) node[midway, above] {$1_{st}\ its$};

        \coordinate (AA1) at (1.25+4.0, 2.5-0.75);
        \coordinate (AA2) at (1.55+4.0, 2.5-0.65);
        \coordinate (AA3) at (1.75+4.0, 2.5-0.35);
        \coordinate (AA4) at (1.85+4.0, 2.5-0.05);
        \coordinate (AA5) at (2.15+4.0, 2.5-0.00);
        \coordinate (AA6) at (2.45+4.0, 2.5-0.15);
        \coordinate (AA7) at (2.60+4.0, 2.5-0.40);
        \coordinate (AA8) at (2.65+4.0, 2.5-0.75);
        \coordinate (AA9) at (2.50+4.0, 2.5-1.15);
        \coordinate (AA10) at (2.52+4.0, 2.5-1.55);
        \coordinate (AA11) at (2.35+4.0, 2.5-1.85);
        \coordinate (AA12) at (2.25+4.0, 2.5-2.15);
        \coordinate (AA13) at (2.10+4.0, 2.5-2.35);
        \coordinate (AA14) at (1.80+4.0, 2.5-2.45);
        \coordinate (AA15) at (1.45+4.0, 2.5-2.45);
        \coordinate (AA16) at (1.15+4.0, 2.5-2.35);
        \coordinate (AA17) at (0.85+4.0, 2.5-2.45);
        \coordinate (AA18) at (0.55+4.0, 2.5-2.35);
        \coordinate (AA19) at (0.35+4.0, 2.5-2.05);
        \coordinate (AA20) at (0.05+4.0, 2.5-1.90);
        \coordinate (AA21) at (0.00+4.0, 2.5-1.60);
        \coordinate (AA22) at (0.15+4.0, 2.5-1.35);
        \coordinate (AA23) at (0.35+4.0, 2.5-1.15);
        \coordinate (AA24) at (0.50+4.0, 2.5-0.85);
        \coordinate (AA25) at (0.75+4.0, 2.5-0.95);
        \coordinate (AA26) at (0.95+4.0, 2.5-0.65);

        \coordinate (BB1) at (1.05+4.0, 2.5-1.20);
        \coordinate (BB2) at (1.30+4.0, 2.5-1.26);
        \coordinate (BB3) at (1.45+4.0, 2.5-1.43);
        \coordinate (BB4) at (1.48+4.0, 2.5-1.68);
        \coordinate (BB5) at (1.37+4.0, 2.5-1.89);
        \coordinate (BB6) at (1.16+4.0, 2.5-2.00);
        \coordinate (BB7) at (0.90+4.0, 2.5-1.95);
        \coordinate (BB8) at (0.73+4.0, 2.5-1.79);
        \coordinate (BB9) at (0.69+4.0, 2.5-1.54);
        \coordinate (BB10) at (0.82+4.0, 2.5-1.30);

        \coordinate (C1) at (1.672522401095416+4.0, 0.9710854464188349);
        \coordinate (C2) at (1.606142355792568+4.0, 0.6104287181349322);
        \coordinate (C3) at (1.666277714796044+4.0, 1.216728247152696);
        \coordinate (C4) at (1.865689865537739+4.0, 1.896204592983967);
        \coordinate (C5) at (1.898079508133067+4.0, 1.122582631689564);
        \coordinate (C6) at (1.892157877967352+4.0, 0.8499316650531591);
        \coordinate (C7) at (2.001794919243112+4.0, 0.3799038105676661);
        \coordinate (C8) at (0.2323703349741227+4.0, 0.8714163514756005);
        \coordinate (C9) at (2.311813878543989+4.0, 2.103936994430781);
        \coordinate (C10) at (1.846657125226136+4.0, 1.39872731068162);
        \coordinate (C11) at (1.486560081971316+4.0, 1.575277417287268);
        \coordinate (C12) at (2.314499963042933+4.0, 1.642906896640141);
        \coordinate (C13) at (1.496095239645822+4.0, 0.42259481089203);
        \coordinate (C14) at (2.165560853342974+4.0, 0.9630781560133278);
        \coordinate (C15) at (0.6417636123997763+4.0, 0.4649136045918634);
        \coordinate (C16) at (1.746864202260333+4.0, 0.4126376084823277);
        \coordinate (C17) at (0.5774521389911499+4.0, 1.194465878591491);
        \coordinate (C18) at (1.240710723153943+4.0, 1.456979452971173);
        \coordinate (C19) at (2.042657883087715+4.0, 2.213863351922206);
        \coordinate (C20) at (0.4749440370858681+4.0, 0.8196034333966584);
        \coordinate (C21) at (0.8050839258464204+4.0, 0.3069772053208298);
        \coordinate (C22) at (1.019717020045174+4.0, 1.55050533434914);
        \coordinate (C23) at (2.146212724079962+4.0, 1.319755686439904);
        \coordinate (C24) at (1.836948171625927+4.0, 0.619105830287643);
        \coordinate (C25) at (2.081827390862245+4.0, 0.6057777727024676);
        \coordinate (C26) at (1.617297019277091+4.0, 0.2371226484138388);
        \coordinate (C27) at (0.618366212939725+4.0, 1.402329314927789);
        \coordinate (C28) at (1.726807743619687+4.0, 1.628287108069839);
        \coordinate (C29) at (2.008603844698391+4.0, 1.630443565890829);
        \coordinate (C30) at (1.493587743925357+4.0, 1.30097322780122);
        \coordinate (C31) at (0.4088944820160246+4.0, 1.060000078029669);
        \coordinate (C32) at (2.15406305675875+4.0, 1.897043316510624);
        \coordinate (C33) at (0.3093271138579296+4.0, 0.6821819715006911);
        \coordinate (C34) at (2.420453324454054+4.0, 1.899397339267561);
        \coordinate (C35) at (1.706946521223329+4.0, 0.778234059147593);
        \coordinate (C36) at (2.262943394355734+4.0, 1.145708460613308);
        \coordinate (C37) at (1.883191228156107+4.0, 0.2459328134927665);
        \coordinate (C38) at (0.8575913971170984+4.0, 1.396052553123671);
        \coordinate (C39) at (1.369519134311921+4.0, 0.2855574029357616);
        \coordinate (C40) at (0.993079368410553+4.0, 0.3279120180454128);
        \coordinate (C41) at (1.648254281128913+4.0, 1.432424714881378);
        \coordinate (C42) at (0.5054977485522383+4.0, 0.6225044222303509);
        \coordinate (C43) at (2.218617940407926+4.0, 2.271950086588246);
        \coordinate (C44) at (1.357063882100577+4.0, 0.4009573774305945);

        \draw[thick, blue, line width=0.45pt] (AA1)--(C11)--(AA2);
        \draw[fill=blue, fill opacity=0.2] (AA1)--(C11)--(AA2)--cycle;
        \draw[thick, blue, line width=0.45pt] (AA2)--(C4)--(AA3);
        \draw[fill=blue, fill opacity=0.2] (AA2)--(C4)--(AA3)--cycle;
        \draw[thick, blue, line width=0.45pt] (AA3)--(C19)--(AA4);
        \draw[fill=blue, fill opacity=0.2] (AA3)--(C19)--(AA4)--cycle;
        \draw[thick, blue, line width=0.45pt] (AA4)--(C19)--(AA5);
        \draw[fill=blue, fill opacity=0.2] (AA4)--(C19)--(AA5)--cycle;
        \draw[thick, blue, line width=0.45pt] (AA5)--(C43)--(AA6);
        \draw[fill=blue, fill opacity=0.2] (AA5)--(C43)--(AA6)--cycle;
        \draw[thick, blue, line width=0.45pt] (AA6)--(C9)--(AA7);
        \draw[fill=blue, fill opacity=0.2] (AA6)--(C9)--(AA7)--cycle;
        \draw[thick, blue, line width=0.45pt] (AA7)--(C34)--(AA8);
        \draw[fill=blue, fill opacity=0.2] (AA7)--(C34)--(AA8)--cycle;
        \draw[thick, blue, line width=0.45pt] (AA8)--(C12)--(AA9);
        \draw[fill=blue, fill opacity=0.2] (AA8)--(C12)--(AA9)--cycle;
        \draw[thick, blue, line width=0.45pt] (AA9)--(C36)--(AA10);
        \draw[fill=blue, fill opacity=0.2] (AA9)--(C36)--(AA10)--cycle;
        \draw[thick, blue, line width=0.45pt] (AA10)--(C14)--(AA11);
        \draw[fill=blue, fill opacity=0.2] (AA10)--(C14)--(AA11)--cycle;
        \draw[thick, blue, line width=0.45pt] (AA11)--(C25)--(AA12);
        \draw[fill=blue, fill opacity=0.2] (AA11)--(C25)--(AA12)--cycle;
        \draw[thick, blue, line width=0.45pt] (AA12)--(C7)--(AA13);
        \draw[fill=blue, fill opacity=0.2] (AA12)--(C7)--(AA13)--cycle;
        \draw[thick, blue, line width=0.45pt] (AA13)--(C37)--(AA14);
        \draw[fill=blue, fill opacity=0.2] (AA13)--(C37)--(AA14)--cycle;
        \draw[thick, blue, line width=0.45pt] (AA14)--(C26)--(AA15);
        \draw[fill=blue, fill opacity=0.2] (AA14)--(C26)--(AA15)--cycle;
        \draw[thick, blue, line width=0.45pt] (AA15)--(C39)--(AA16);
        \draw[fill=blue, fill opacity=0.2] (AA15)--(C39)--(AA16)--cycle;
        \draw[thick, blue, line width=0.45pt] (AA16)--(C40)--(AA17);
        \draw[fill=blue, fill opacity=0.2] (AA16)--(C40)--(AA17)--cycle;
        \draw[thick, blue, line width=0.45pt] (AA17)--(C21)--(AA18);
        \draw[fill=blue, fill opacity=0.2] (AA17)--(C21)--(AA18)--cycle;
        \draw[thick, blue, line width=0.45pt] (AA18)--(C15)--(AA19);
        \draw[fill=blue, fill opacity=0.2] (AA18)--(C15)--(AA19)--cycle;
        \draw[thick, blue, line width=0.45pt] (AA19)--(C33)--(AA20);
        \draw[fill=blue, fill opacity=0.2] (AA19)--(C33)--(AA20)--cycle;
        \draw[thick, blue, line width=0.45pt] (AA20)--(C8)--(AA21);
        \draw[fill=blue, fill opacity=0.2] (AA20)--(C8)--(AA21)--cycle;
        \draw[thick, blue, line width=0.45pt] (AA21)--(C8)--(AA22);
        \draw[fill=blue, fill opacity=0.2] (AA21)--(C8)--(AA22)--cycle;
        \draw[thick, blue, line width=0.45pt] (AA22)--(C31)--(AA23);
        \draw[fill=blue, fill opacity=0.2] (AA22)--(C31)--(AA23)--cycle;
        \draw[thick, blue, line width=0.45pt] (AA23)--(AA25);
        \draw[fill=blue, fill opacity=0.2] (AA23)--(AA24)--(AA25)--cycle;
        \draw[thick, blue, line width=0.45pt] (AA25)--(C22)--(AA26);
        \draw[fill=blue, fill opacity=0.2] (AA25)--(C22)--(AA26)--cycle;
        \draw[thick, blue, line width=0.45pt] (AA26)--(C22)--(AA1);
        \draw[fill=blue, fill opacity=0.2] (AA26)--(C22)--(AA1)--cycle;

        \draw[thick, blue, line width=0.45pt] (BB1)--(C18)--(BB2);
        \draw[fill=blue, fill opacity=0.2] (BB1)--(C18)--(BB2)--cycle;
        \draw[thick, blue, line width=0.45pt] (BB2)--(C30)--(BB3);
        \draw[fill=blue, fill opacity=0.2] (BB2)--(C30)--(BB3)--cycle;
        \draw[thick, blue, line width=0.45pt] (BB3)--(C1)--(BB4);
        \draw[fill=blue, fill opacity=0.2] (BB3)--(C1)--(BB4)--cycle;
        \draw[thick, blue, line width=0.45pt] (BB4)--(C2)--(BB5);
        \draw[fill=blue, fill opacity=0.2] (BB4)--(C2)--(BB5)--cycle;
        \draw[thick, blue, line width=0.45pt] (BB5)--(C44)--(BB6);
        \draw[fill=blue, fill opacity=0.2] (BB5)--(C44)--(BB6)--cycle;
        \draw[thick, blue, line width=0.45pt] (BB6)--(C40)--(BB7);
        \draw[fill=blue, fill opacity=0.2] (BB6)--(C40)--(BB7)--cycle;
        \draw[thick, blue, line width=0.45pt] (BB7)--(C15)--(BB8);
        \draw[fill=blue, fill opacity=0.2] (BB7)--(C15)--(BB8)--cycle;
        \draw[thick, blue, line width=0.45pt] (BB8)--(C20)--(BB9);
        \draw[fill=blue, fill opacity=0.2] (BB8)--(C20)--(BB9)--cycle;
        \draw[thick, blue, line width=0.45pt] (BB9)--(C17)--(BB10);
        \draw[fill=blue, fill opacity=0.2] (BB9)--(C17)--(BB10)--cycle;
        \draw[thick, blue, line width=0.45pt] (BB10)--(C38)--(BB1);
        \draw[fill=blue, fill opacity=0.2] (BB10)--(C38)--(BB1)--cycle;

        \foreach \point in {C1, C2, C4, C7, C8, C9, C11, C12, C14, C15, C17, C18, C19, C20, C21, C22, C25, C26, C30, C31, C33, C34, C36, C37, C38, C39, C40, C43, C44} {
            \filldraw[fill=blue] (\point) circle (0.65pt);
        }

        \foreach \point in {AA1, AA2, AA3, AA4, AA5, AA6, AA7, AA8, AA9, AA10, AA11, AA12, AA13, AA14, AA15, AA16, AA17, AA18, AA19, AA20, AA21, AA22, AA23, AA24, AA25, AA26} {
            \filldraw[fill=red] (\point) circle (0.5pt);
        }
        \draw[thick, red, line width=0.45pt] (AA1)--(AA2)--(AA3)--(AA4)--(AA5)--(AA6)--(AA7)--(AA8)--(AA9)--(AA10)--(AA11)--(AA12)--(AA13)--(AA14)--(AA15)--(AA16)--(AA17)--(AA18)--(AA19)--(AA20)--(AA21)--(AA22)--(AA23)--(AA24)--(AA25)--(AA26)--(AA1);
        
        \foreach \point in {BB1, BB2, BB3, BB4, BB5, BB6, BB7, BB8, BB9, BB10} {
            \filldraw[fill=red] (\point) circle (0.5pt);
        }
        \draw[thick, red, line width=0.45pt] (BB1)--(BB2)--(BB3)--(BB4)--(BB5)--(BB6)--(BB7)--(BB8)--(BB9)--(BB10)--(BB1);
    \end{tikzpicture}
    \caption{Example of advancing front. Blue points on the right-hand side are the advancing points generated in the first iteration of the AFT.}
    \label{fig: AFT}
\end{figure}

In what follows, we assume the normal direction of each front points towards the interior of the domain $\Omega_h$, which servers as the direction of an advancement during one iteration. As shown in Figure~\ref{fig: AFT}, during each iteration of the AFT, the front set $\mathcal{F}$ is updated and traversed by searching for potential advancing points, which can be categorized into two cases: new potential points on the perpendicular and existing vertices of nearby fronts. To control the quality of elements generated by the AFT, two criteria \citep{zhou2022saft, zienkiewicz2005finite} are adopted to determine whether a potential point is a legal advancement. Assuming that the maximum scale is $h_M$, the minimum scale is $h_m$, and the local scale near front $f \in \mathcal{F}$ is denoted as $h$, the widely used two criteria can be described as follows.

\smallskip
\textbf{Criterion A} (for the new advancing points).
\begin{itemize}
    \item The newly generated front should not intersect with other nearby fronts, and the newly generated elements should not contain any other fronts.
    \item The distance between nearby vertices and the newly generated fronts should be no less than $\beta_{1} h$, and the distance between the advancing point $p$ and the nearby fronts should be no less than $\beta_{1} h$.
    \item If the triangular/tetrahedral element composed of the advancing point $p$ and a nearby front $f$ does not intersect with any other fronts, and the distance between $p$ and the center of front $f$ is no less than $R(h_M, h_m, h)$, then it is required that the perpendicular distance from $p$ to $f$ should be no less than $\eta h$.
\end{itemize}

\textbf{Criterion B} (for the existing advancing points).
\begin{itemize}
    \item The newly generated fronts should not intersect with other nearby fronts, and the newly generated elements should not contain any other fronts.
    \item The distance between nearby vertices and the newly generated fronts should be no less than $\beta_{2} h$.
    \item The distance between the advancing point $p$ and the center of front $f$ should be no less than $R(h_M, h_m, h)$, while the perpendicular distance from $p$ to $f$ is no less than $\eta h$.
    \item (Optional) The newly generated fronts should not cause mutual obstruction with nearby fronts.
\end{itemize}
In the above two criteria, the parameters $\beta_{1}$ and $\beta_{2}$, which are both in the range of $(0, 1)$, are used to control the mesh quality. Meanwhile, the parameter $\eta\in (0, 1)$ ensures a lower bound on the measurement of the newly generated face, denoted as $\epsilon h^{d-1}$. $R(h_M, h_m, h)$ is a user-defined function to control the size of the element, such as $R(h_M, h_m, h)=\min(3h, 1.5h_M)$. The Minkowski difference and Gilbert--Johnson--Keerthi (GJK) algorithm \citep{gilbert1988fast} are utilized to determine whether two polyhedrons intersect with each other.

After the AFT terminating, a set of triangles/tetrahedra is generated. According to ref. \citep{cavalcante2001algorithm}, the mesh quality of triangles/tetrahedra can be quantified by  
\begin{equation}
  \alpha = {d} \frac{R_i}{R_c},
  \label{equ: quality}
\end{equation}
where $R_i$ and $R_c$ denote the radii of the inscribed and circumscribed spheres of triangles/tetrahedra in $\mathbb{R}^{d} (d=2,3)$, respectively.
The range of $\alpha$ is $(0, 1]$, and its optimal value is $1.0$ when the triangle/tetrahedron is regular. In this work, we require that the mesh quality indicator $\alpha$ should be greater than 0.3. To achieve this goal, a mesh quality optimization technique, such as surface preserving Laplacian smoothing approach \citep{zienkiewicz2005finite, cavalcante2001algorithm, cavendish1974automatic}, should be employed to enhance the mesh quality. In the $(k+1)_{th}$ iteration of the surface preserving Laplacian smoothing approach, the mesh quality is optimized by re-positioning the internal node $p_i^{(k)}$ to
\begin{equation}
 p_i^{(k+1)} = \frac{1}{N}\sum_{j=0}^{N-1} p_{n(j)}^{(k)},
 \label{equ: smooth}
\end{equation}
where $ p_{n(j)}^{(k)}~\hbox{with} \ j = 0, \dots, N-1, $ represent the positions of the $N$ neighboring nodes of the internal node $p_i^{(k)}$ at the $k_{th}$ iteration step. As stated in refs. \citep{zienkiewicz2005finite, cavalcante2001algorithm, cavendish1974automatic}, significant improvements in mesh quality can be achieved with only a few iterations.

The two criteria and the mesh quality enhancement approach play a crucial role in ensuring the quality of the newly generated mesh. Among these criteria, the intersection judgments are particularly important. {The intersection judgments are divided into two categories: intersections between newly generated elements and existing elements, and mutual intersections between newly generated elements.} Notably, the intersection judgment introduces irregular communication, which poses challenges in designing parallel AFT algorithms. Currently, only a few parallel AFT algorithms exist, such as those proposed in refs. \citep{ito2007parallel, lohner2001parallel, lohner2014recent, zhou2022saft}. The unstructured triangular/tetrahedral meshes generated by these parallel AFT algorithms are usually not consistent with those obtained by  sequential AFT, i.e., the parallel consistency is not guaranteed. In the next section, we propose the CPAFT algorithm, which can generate unstructured triangular/tetrahedral meshes consistent with the sequentially generated ones.

\section{The consistent parallel advancing front technique (CPAFT)}
\label{section: CPAFT}

In this section, we present a novel consistent parallel AFT (CPAFT) for generating unstructured triangular/tetrahedral meshes. To address the mesh intersection issue with parallel consistency, we combine three key techniques: the SFC-based domain decomposition, the distributed forest-of-overlapping-trees approach, and the consistent parallel MIS approach. 

\subsection{Domain decomposition based on the space-filling curve (SFC)}
\label{subsection: DD-SFC}

We first employ the non-overlapping domain decomposition method to divide the domain $\Omega_h$ into multiple subdomains, with each subdomain being handled by a processor. The non-overlapping domain decomposition enables distributed storage of the mesh information during the generation process, thus can overcome the memory limitations in large-scale mesh generation. The partitioning of the domain $\Omega_h$ is performed on top of the Cartesian background meshes, as depicted in Figure~\ref{fig:DD-SFC}. The Cartesian background meshes are geometrically sequenced small boxes generated by the SFC approach, which is a popular method for mesh partitioning, i.e., domain decomposition of pre-generated meshes \citep{aluru1997parallel, bader2012space, borrell2018parallel, sagan2012space}. 
To the best of our knowledge, this is the first attempt to integrate the SFC approach into the mesh generation process to achieve parallelism, instead of performing partitioning after the mesh is generated.
 
Consider a ${\mathcal{L}}_{th}$ level SFC with ${\cal B}=\{b_k,\ k = 0, \dots, (2^{\cal L})^{d}-1\}$ representing the set of all small boxes. Let $\Omega_h^i$ with $i=1,\ldots, n$ correspond to a non-overlapping domain decomposition of $\Omega_h$. As the level of SFC influences the balance of domain decomposition, it is necessary for the level of SFC to be matched with the mesh scale $h$. Each subdomain $\Omega_h^i$ is covered by a union of multiple small boxes from ${\cal B}$, denoted as ${\cal B}_{i}$ (called background subdomain). Additionally, the domain decomposition must satisfy the constraint that for any two small boxes $b_{k(i)}\in{\cal B}_{i}$ and $b_{k(j)}\in{\cal B}_{j}$ with $i<j$, it holds that $k(i) < k(j)$. According to the domain decomposition, the front set ${\cal F}$ at each iteration is divided into $n$ disjoint subsets, denoted as ${\cal F}_1,\ldots,{\cal F}_n$. The $n$ subsets satisfy that
\begin{equation}
\begin{aligned}
&\mathcal{F}_1 \cup \mathcal{F}_2 \cup \dots \cup \mathcal{F}_n = \mathcal{F},\\
&\mathcal{F}_i \cap \mathcal{F}_j = \emptyset, \ \ \ \ \  \forall i \neq j.
\end{aligned}
\end{equation}

Analogous to $\Omega_h^i$, the subset ${\cal F}_i$ is also assigned to the $i_{th}$ processor. Let $|{\cal F}_i|$ represent the number of fronts in ${\cal F}_i$ and $N_e^i$ denote the number of generated elements stored in processor $i$. 
We introduce an indicator $W_i = k_f |{\cal F}_i| + k_e N_e^i$ for $\Omega_h^i$, where $k_f$ and $k_e$ are typically set to $(3, 1)$. To achieve good parallel efficiency, we divide the domain $\Omega_h$ based on the following three criteria.

\smallskip
\textbf{Criteria for domain decomposition} 
\begin{itemize}
    \item \textbf{C1.} Each subdomain $\Omega_h^i$ is an intersection between $\Omega_h$ and a union of some background meshes connected by an SFC.
    \item \textbf{C2.} For the requirement of parallel efficiency, the indicators must satisfy the constraint $\min\limits_{1\leq i \leq n} W_i > 0.5\ \max\limits_{1\leq i \leq n}W_i$. The threshold parameter $0.5$ can be adjusted based on the specific geometry of the domain $\Omega_h$.
    \item \textbf{C3.} Furthermore, the area of the interface between subdomains should be minimized.
\end{itemize}
Since the front set ${\cal F}$ is updated dynamically at each iteration, the indicator $W_i$ will also change accordingly. In this study, we repartition the domain decomposition based on \textbf{criterion C2} every few iteration steps, to maintain load balance. 

\ 

\begin{figure}[H]
		\centering
		\subfigure[Decomposition for $\Omega_h$ and current front]{
		\begin{minipage}[c]{0.465\textwidth}
		  \centering
            \includegraphics[width=0.85\textwidth]{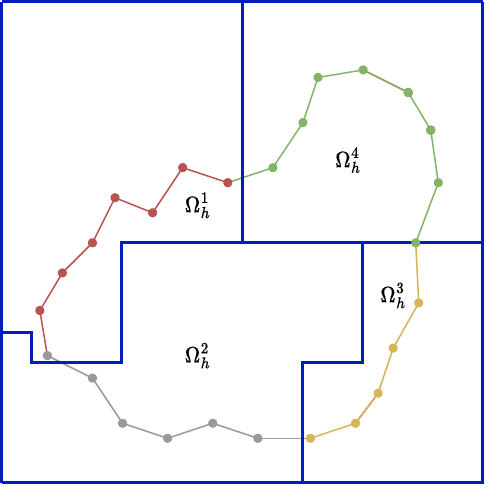}
		\end{minipage}
		} \subfigure[SFC-based Decomposition for background]{
		\begin{minipage}[c]{0.465\textwidth}
		  \centering
            \includegraphics[width=0.85\textwidth]{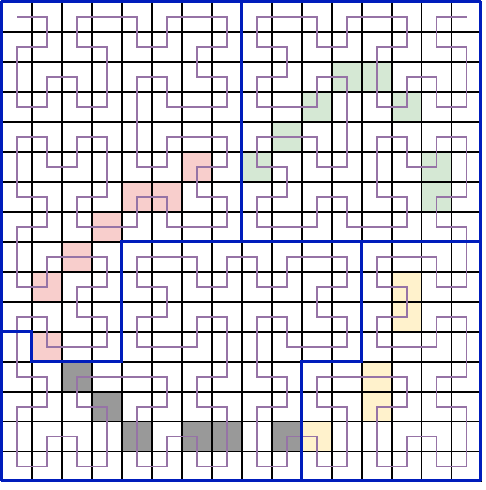}
		\end{minipage}
		} 
\caption{(a) Domain decomposition for a 2D domain with 4 processors: $\Omega_h=\cup_{i=1}^4 \Omega_h^i$. The blue lines in the plots are the boundaries of background subdomains. The current fronts are simultaneously assigned to four processors. The fronts with the same color are handled by the same processor. (b) Domain decomposition on top of the Cartesian background meshes based on a Hilbert-SFC. The boxes filled with different colors correspond to the current fronts.}
\label{fig:DD-SFC}
\end{figure}

\ 

For a given set of subsets ${\cal F}_i$ with $i=1,\ldots, n$, we define the local index of front $f\in {\cal F}_k$ as $\textsl{\textbf{Id}}(f, k; n)$. The local index $\textsl{\textbf{Id}}(f, k; n)$ sorts from 0 to $|{\cal F}_k|-1$ in ascending order of small box index, where $f$ is embedded, as shown in Figure~\ref{fig:DD-SFC}. The global index of $f$ is then set as: 
$$\textbf{GI}(f):=\textsl{\textbf{Id}}(f, k; n) + \sum\limits_{i=1}^{k-1}|\mathcal{F}_{i}|.$$ 
According to the following Property~\ref{idconsistency} of SFC approach \citep{bader2012space, sagan2012space}, we can conclude that the global index $\textbf{GI}(f)$ remains invariant even when the domain decomposition of $\Omega_h$ changes. 
This global invariant index $\textbf{GI}(f)$ plays a crucial role in the process of identifying a set of non-intersecting advancements.

\begin{property}
\label{idconsistency}
Consider two decompositions of current fronts set ${\cal F}$, denoted as ${\cal F}=\cup_{i=1}^n{\cal F}_i$ and ${\cal F}=\cup_{i=1}^m\tilde{\cal F}_i$. Assume that a front $f$ is in subset ${\cal F}_k$ and also in subset $\tilde{\cal F}_l$, respectively. Then we have
\begin{equation}
    \textbf{Id}(f, k; n) + \sum_{i=1}^{k-1}|\mathcal{F}_{i}| = \textbf{Id}(f, l; m) + \sum_{i=1}^{l-1}|\tilde{\mathcal{F}}_{i}|.
\end{equation}
\end{property}


\subsection{Distributed forest-of-overlapping-trees approach}
\label{subsection: overlap}

On parallel computers, the intersection judgements between newly generated elements and existing elements could involve the interaction of fronts belonging to different processors. To address the interaction without imposing constraints on subdomain boundaries, we introduce an overlapping domain decomposition on top of the background meshes based on the non-overlapping domain decomposition.  
As shown in Figure~\ref{fig: OverlappingTree}, we extend each background subdomain with $\delta$ layers background meshes to create a larger background subdomain. Similarly, the sub-front set ${\cal F}_i$ is also extended to  ${\cal F}_i^\delta$. Then, we define the neighboring set of a front $f$ as $Neighbor(f) = \{ \hat{f} \in \mathcal{F}_{i}^{\delta} \ |\ \textbf{dist}(\hat{f}, f) < \gamma(h) \}$ for searching potential advancing points and the intersection judgements using \textbf{Criterion A} and \textbf{Criterion B}.
Typically, $\gamma(h)$ is set to $1.5h$. To ensure that the neighborhoods of any front $f$ in ${\cal F}_i$ are included in ${\cal F}_i^\delta$, it is important to have a sufficiently large overlap $\delta$ for each processor. The value of $\delta$ is determined by the ratio of the background mesh scale (associated with the level of SFC) to the mesh scale $h$. In conclusion, the overlapping domain decomposition ensures that the interaction is not limited by the boundaries of subdomains, so that the attainment of a set of potential advancements is consistent with the sequential AFT.

\begin{figure}[H]
		\centering
		\subfigure[overlapping quadtree 1]{
		\begin{minipage}[c]{0.465\textwidth}
		  \centering
            \includegraphics[width=0.8\textwidth]{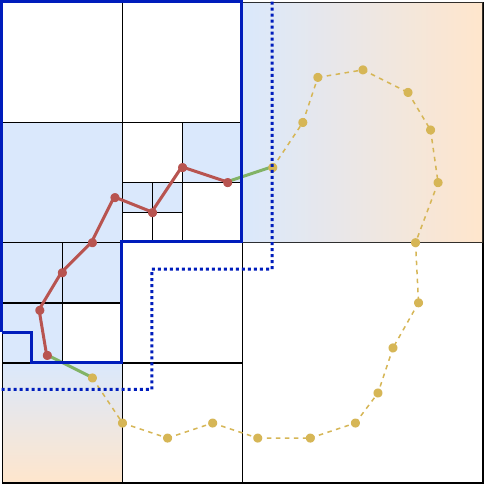}
        \end{minipage}
		} \subfigure[overlapping quadtree 4]{
		\begin{minipage}[c]{0.465\textwidth}
		  \centering
            \includegraphics[width=0.8\textwidth]{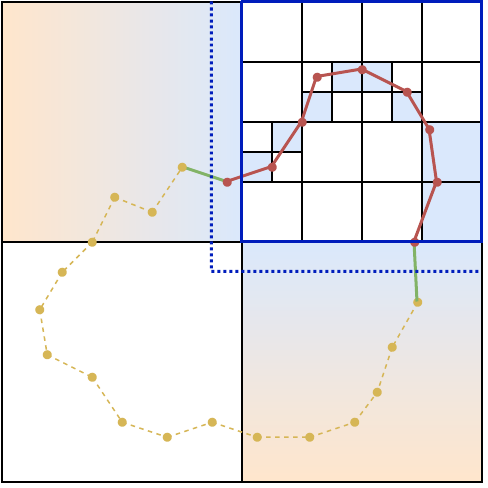}
		\end{minipage}
		} 

        \subfigure[overlapping quadtree 2]{
		\begin{minipage}[c]{0.465\textwidth}
		  \centering
            \includegraphics[width=0.8\textwidth]{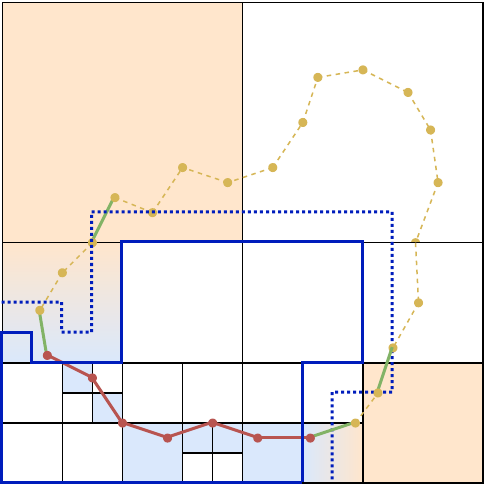}
		\end{minipage}
		} \subfigure[overlapping quadtree 3]{
		\begin{minipage}[c]{0.465\textwidth}
		  \centering
            \includegraphics[width=0.8\textwidth]{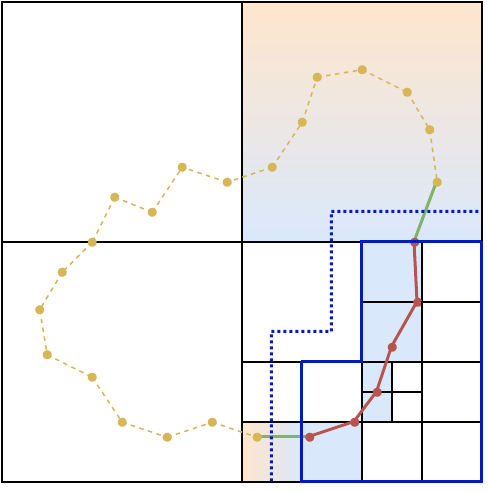}
		\end{minipage}
		} 
\caption{Forest of overlapping trees for a 2D problem with 4 processors. The red-colored, green-colored, and yellow-colored faces represent the fronts belonging to the current fronts subset, the extended neighboring fronts set, and the others, respectively. The solid blue lines and the dotted blue lines respectively represent the boundaries of the non-overlapping background subdomains and overlapping background subdomains with $\delta = 1$. The blue boxes denote the quadtree owned by the current subdomain, and the boxes transitioning from blue to yellow illustrate the expansion of the owned quadtree. }
\label{fig: OverlappingTree}
\end{figure}


In order to efficiently represent the overlapping subsets ${\cal F}_{i}^{\delta}$ with $i = 1, \dots, n$, we  introduce a new distributed forest-of-overlapping-trees approach by extending the origin distributed forest-of-trees approach \citep{isaac2015recursive}. As illustrated in Figure~\ref{fig: OverlappingTree}, in the case of a 2D domain, a leaf node in the quadtree data is divided into several leaf nodes until each leaf node stores at most one front. This process establishes an injection mapping from the front set ${\cal F}$ to the quadtree data. The injection mapping described earlier facilitates efficient search and insertion at a computational cost of $O(\log|\mathcal{F}_i|)$. Meanwhile, the construction of overlapping fronts ${\cal F}_{i}^{\delta} \backslash {\cal F}_{i}$ involves searching in the data of neighboring processors and inserting into the data of the current processor, which correspond to the packing and unpacking processes, respectively, in the communication procedure. As a consequence, the overlapping octree (quadtree) data structure significantly enhances the parallel efficiency of the newly proposed parallel algorithm.

\subsection{Consistent parallel maximal independent set (MIS) approach}
\label{subsection: CPMIS}


In each iteration step of AFT, after identifying the potential advancements satisfying \textbf{Criterion A} and \textbf{Criterion B}, searching for mutually non-intersecting advancements within the potential advancements becomes a new challenge, which is essentially equivalent to finding the largest possible subset. The sequential AFT typically achieves this through greedy search, which is computationally expensive and results in significant communication overhead in distributed scenarios. Additionally, the results are influenced by the order of the search. In this work, to improve the efficiency and parallel performance of the AFT algorithm, we propose the consistent parallel MIS approach to replace the greedy search, which can produce more consistent and reliable results by removing the dependence on the search order. 

\ 

\begin{figure}[H]
\centering
\includegraphics[width=0.9\textwidth]{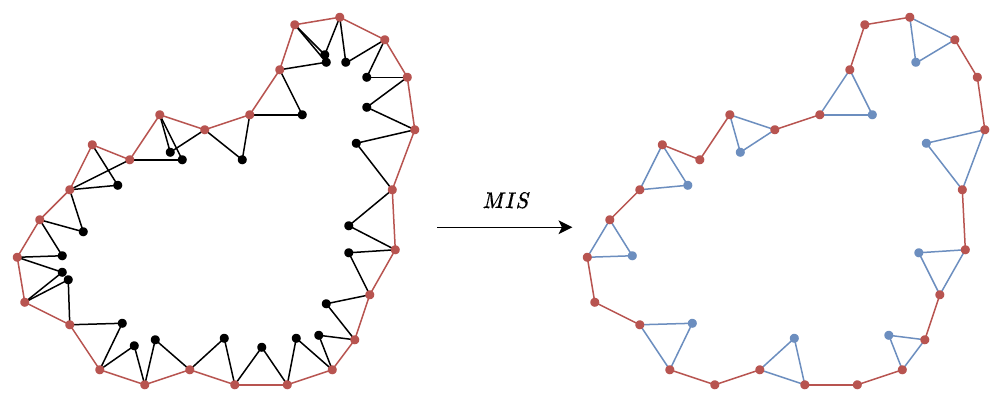}
\caption{The newly generated elements set ${\cal V}$ (left) and a maximal independent set ${\cal V}_s$ (right).
}
\label{fig: MISGlobal}
\end{figure}

\ 

To begin, let us define a graph $\hat{\cal{F}}:=({\cal{V}}, {\cal{E}})$ for the current fronts set ${\cal F}$, where ${\cal V}$ and $\cal E$ are the sets of vertices and edges, respectively. For a front $f_i \in {\cal{F}}$, if there exists a corresponding optimal potential advancing point $p_i\in P$, we denote vertex $(f_i,p_i)$ as a newly generated element in $\cal V$. Here, $P$ is the set of all optimal potential advancing points. Two vertices $v_i = (f_i, p_i)$ and $v_j = (f_j, p_j)$ in $\hat{\cal{F}}$ are connected by an edge $e_{i, j} \in {\cal{E}}$ if and only if $\textbf{dist}(p_i, p_j) > h_i+h_j$, where $h_i$ and $h_j$ are the local scales near front $f_i$ and $f_j$, respectively. According to refs. \citep{chernikov2006parallel, zhou2022saft}, the fact that two vertices $v_i$ and $v_j$ are not connected by an edge $e_{i, j}$ is a sufficient but unnecessary condition to efficiently ensure that the two newly generated element $(f_i, p_i)$ and $(f_j, p_j)$ do not intersect. Moreover, this condition prevent a single advancement from causing two fronts to be too close, thereby facilitating the achievement of high-quality mesh.
If there exists a sub-graph $\hat{\cal{F}}_s:=({\cal{V}}_s,\emptyset)\subset\hat{\cal{F}}$ such that for every pair of vertices $v_i, v_j \in {\cal{V}}_s$, there is no edge connecting them, then the set ${\cal V}_s$ is defined as an independent set. An independent set is considered as a maximal independent set if every vertex $v_j\in{\cal{V}}\setminus{\cal{V}}_s$ is connected to at least one vertex $v_i\in{\cal{V}}_s$. Therefore, finding the largest possible mutually non-intersecting advancements is equivalent to finding a maximal independent set of graph $\hat{\cal{F}}$, which is a well-known NP-hard problem \citep{lawler1980generating}. In the implementation process, it is common to find an independent set that is as large as possible to improve efficiency. In fact, any non-empty independent subset can be used in the AFT approach. Meanwhile, for any non-empty graph $\hat{\cal{F}}$, there exists at least one non-empty independent subset ${\cal{V}}_s:=\{v\}$ for any $v \in {\cal{V}}$. An illustrative example of the independent set for a given graph $\hat{\cal F}$ is shown in Figure~\ref{fig: MISGlobal}.  

The MIS algorithm \citep{luby1985simple, gfeller2007randomized, panconesi1997randomized} is a well-known method in graph theory, which can find a maximal independent set or as large as possible independent set for a given graph. Inspired by ref. \citep{chan1998agglomeration}, we introduce the heuristic MIS algorithm \citep{luby1985simple} into solving parallel mesh generation problems, which can help resolve the issue of mesh intersections by finding an independent set that is as large as possible. The origin heuristic MIS algorithm \citep{luby1985simple} has two variants, including a parallel Monte Carlo version with random selection and a deterministic version without randomness. In this work, we introduce a new parallel MIS algorithm that is based on the aforementioned deterministic heuristic MIS algorithm and the domain decomposition method proposed in Subsection~\ref{subsection: DD-SFC}.
In the deterministic heuristic MIS algorithm, the vertex set $\cal{V}$ is divided into three disjoint subsets: $\cal{A}$, ${\cal{U}}$, and $\cal{P}$, where ${\cal{V}} = {\cal{A}} \cup {\cal{U}} \cup {\cal{P}}$. These subsets represent the legal advancements, the advancements that need to be discarded (discarded set), and the pending set, respectively. The algorithm begins with ${\cal A} = \emptyset,\ {\cal U} = \emptyset,\ {\cal P} = \cal{V}$ and iteratively moves the vertices in ${\cal P}$ into ${\cal A} $ or  ${\cal U}$ until ${\cal P} = \emptyset$. For more detailed information about the deterministic heuristic MIS algorithm, we recommend referring to the reference \citep{luby1985simple}.

\ 

\begin{figure}[H]
\centering
\includegraphics[width=0.9\textwidth]{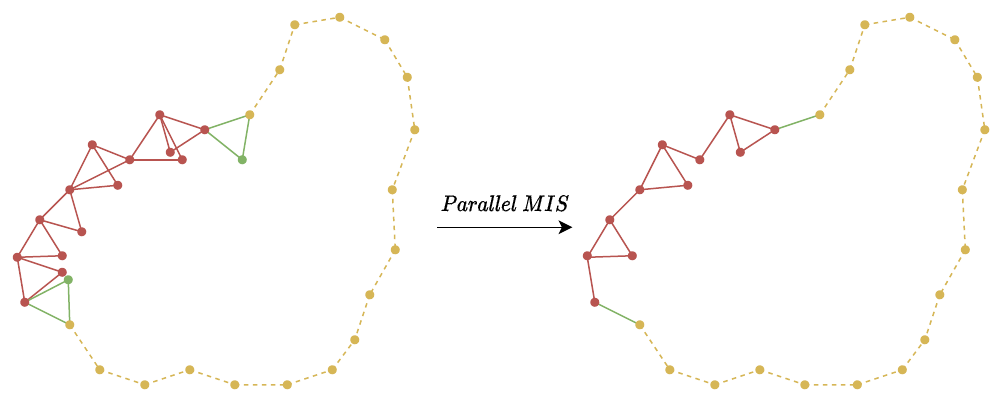}
\caption{In the parallel algorithm with 4 subdomains, we take the fronts in the first subdomain as an example. The newly generated elements set ${\cal V}_1^\delta$ (left) and an independent set ${\cal V}_{1,s}$ (right).
}
\label{fig: MISLocal}
\end{figure}

\ 

Next, let us introduce a new parallel MIS algorithm based on the domain decomposition method mentioned in the above subsection.
According to the non-overlapping domain decomposition, the front set ${\cal{F}}=\mathcal{F}_1 \cup \mathcal{F}_2 \cup \dots \cup \mathcal{F}_n$. Similarly, the vertex set $\cal{V}$ is divided into $n$ disjoint subsets, denoted as ${\cal{V}}_1, \dots, {\cal{V}}_n$. Each subset ${\cal{V}}_i$ is assigned to and processed by one processor. In a graph $\hat{\cal{F}}$, the neighbor of a vertex $v \in {\cal{V}}_i$ is defined as 
\begin{equation}
    {\cal{N}}(v) = \{\hat v \in {\cal{V}} | v \hbox{ and } \hat v \hbox{ are connected by an edge } e \in {\cal{E}}\}. 
\end{equation}
The extension of $\mathcal{V}_i$ is then defined as $\mathcal{V}_i^\delta=\mathcal{V}_i\cup_{v\in\mathcal{V}_i} {\cal{N}}(v) $ based on the overlapping domain decomposition. Similar to the deterministic heuristic MIS algorithm, the vertex set $\mathcal{V}_i$ and its extension $\mathcal{V}_i^\delta$ are divided as 
\begin{equation}
    \begin{aligned}
        & {\cal{V}}_i = {\cal{A}}_i \cup {\cal{U}}_i \cup {\cal{P}}_i, \ \ \ \ \ \ \ \ \ \ \ i = 1, \dots, n, \\ 
        & {\cal{V}}_i^{\delta} = {\cal{A}}_i^{\delta} \cup {\cal{U}}_i^{\delta} \cup {\cal{P}}_i^{\delta}, \ \ \ \ \ \ \ \ \ i = 1, \dots, n .
    \end{aligned}
\end{equation}

With these notations, the newly proposed parallel MIS algorithm can be summarized in Algorithm~\ref{PMIS}. This algorithm continues until all pending sets ${\cal P}_i$ are empty, at which point it obtains a maximal independent set for the given graph $\hat{\cal F}$. To meet efficient requirements, the Algorithm~\ref{PMIS} terminates after a few iterations to achieve an independent set with a satisfactory amount of advancements. Therefore, the computational complexity of Algorithm~\ref{PMIS} is reduced to the order of  $O(M|{\cal{V}}_{i}|)$ for the $i_{th}$ processor, where $M$ is the predefined number of maximum iteration.

\begin{algorithm} [H]
\caption{The consistent parallel MIS algorithm.}
\label{PMIS}
\begin{algorithmic}
\STATE{\textbf{Initialization}: For each subset ${\cal{V}}_i$, let ${\cal{A}}_i={\cal{A}}_i^\delta = \emptyset,\ {\cal{U}}_i={\cal{U}}_i^\delta = \emptyset,\ {\cal{P}}_i = {\cal{V}}_i,\ {\cal{P}}_i^\delta = {\cal{V}}_i^\delta,\ i = 1, \dots, n$. 
For a vertex $v = (f, p)$, set $ID(v)=\textbf{GI}(f)$.}
\WHILE{$its < M$}
\STATE{\textbf{1.} For a vertex $v \in {\cal P}_i$, if $ID(v) \leq ID(\hat v) $ for all $\hat v \in {\cal N}(v) \cap {\cal P}_i^{\delta}$, then move it into a temporary set ${\cal{T}}_i$. }
\STATE{\textbf{2.} For each entry $v_t \in {\cal{T}}_i$, move $v_t$ into ${\cal{A}}_i$ and move vertices belonging in ${\cal{N}}(v_t) \cap ({{\cal{A}}_{i}^{\delta}} \cup {{\cal{P}}_{i}^{\delta}})$ into ${{\cal{U}}_{i}^{\delta}}$. The discarded set ${\cal U}_i$ should be updated correspondingly. }
\STATE{\textbf{3.} For each vertex in ${\cal U}_i^{\delta} \backslash {\cal U}_i$, adjust it into ${\cal U}_j$ on the corresponding processor via communication. The other subsets should be updated by communication in a similar manner.}
\STATE{\textbf{4.} If $|{\cal{P}}_1 \cup \dots \cup {\cal{P}}_n| > 0$ we return to step \textbf{1} and $its := its + 1$; Otherwise, break.}
\ENDWHILE
\RETURN {${\cal{A}}_i,\ {\cal{U}}_i,\ {\cal{P}}_i$, $i = 1, \dots, n$, and the independent set ${\cal V}_s={\cal{A}}_1\cup\cdots\cup{\cal{A}}_n$. 
}
\end{algorithmic}
\end{algorithm}

The following lemma will show that Algorithm~\ref{PMIS} can lead to a non-empty independent set ${\cal V}_s$.
\begin{lemma}\label{lemma:PMIS}
For the vertex set ${\cal{V}}={\cal{V}}_{1} \cup \dots \cup {\cal{V}}_n$ with $|{\cal{V}}| > 0$,  Algorithm~\ref{PMIS} can obtain an independent set ${\cal V}_s$ with $|{\cal V}_s|>0$.
\end{lemma}
\begin{proof}
    Let $v$ be the vertex with the minimum $ID$. In the first iteration of the algorithm, $v$ is added to the set ${\cal{A}}_1 \cup \dots \cup {\cal{A}}_n$ and its neighbors are added to the set ${\cal{U}}_1 \cup \dots \cup {\cal{U}}_n$. In the subsequent iterations, the vertex $v$ always stays in the set ${\cal{A}}_1 \cup \dots \cup {\cal{A}}_n$. This implies that $|{\cal V}_s|\geq 1$.
\end{proof}

Based on the overlapping domain decomposition and global indicator $\textbf{GI}$, we can prove that Algorithm~\ref{PMIS} is parallel consistent, as shown in Theorem~\ref{consistency}.

\begin{theorem}[Parallel consistency]
    \label{consistency}
    The independent set ${\cal V}_s$ obtained by Algorithm~\ref{PMIS} is independent with respect to the decomposition ${\cal V}={\cal{V}}_1\cup\cdots\cup{\cal{V}}_n$ and the number of processors.
\end{theorem}
\begin{proof}
    We prove this theorem by shown that the independent sets ${\cal V}_s$ obtained by $n$ processors and $1$ processor are the same. Initially, ${\cal{P}}_1 \cup \dots \cup {\cal{P}}_n = {\cal{P}}= {\cal{V}}$. At the first iteration, for any vertex $ v \in {\cal{P}}_i \subset {\cal{P}}$, we have 
         \begin{equation}
         \label{theorem:eq1}
        {\cal N}(v) \cap {\cal P}_i^{\delta} = {\cal N}(v) \cap {\cal P}.
     \end{equation}
     According to Property~\ref{idconsistency}, the $ID$ of each vertex remains the same regardless of whether the algorithm is executed using $n$ processors or just $1$ processor. In other words, for any vertex $v$, we have 
     \begin{equation}
     ID(v, n) = ID(v, 1). 
     \end{equation}
     Then, the temporary set in step 1 satisfies
     \begin{equation}
        {\cal{T}}_{1} \cup \dots \cup {\cal{T}}_{n} = {\cal{T}},
     \end{equation}
     where the temporary set ${\cal{T}}$ is obtained by using $1$ processor.
     In step 2, before updating discarded set ${\cal{U}}_i$ and ${\cal{U}}$, we have ${\cal{A}}={\cal{A}}_{1} \cup \dots \cup {\cal{A}}_{n}$. Here ${\cal{A}}$ and ${\cal{U}}$ are obtained by using $1$ processor.
     For a vertex $v_t \in {\cal{T}}_i \subset {\cal{T}}$, similar to equation \eqref{theorem:eq1}, we obtain that
     \begin{equation}
     {\cal{N}}(v_t) \cap ({{\cal{A}}_{i}^{\delta}} \cup {{\cal{P}}_{i}^{\delta}}) = {\cal{N}}(v_t) \cap ({\cal{A}} \cup {\cal{P}}),
     \end{equation}
     which implies ${\cal{U}}={\cal{U}}^\delta_{1} \cup \dots \cup {\cal{U}}^\delta_{n}$. After completing communication in step 3, we get that
     \begin{equation}
     \label{theorem:eq2}
        {\cal{A}}_{1} \cup \dots \cup {\cal{A}}_{n} = {\cal{A}},
     \end{equation}
     and
     \begin{equation}
     \label{theorem:eq3}
        {\cal{U}}_{1} \cup \dots \cup {\cal{U}}_{n} = {\cal{U}}.
     \end{equation}
     Assume that equation \eqref{theorem:eq2} and equation \eqref{theorem:eq3} remain constant after $its$ iterations. Now, we will show that these equations hold at the ${(its+1)}_{th}$ iteration. According to equation \eqref{theorem:eq2} and equation \eqref{theorem:eq3}, we have ${\cal{P}}_1 \cup \dots \cup {\cal{P}}_n = {\cal{P}}$. By repeating the proof of the first iteration, one can prove equation \eqref{theorem:eq2} and equation \eqref{theorem:eq3} remain constant after $its+1$ iterations. By induction, we have that the independent sets ${\cal V}_s$ obtained by $n$ processors and $1$ processor are the same, which completes the proof of this theorem.
\end{proof}

\subsection{Algorithm workflow and parallel mesh quality enhancement}
\label{subsection: workflow}
In this subsection, we summarize the workflow of the newly proposed CPAFT algorithm in Algorithm~\ref{workflow} and analyze its convergence. A parallel mesh quality enhancement approach is also introduced to improve the quality of the newly generated mesh.

\begin{algorithm} [H]
\caption{The consistent parallel advancing front technique (CPAFT).}
\label{workflow}
\begin{algorithmic}
\STATE{\textbf{Initialization}: A maximum scale $h_M$ and minimum scale $h_m$ and a reasonable initial front set $\mathcal{F}$ which is consistent with scale $h_M$ and $h_m$.}
\WHILE{$true$}
\STATE {\textbf{1.} According to criteria \textbf{C1-3}, construct subdomains $\Omega_h^i$, $\Omega_h^{i,\delta}$ and subsets ${\mathcal{F}}_{i}$, ${\mathcal{F}}_{i}^\delta$ by using the domain decomposition method proposed in Subsection~\ref{subsection: DD-SFC} for the current front set $\mathcal{F}$.}
\FOR{$\forall f$ in $\mathcal{F}_{i}$ }
\STATE {\textbf{2.} Find the optimal point $p$ (if exists) from the potential advancing points on the perpendicular line of front $f$ and the fronts in $Neighbor(f)$ according to \textbf{criterion A} and \textbf{criterion B}. Generate a new element $(f,p)$ and add it into vertex set $\cal V$.}
\ENDFOR
\STATE {\textbf{3.} Build the graph $\hat{\cal F}:=({\cal{V}}, {\cal{E}})$ and apply Algorithm~\ref{PMIS} to obtain ${\cal{A}}_i,\ {\cal{U}}_i,\ {\cal{P}}_i$. }
\STATE {\textbf{4.} Add elements in ${\cal{A}}_i$ into $\mathbb{E}_{h}^{i}$ and update front sets ${\cal{F}}$ and ${\cal{F}}_{1}, \dots, {\cal{F}}_{n}$ correspondingly.}
\STATE {\textbf{5.} If $|{\cal{F}}_{1}\cup \dots \cup {\cal{F}}_{n}|>0$ and $ |{\cal V}_s| > 0$, return step \textbf{1}; otherwise, break.}
\ENDWHILE
\RETURN {All generated elements which are represent in $n$ subsets $\mathbb{E}_{h}^{i}, \  i = 1, \dots, n$.}
\end{algorithmic}
\end{algorithm}

According to Theorem~\ref{consistency}, the CPAFT algorithm is parallel consistent. The following Theorem~\ref{thm:finiteConvergence} shows that CPAFT can terminate after a finite number of iterations.

\begin{theorem} \label{thm:finiteConvergence}
Algorithm~\ref{workflow} terminates after a finite number of iterations, which is less than $\lceil V_0 / (\epsilon \eta h_m^{d}) \rceil$.
\end{theorem}
\begin{proof}
    After the $k_{th}$ iteration, let $V_k$ be the volume of the polyhedron enclosed by current $\mathcal{F}$.
    Then $\{ V_k \}$ is a monotone sequence, i.e.,
    \begin{equation}
        |\Omega_h|=V_0 \geq \cdots \geq V_{k-1} \geq V_{k} \cdots.
    \end{equation}
    We denote $\lceil V_0 / (\epsilon \eta h_m^{d}) \rceil$ as $M$, in which $(\epsilon,\eta,h_m)$ are constants mentioned in Section~\ref{section: preliminaries}. If Algorithm~\ref{workflow} does not terminate after $M$ iterations, then we have $|{\cal{F}}_{1}\cup \dots \cup {\cal{F}}_{n}|>0$ and $ |{\cal V}_s| > 0$ at the $k_{th}$ iteration with $k = 0, \dots, M$. Based on Lemma~\ref{lemma:PMIS} and the criteria mentioned in Section~\ref{section: preliminaries}, at the $k_{th}$ iteration, there exists at least one element in $ {\cal V}_s$, whose volume is not less than $\epsilon \eta h_m^{d}$.  Then the $V_k$ is a strictly monotone decreasing sequence with 
    \begin{equation}
        V_{k} - V_{k+1} \geq \epsilon \eta h_m^{d}>0,\ \ k = 0, \dots, M. 
        \label{bound}
    \end{equation}
    By summing up equation \eqref{bound} with $k = 0, \dots, M$, we get that 
    \begin{equation}
     V_{0} - V_{M+1} \geq (M+1) (\epsilon \eta h_m^{d}) > V_{0}.
     \label{contradiction}
    \end{equation}
    Then, equation \eqref{contradiction} indicates that $V_{M+1} < 0$, which is contradictive with the non-negativity of volume. Therefore, Algorithm~\ref{workflow} should stop after a finite number of iterations, which is less than $\lceil V_0 / (\epsilon \eta h_m^{d}) \rceil$.
\end{proof}

The mesh generation process is completed by Algorithm~\ref{workflow} if it terminates with $|{\cal{F}}_{1}\cup \dots \cup {\cal{F}}_{n}|=0$ and $|{\cal V}_s| = 0$. However, in some cases,   Algorithm~\ref{workflow} may terminate with $|{\cal{F}}_{1}\cup \dots \cup {\cal{F}}_{n}|>0$ and $|{\cal V}_s| = 0$. In this situation, the mesh generation process is not completed because the polyhedron enclosed by remaining front set ${\cal F}$ is not empty. This usually occurs for 3D complex models. In particular, the polyhedron enclosed by remaining front set ${\cal F}$ may be too small, making it impossible to find an optimal advancing point $p$ for any $f$ that satisfies \textbf{criterion A} and \textbf{criterion B}. As a result, the vertex set $\cal V$ is empty, and consequently, the independent set ${\cal V}_s$ is also empty. To complete the mesh generation process, the remaining fronts in ${\cal F}$ can be handled by merging close points. We start by selecting the remaining front $f_m$ with minimum $ID$ and merge all of the vertices in $Neighbor(f_m)$. Then, we return to step \textbf{1} of Algorithm~\ref{workflow}. This procedure is repeated alternately until $|{\cal{F}}_{1}\cup \dots \cup {\cal{F}}_{n}|=0$.

After generating the unstructured meshes by CPAFT, it is necessary to further enhance the mesh quality. In this work, we employ the surface preserving Laplacian smoothing approach introduced in Section~\ref{section: preliminaries} with the domain decomposition method proposed in Subsection~\ref{subsection: DD-SFC} to conduct parallel mesh quality enhancement. For each elements subset $\mathbb{E}_{h}^{i},$ with $ \  i = 1, \dots, n$, we extend it to a large subsets $\mathbb{E}_{h}^{i,\delta}$ based on the overlapping domain decomposition $\Omega_{h}^{i, \delta}$. According to equation \eqref{equ: smooth}, the position of a point $p_i$ is updated by using the positions of all its neighboring nodes. Therefore, with this extension, we can run the surface preserving Laplacian smoothing approach in parallel. An example is given in Figure~\ref{fig: smoothing} to show the parallel Laplacian smoothing for a point on the boundary of two subdomains.

\begin{figure}[H]
\centering
\includegraphics[width=1.0\linewidth]{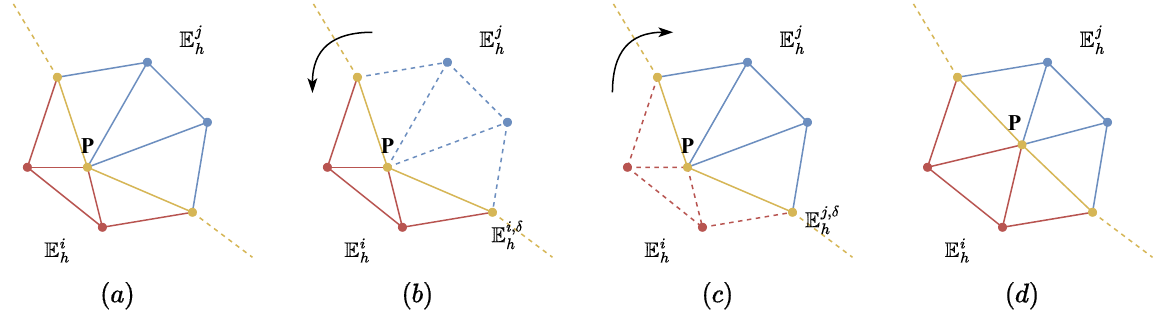}
\caption{
The parallel Laplacian smoothing for a point on the boundary of two subdomains: (a) the initial mesh, and (d) the final smoothed mesh. The red and the blue parts represent the meshes of two different subdomains, respectively. The yellow lines represent the interface between the two subdomains. In (b) and (c), the elements with dashed lines are added to the extension subset via communication.
}
\label{fig: smoothing}
\end{figure}

{\color{red}
Laplacian smoothing usually produces meshes with acceptable quality for 2D models and 3D models with simple structures. However, its effectiveness diminishes for complex 3D models, where mesh quality improvements are often insufficient. To achieve better results, it is crucial to optimize the adjacency relationships through operations such as edge splits, edge collapses, and edge flips. In this work, we implement a parallel local remeshing algorithm, based on the computational geometry algorithms library (CGAL) \citep{fabri2009cgal}. Our implementation allows each processor to independently identify and remove elements with a mesh quality indicator $\alpha < 0.3$ within its respective set $\mathbb{E}_{h}^{i}$, along with a small set of surrounding elements. The ``remeshing" function in CGAL is then used to re-generate the meshes, which are subsequently reintegrated into the subset $\mathbb{E}_{h}^{i}$. Repeating this process 1-3 times significantly enhances mesh quality.  
}

After generating the unstructured mesh using Algorithm~\ref{workflow} and applying parallel mesh quality enhancement approach, we obtain $n$ disjoint subsets $\mathbb{E}_{h}^{i}$, where $i = 1, \dots, n$. These subsets solve the problem stated in Prob.~\ref{prob} in parallel. Simultaneously, all elements are partitioned into $n$ disjoint subsets, which can be considered as the result of mesh partitioning. As a result, the newly proposed Algorithm~\ref{workflow} can be executed concurrently with a parallel numerical PDE solver, making it more convenient for computer aided engineering or computational fluid dynamics simulations.

\

\section{Experiment results}
\label{sec: validation}
To examine the performance of the newly proposed CPAFT algorithm, we carry out numerical experiments on a series of 2D and 3D problems. All simulations are carried out on a supercomputer with multi-nodes, each equipped with two AMD EPYC 7452 32-Core CPUs and 256GB of local memory. The nodes are interconnected via an Infiniband high-performance network. In all tests, the surface meshes of all geometry models are segmented with Gmsh \citep{geuzaine2009gmsh}. We mainly focus on: (1) the parallel consistency, (2) the mesh quality, (3) the parallel scalability, {\color{red} and (4) comparisons with three open-source software tools: Gmsh \citep{geuzaine2009gmsh}, NETGEN \citep{schoberl1997netgen}, and TetGen \citep{hang2015tetgen}}.

\ 

\subsection{Triangular mesh generation in 2D}
Firstly, we demonstrate the iteration of the proposed algorithm using a 2D gear with an inner radius of 0.8, an outer radius of 1.8, 12 teeth with radius $r=2.0$, and angle $\theta = \pi/20$. We show the geometry of the gear and display the mesh generation procedure of the CPAFT algorithm with 4 processors in Figure~\ref{fig: gear-op}. From the figure we can see that CPAFT can efficiently obtain a high quality triangular mesh after 27 iterations for this problem. Without mesh quality enhancement, all mesh quality indicators $\alpha$ for this problem are above $0.8$.

\ 

\begin{figure}[H]
		\centering
		\subfigure[Geometry]{
		\begin{minipage}[c]{0.312\textwidth}
		  \centering
            \includegraphics[width=0.875\linewidth]{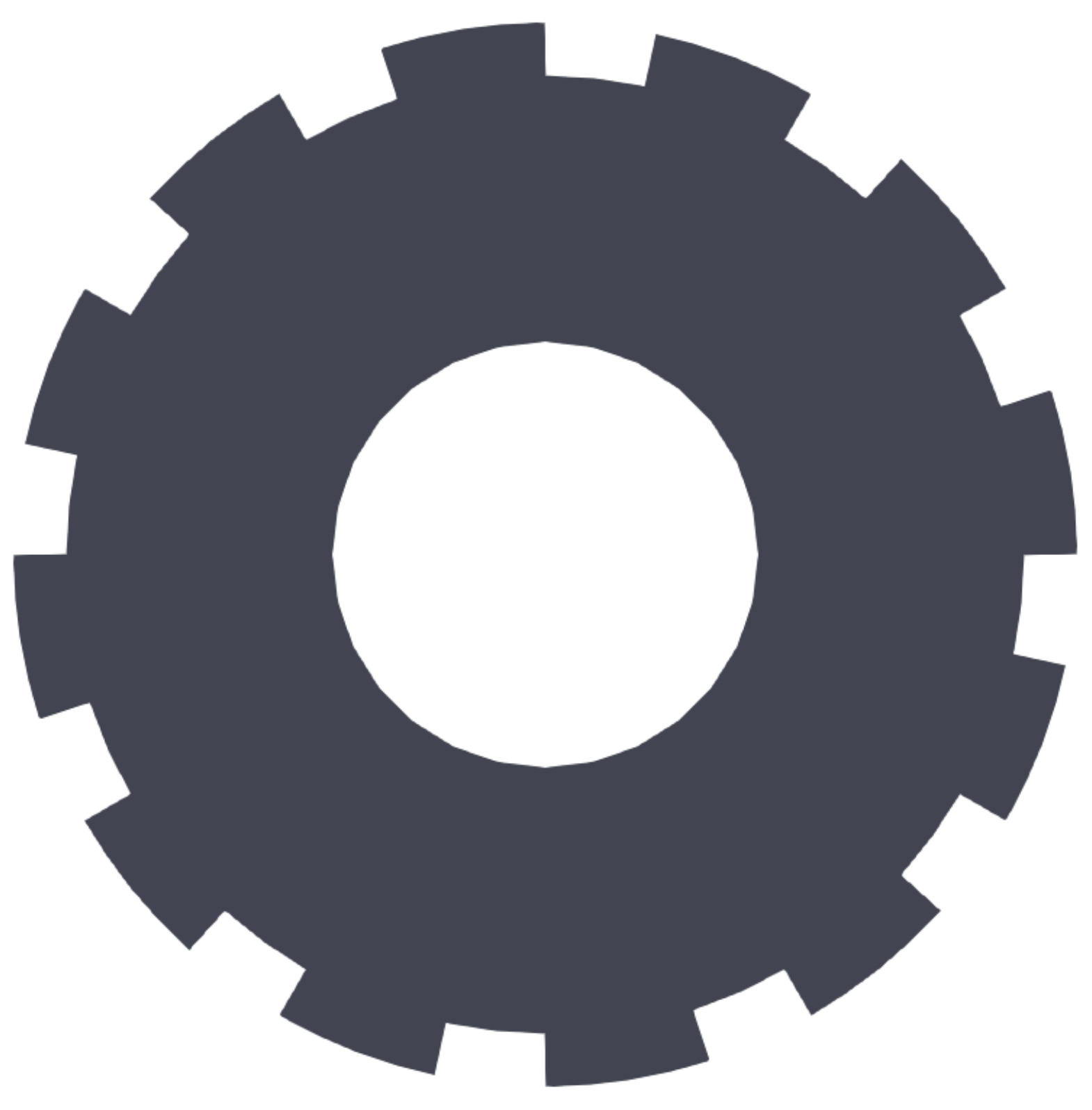}
		\end{minipage}
		} \subfigure[Step 5]{
		\begin{minipage}[c]{0.312\textwidth}
		  \centering
            \includegraphics[width=0.875\linewidth]{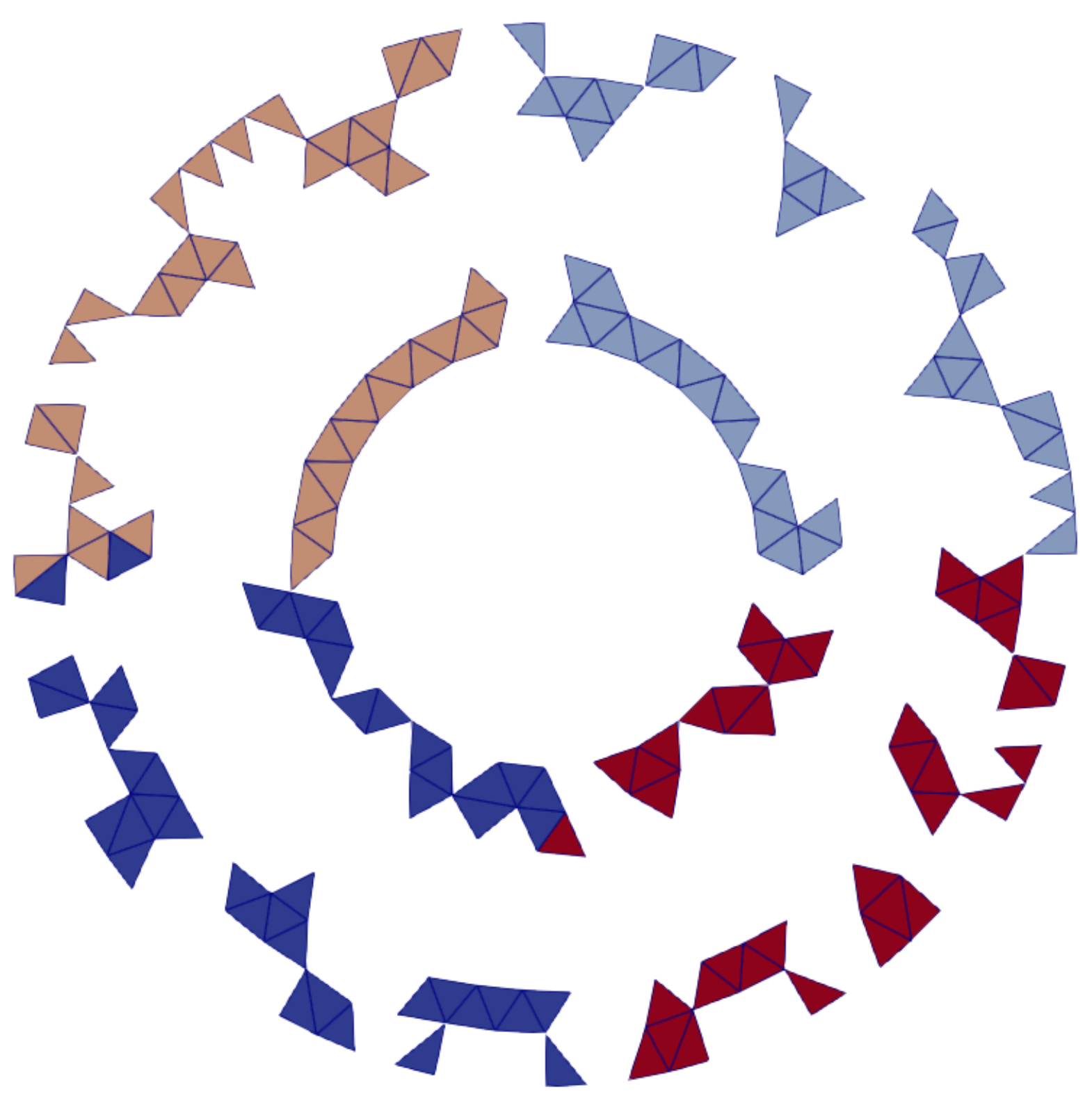}
		\end{minipage}
		} \subfigure[Step 10]{
		\begin{minipage}[c]{0.312\textwidth}
		  \centering
            \includegraphics[width=0.875\linewidth]{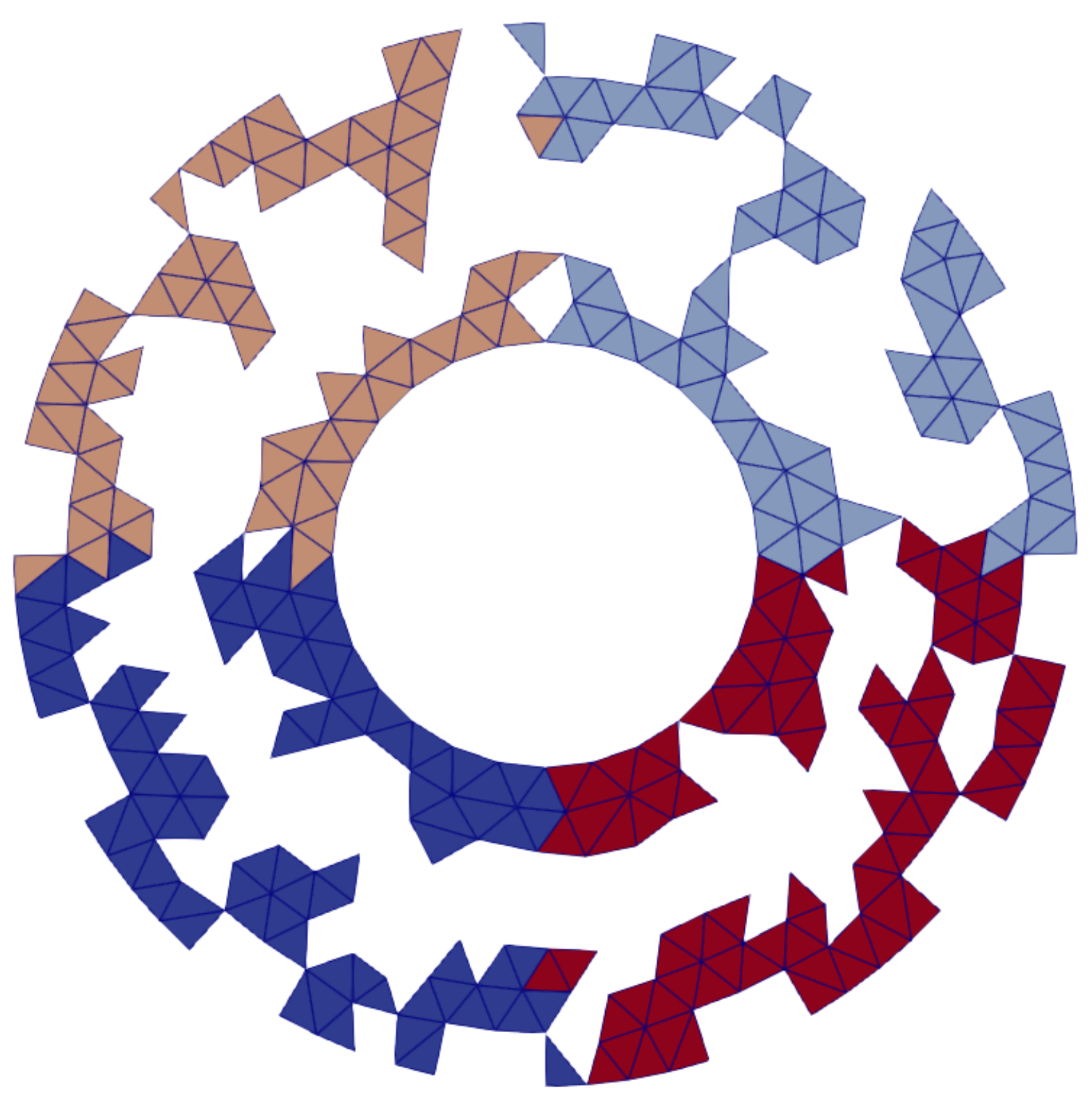}
		\end{minipage}
		} 
  
        \subfigure[Step 15]{
		\begin{minipage}[c]{0.312\textwidth}
		  \centering
            \includegraphics[width=0.875\linewidth]{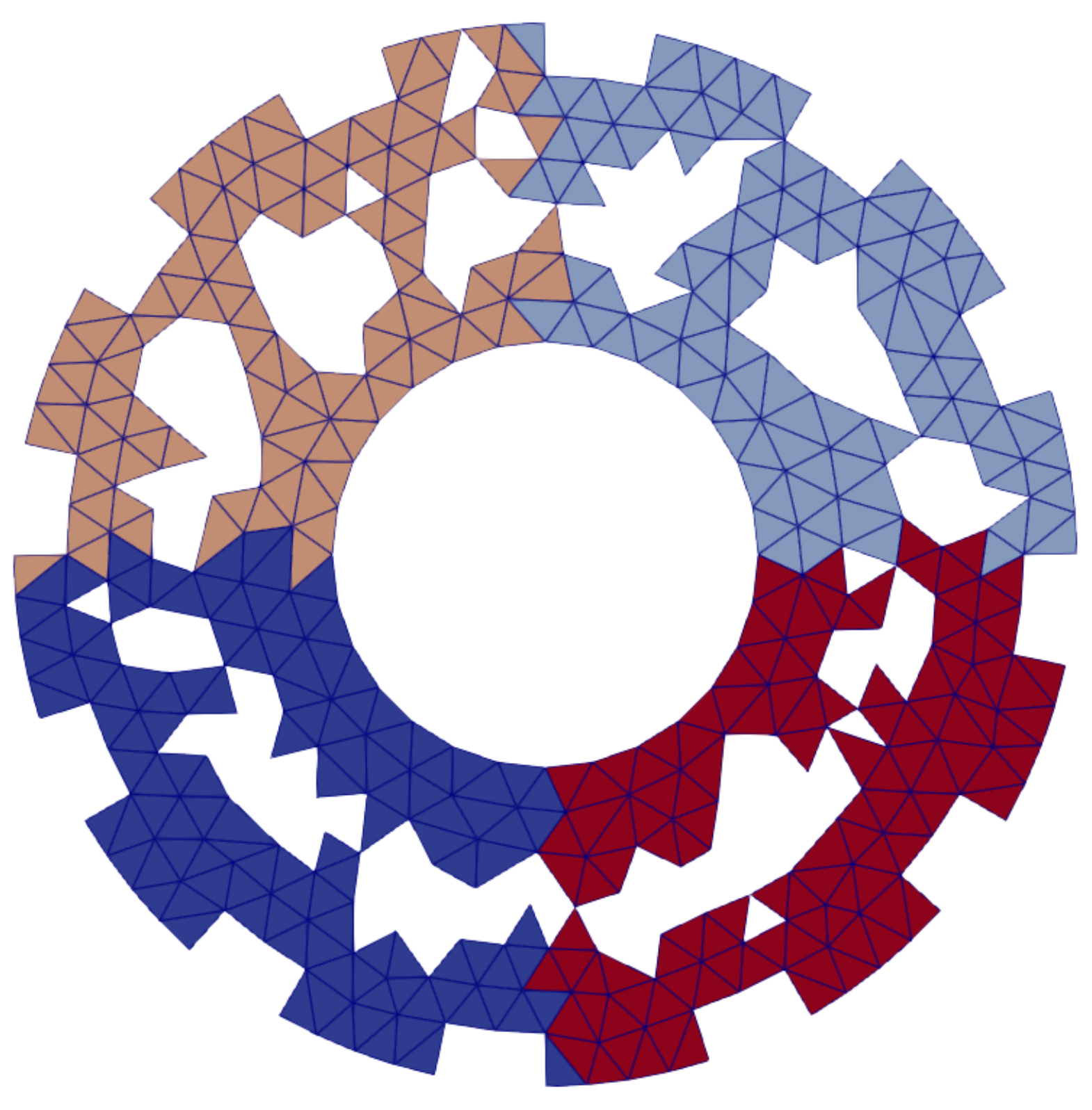}
		\end{minipage}
		} \subfigure[Step 20]{
		\begin{minipage}[c]{0.312\textwidth}
		  \centering
            \includegraphics[width=0.875\linewidth]{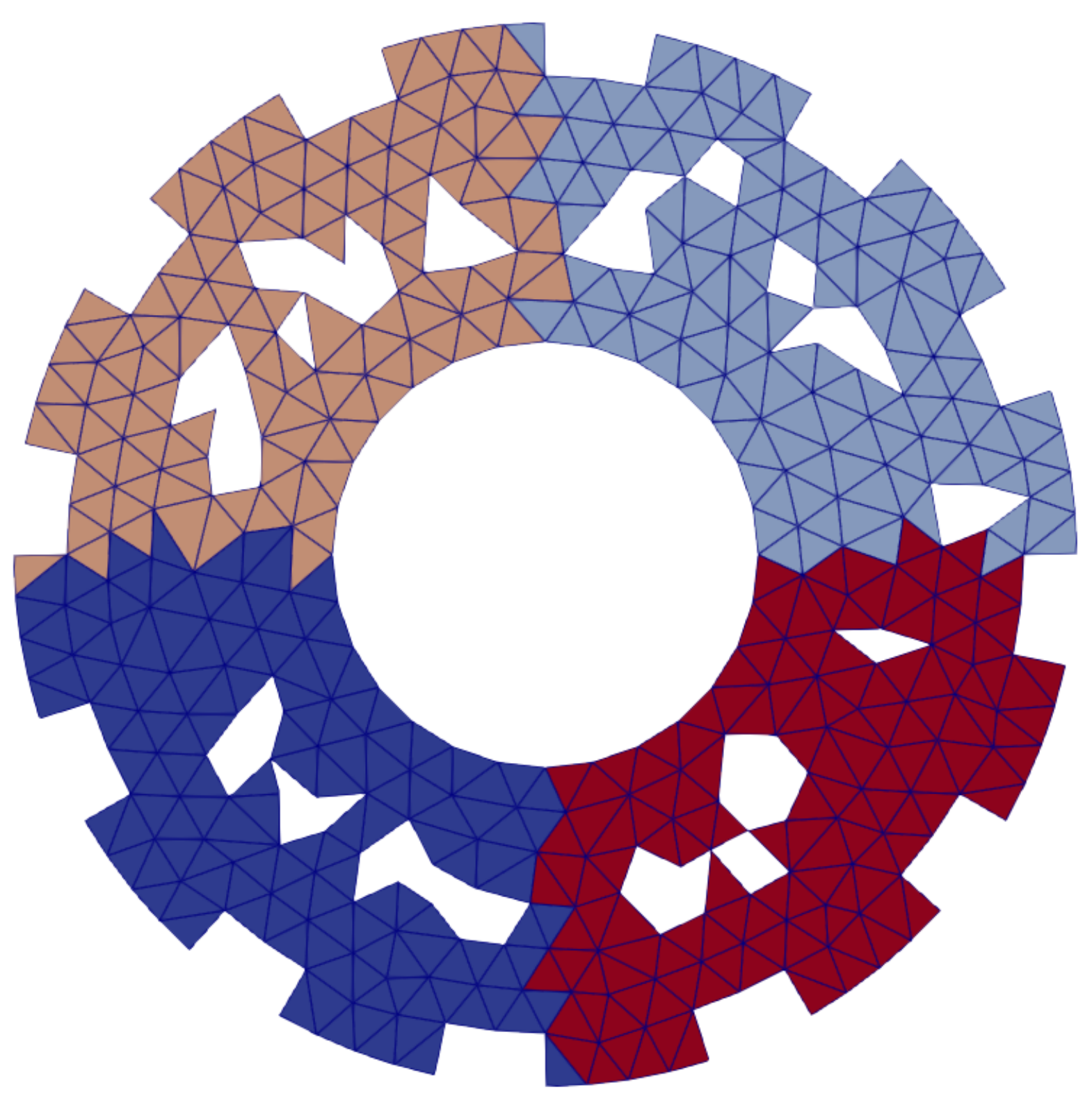}
		\end{minipage}
		} \subfigure[Step 27]{
		\begin{minipage}[c]{0.312\textwidth}
		  \centering
            \includegraphics[width=0.875\linewidth]{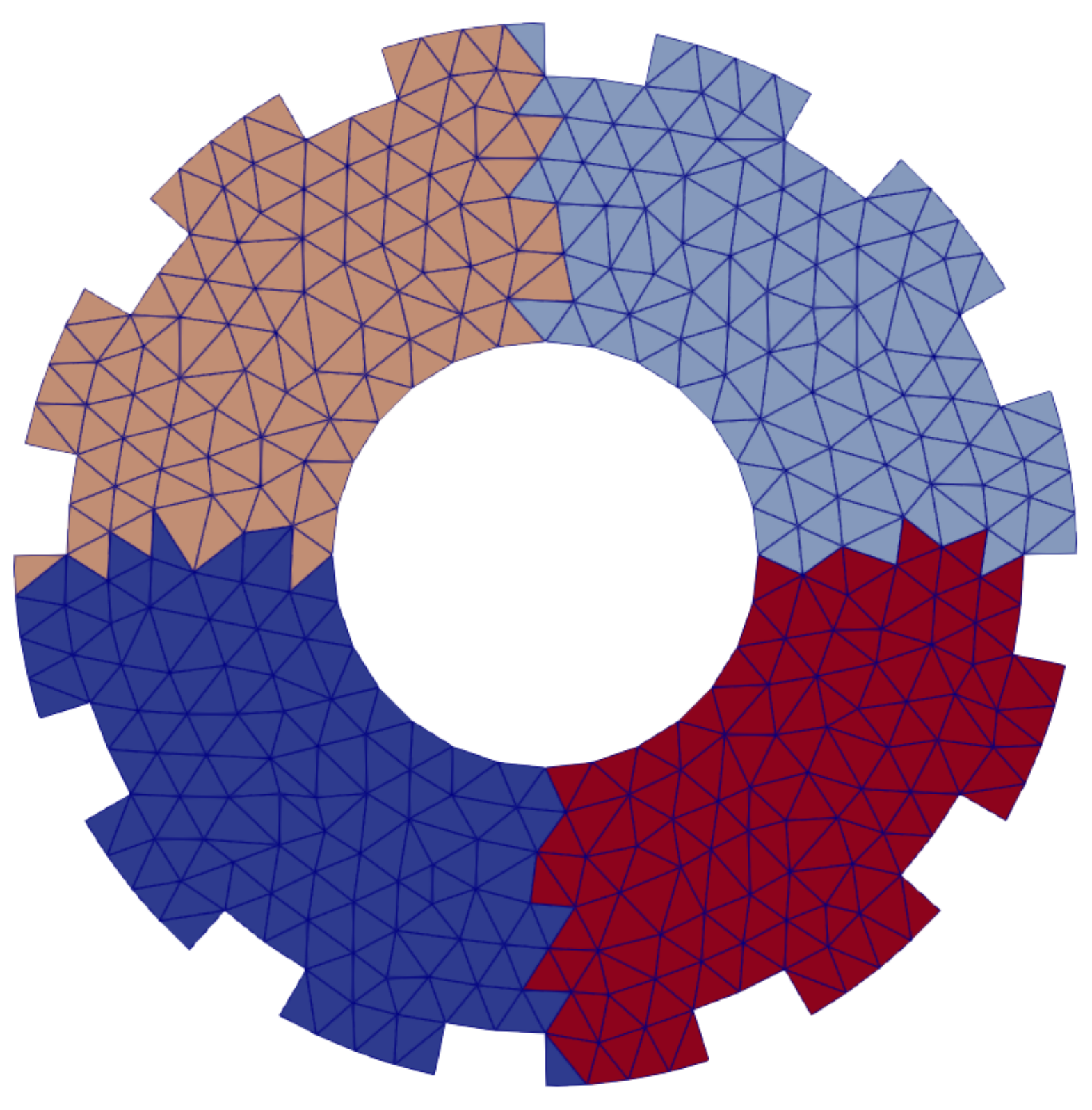}
		\end{minipage}
		}
		\caption{A 2D gear model. The CPAFT algorithm convergences in 27 iterations. Element counts per processor are 158, 153, 153, and 152, respectively.}
		\label{fig: gear-op}
\end{figure}

\ 

{\color{red}
The second test case is the NACA0012 airfoil \citep{jacobs1937airfoil}, a benchmark model commonly used to validate mesh generation algorithms. The primary challenge with this model lies in the potential non-uniformity of mesh scale from the inner to the outer boundary. We employ the CPAFT algorithm combined with Laplacian smoothing to generate meshes for this model, using both a uniform local scale $h$ and varying local scales which decrease from outer to inner boundary. 
All tests in this experiment are conducted using $np=1$ processor and $4$ processors. The output meshes and their mesh quality for a uniform local scale $h$ are displayed in Figure~\ref{fig: naca0012-iso}, which indicates the parallel consistency of the CPAFT algorithm with Laplacian smoothing and the high quality of the generated meshes. Similarly, as illustrated in Figure~\ref{fig: naca0012-aniso}, the CPAFT algorithm with Laplacian smoothing also achieves the parallel consistency and produces high quality meshes when varying local scales are applied.
}

\ 

\begin{figure}[H]
		\centering
		\subfigure[$np=1$]{
		\begin{minipage}[c]{0.312\textwidth}
		  \centering
            \includegraphics[height=0.95\textwidth]{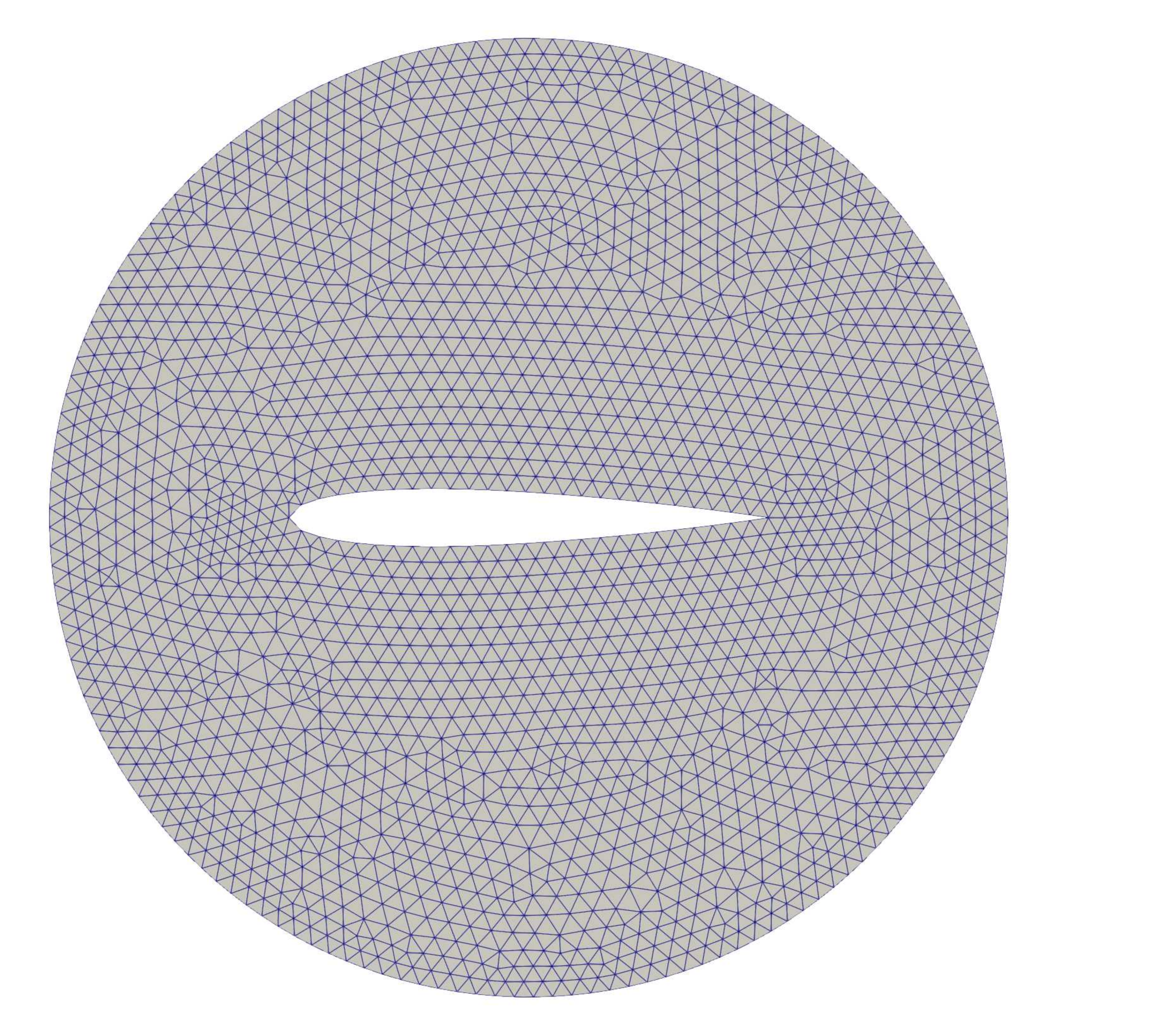}
		\end{minipage}
		} \subfigure[$np=4$]{
		\begin{minipage}[c]{0.312\textwidth}
		  \centering
            \includegraphics[height=0.95\textwidth]{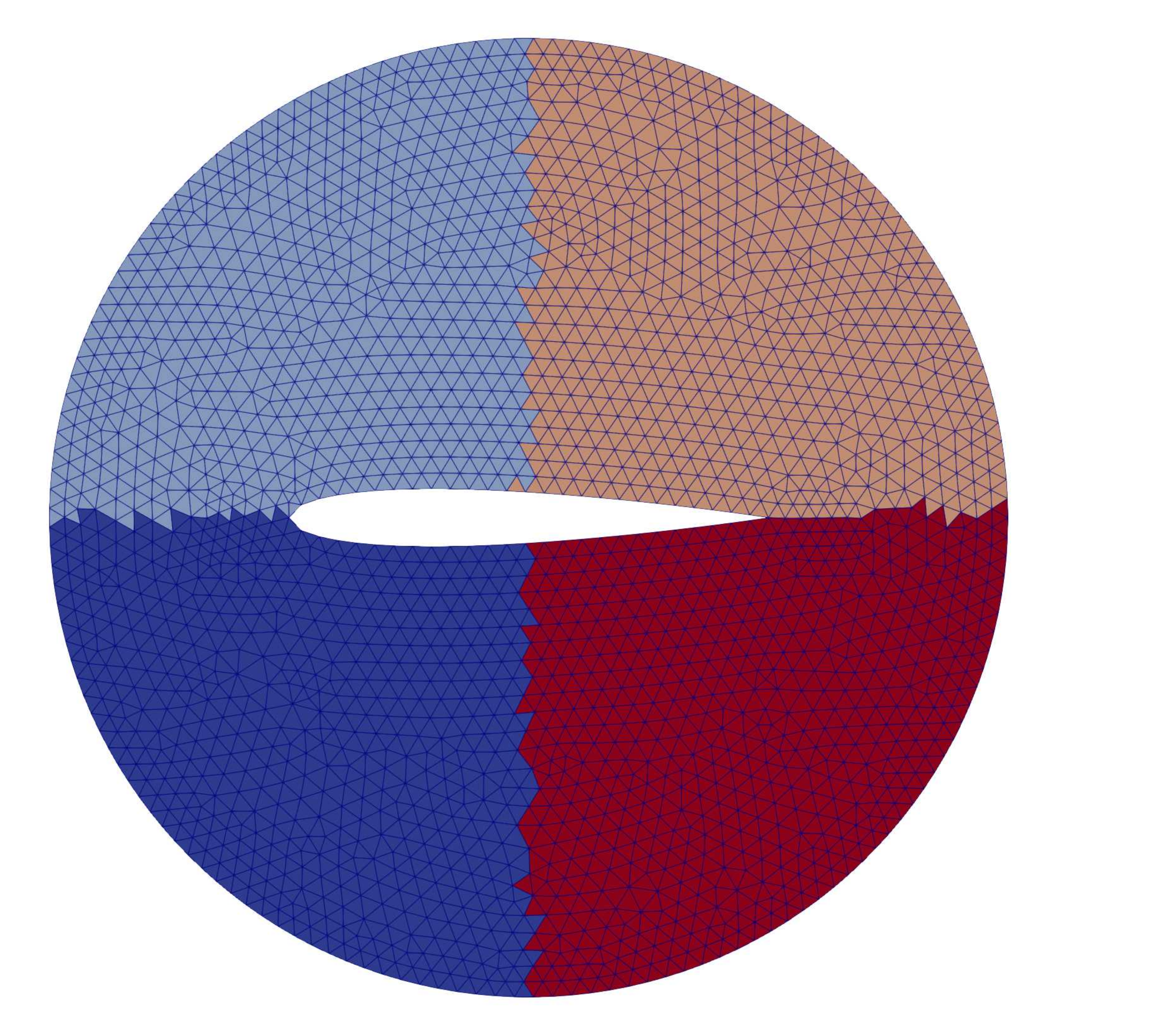}
		\end{minipage}
		} \subfigure[mesh quality]{
		\begin{minipage}[c]{0.312\textwidth}
		  \centering
            \includegraphics[height=0.95\textwidth]{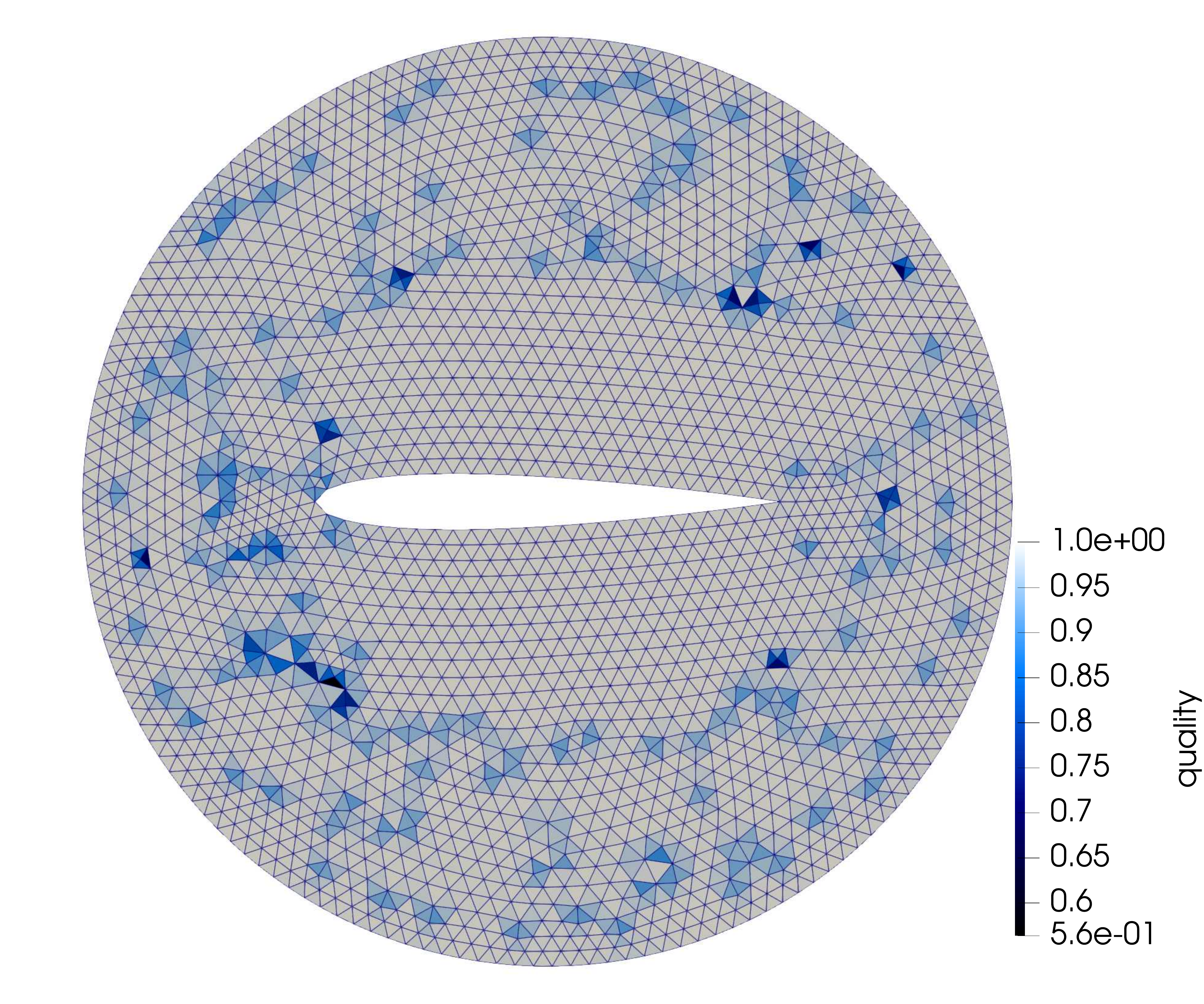}
		\end{minipage}
		} 
		\caption{\color{red} The CPAFT algorithm with Laplacian smoothing and a uniform local scale being $h=0.04$. All simulations start with 209 initial fronts and result in a total of 4,177 elements. The generated mesh with quality $\alpha\geq0.56$.}
		\label{fig: naca0012-iso}
\end{figure}

\ 

\begin{figure}[H]
		\centering
		\subfigure[$np=1$]{
		\begin{minipage}[c]{0.312\textwidth}
		  \centering
            \includegraphics[height=0.95\textwidth]{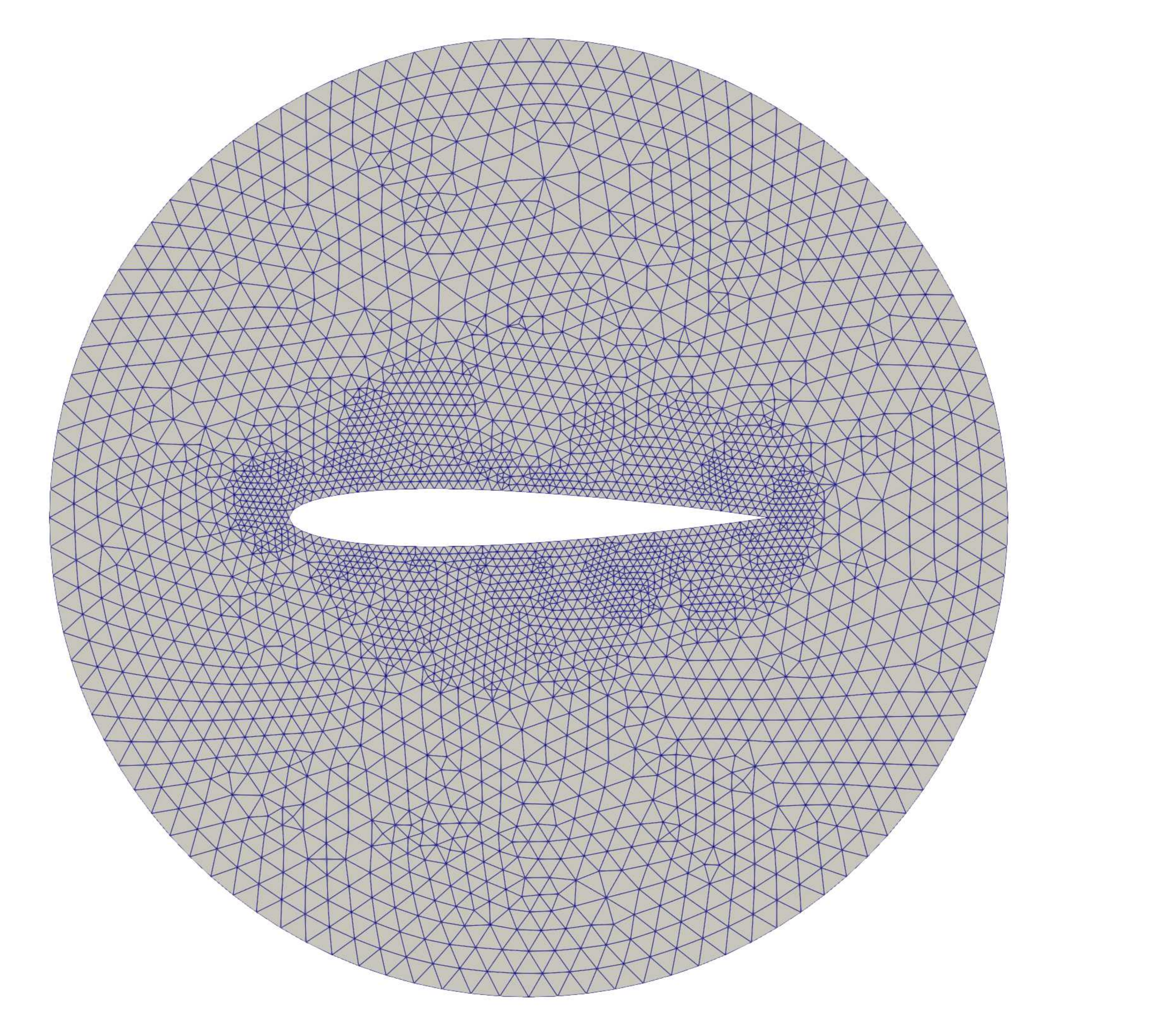}
		\end{minipage}
		} \subfigure[$np=4$]{
		\begin{minipage}[c]{0.312\textwidth}
		  \centering
            \includegraphics[height=0.95\textwidth]{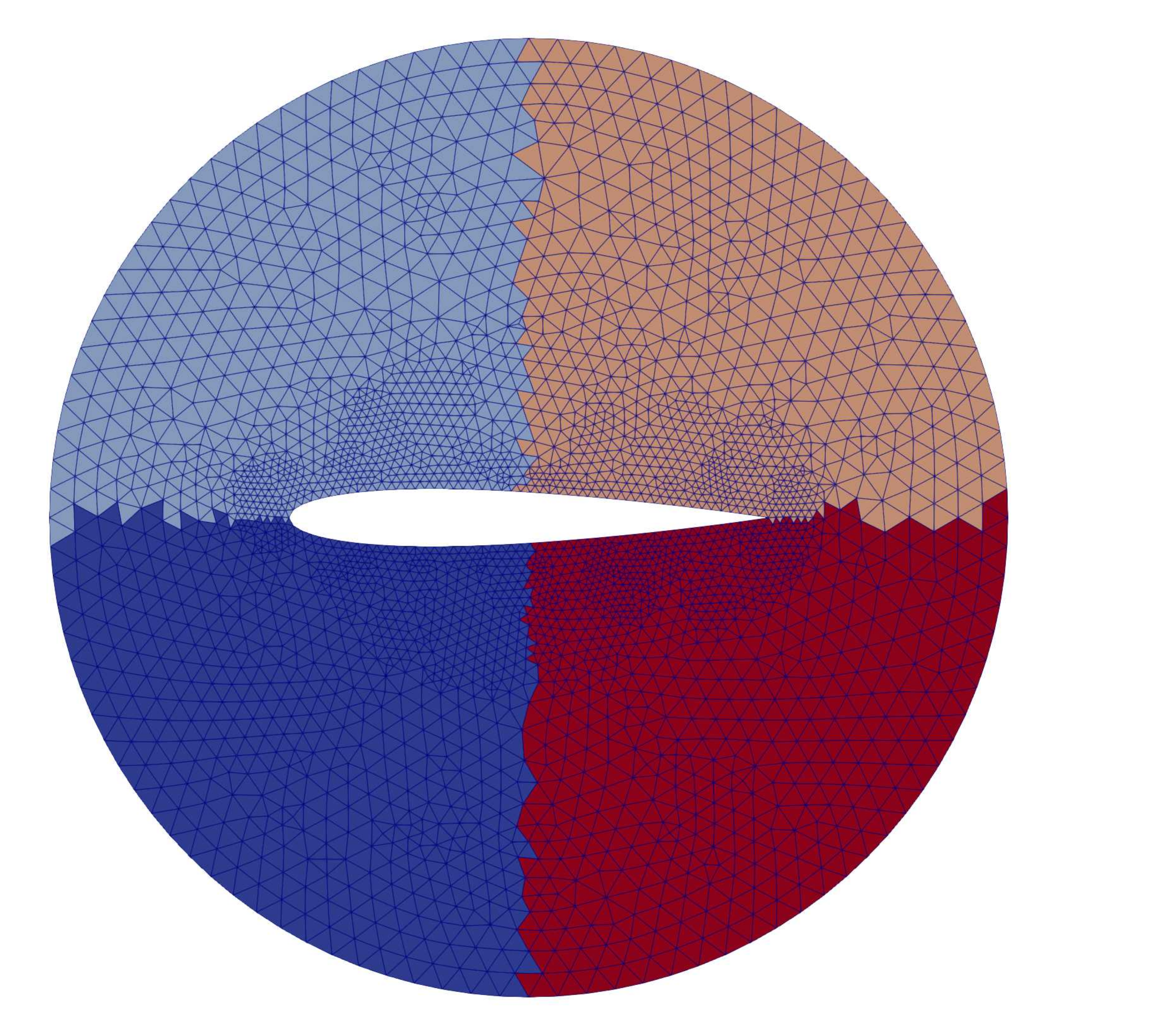}
		\end{minipage}
		} \subfigure[mesh quality]{
		\begin{minipage}[c]{0.312\textwidth}
		  \centering
            \includegraphics[height=0.95\textwidth]{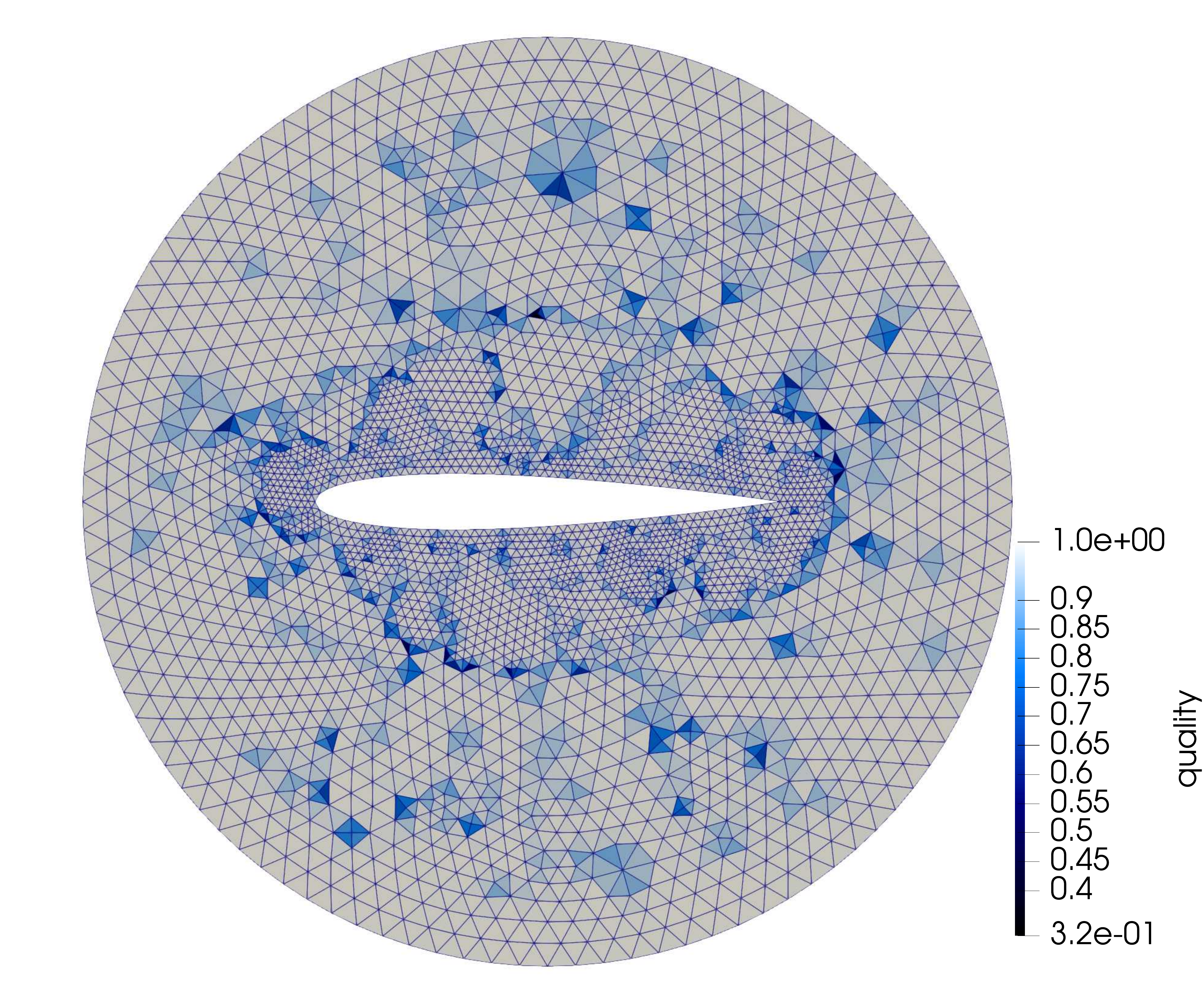}
		\end{minipage}
		} 
		\caption{\color{red} The CPAFT algorithm with Laplacian smoothing. The local scales are 0.02 at the inner boundary and 0.06 at the outer boundary. All simulations start with 206 initial fronts and result in a total of 5,452 elements. The generated mesh with quality $\alpha\geq0.32$.}
		\label{fig: naca0012-aniso}
\end{figure}

{\color{red}
The final 2D test case we present is a concave facial model \citep{gmsh2ddemo}, which serves as a benchmark for evaluating the robustness of mesh generation algorithms in concave domains. We employ the CPAFT algorithm, incorporating Laplacian smoothing, to generate meshes for this model using varying local scales, which increase from left to right boundary. Figure~\ref{fig: face} shows the meshes generated by the CPAFT algorithm using both a single processor ($np = 1$)  and $4$ processors, illustrating the parallel consistency of the algorithm. As shown in Figure~\ref{fig: face}-(c), the generated mesh is of high quality, with $\alpha \geq 0.45$.
}

\begin{figure}[H]
		\centering
		\subfigure[$np=1$]{
		\begin{minipage}[c]{0.312\textwidth}
		  \centering
            \includegraphics[height=1.25\textwidth]{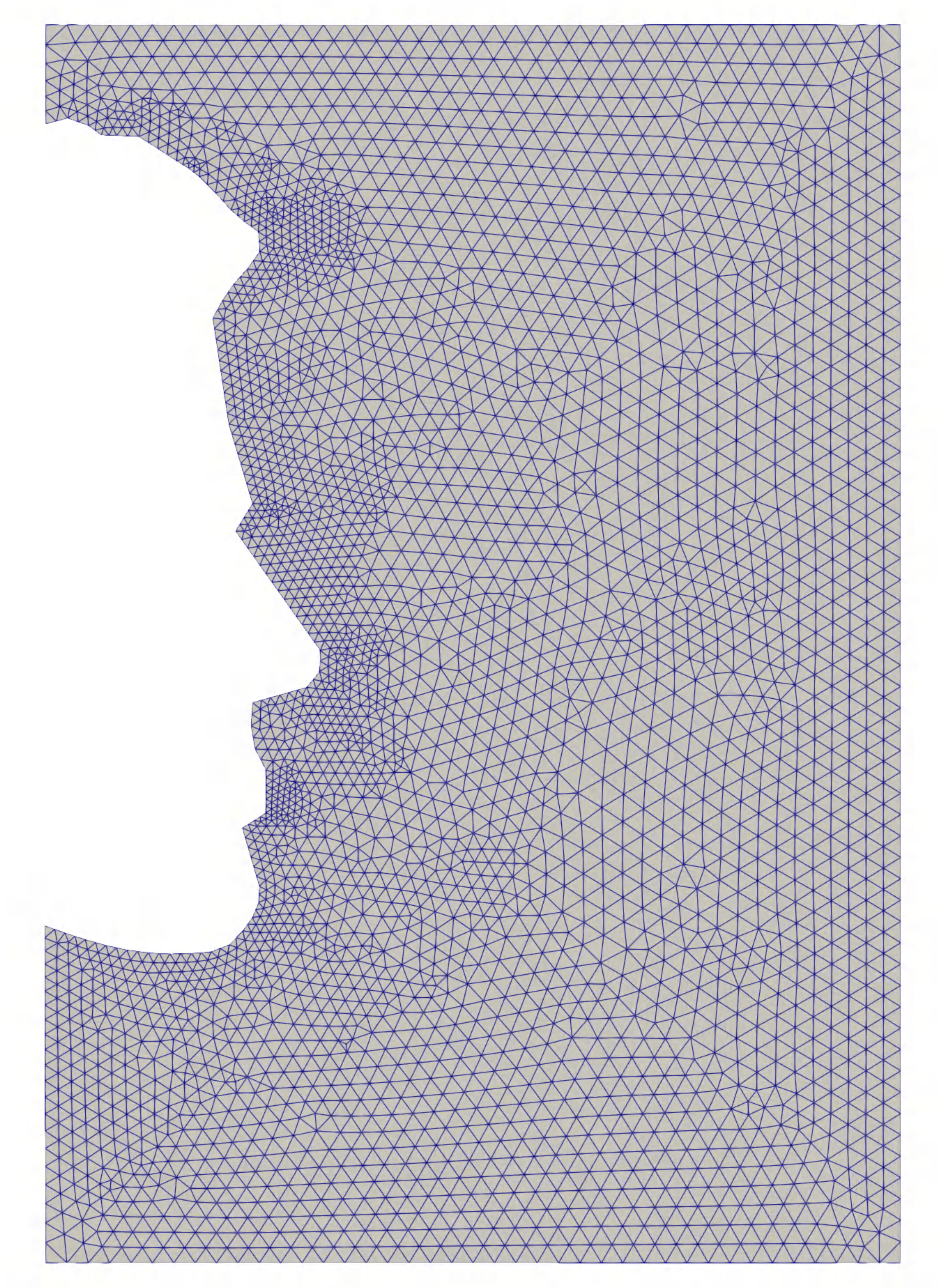}
		\end{minipage}
		} \subfigure[$np=4$]{
		\begin{minipage}[c]{0.312\textwidth}
		  \centering
            \includegraphics[height=1.25\textwidth]{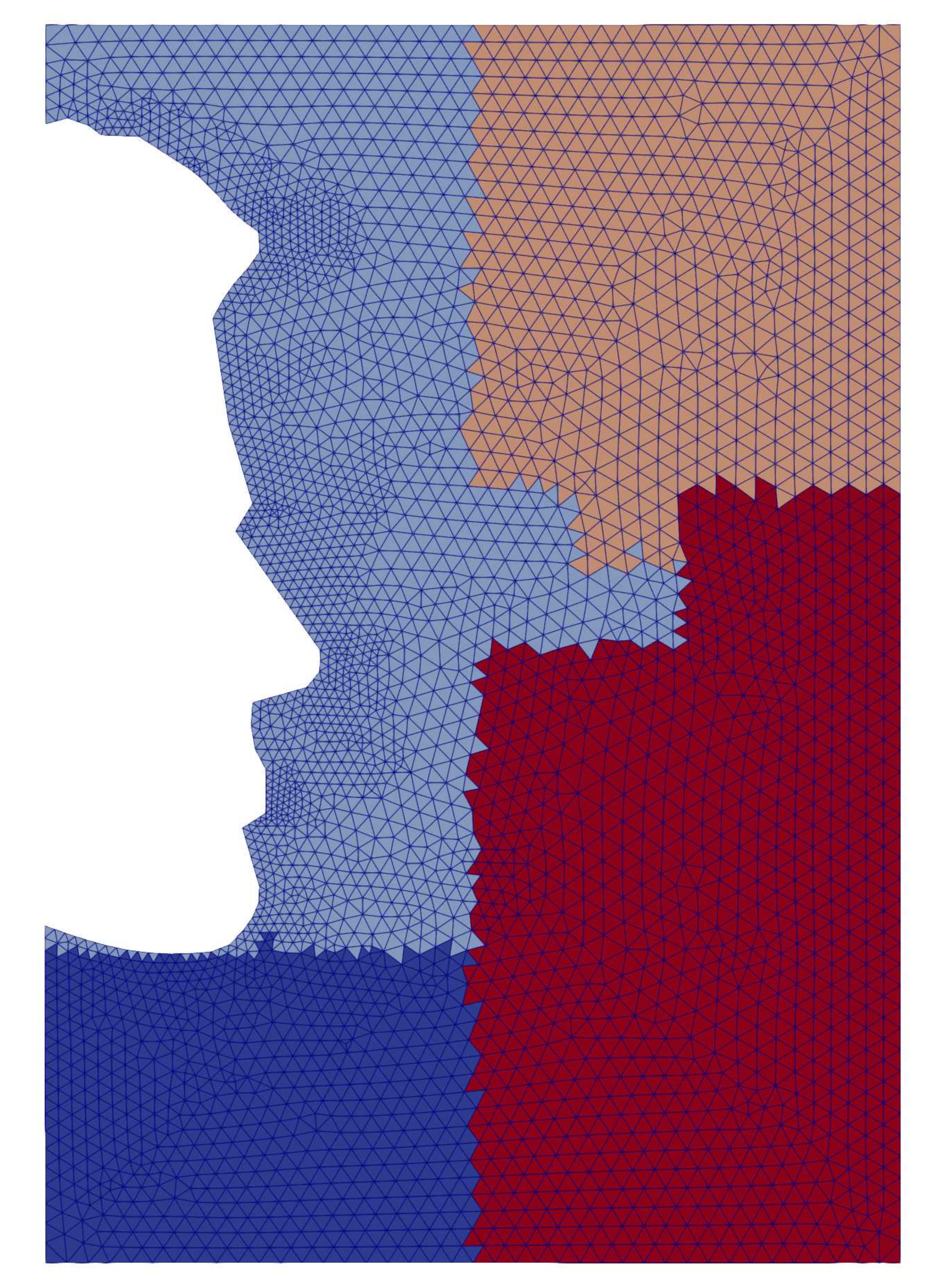}
		\end{minipage}
		} \subfigure[mesh quality]{
		\begin{minipage}[c]{0.312\textwidth}
		  \centering
            \includegraphics[height=1.25\textwidth]{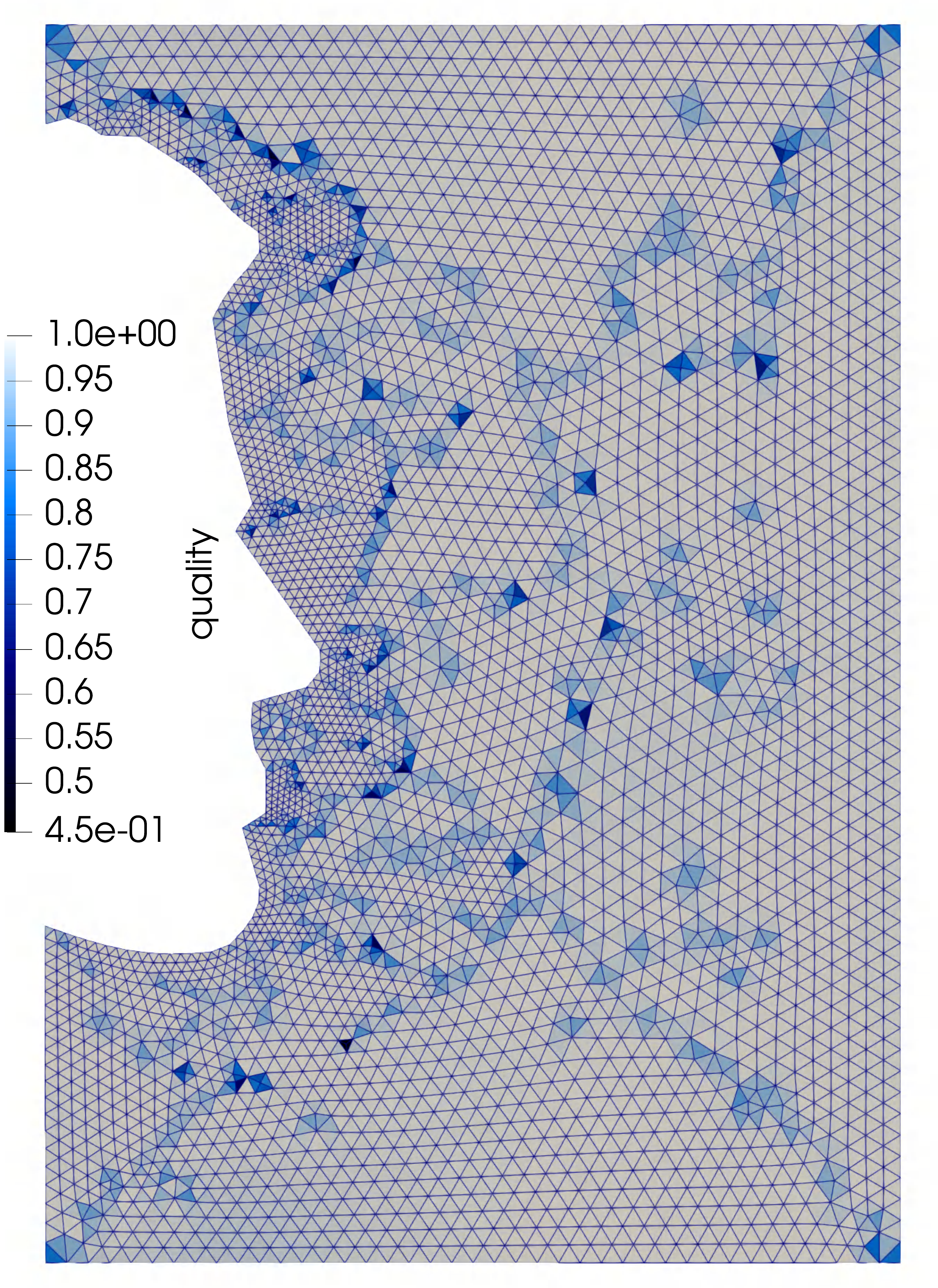}
		\end{minipage}
		} 
		\caption{\color{red} The CPAFT algorithm with Laplacian smoothing. The local scales are 0.03 at the left boundary and 0.05 at the right boundary. All simulations start with 295 initial fronts and result in a total of 6,593 elements. The generated mesh with quality $\alpha\geq0.45$.}
		\label{fig: face}
\end{figure}

\ 

{\color{blue}
\subsection{Effects of parameters on the CPAFT algorithm}

We then focus on analyzing the effects of key parameters in the CPAFT algorithm: (1) the parameter $\beta_{1}$ in Criterion A, which controls distance between newly inserted vertices and the nearby existing fronts; (2) the parameter $\beta_{2}$ in Criterion B, which governs the distance between nearby vertices and the newly generated fronts; and (3) the SFC level $\mathcal{L}$ managing the domain decomposition.
The demo t13 (Figure~\ref{fig: 3D model-op}-(a)) of Gmsh \citep{geuzaine2009gmsh} is used as the benchmark model. Several simulations with parameter ${\cal L}=6$ and varying parameters $(\beta_1, \beta_2)$ are produced. The effects of parameters $(\beta_1, \beta_2)$ on mesh quality and total computing time are summarized in Table~\ref{tab: effect-quality}. From these results, we observe that the mesh quality is almost insensitive to the parameters $(\beta_1,\beta_2)$. As $\beta_1$ increases and $\beta_2$ decreases, the algorithm becomes more inclined to select existing points as advancing points rather than inserting new ones, resulting in a slight decrease in total computing time. After a comprehensive comparison of mesh quality and computing time, the parameters $(\beta_1, \beta_2) = (0.40, 0.15)$ are selected as the optimal configuration.

\begin{table}[H]
    \centering
    \caption{\color{blue}Effects of parameters $(\beta_1, \beta_2)$ on mesh quality and total computing time. Here, $\mathcal{L} = 6$ and the values in this table represent the percentage of elements within the specified range of $\alpha$. ``Time(s)" indicates the total computing time utilizing 8 processors.}
    \setlength{\tabcolsep}{6.0mm}{
        \center
        \begin{tabular} {|c|c|c|c|c|c|}
		  \hline
            \diagbox{$(\beta_1, \beta_2)$}{$\alpha$} &$(0, 0.3)$ &$[0.3, 0.5)$ &$[0.5, 0.7)$ &$[0.7, 1.0]$  &Time\\\cline{1-6}
            $(0.30, 0.25)$ &$3.33\%$ &$10.20\%$ &$45.39\%$ &$41.08\%$ &2.675 \\\cline{1-6}
            $(0.30, 0.15)$ &$2.88\%$ &$8.66\%$ &$45.43\%$ &$43.03\%$ &2.709 \\\cline{1-6}
            $(0.30, 0.05)$ &$2.81\%$ &$8.57\%$ &$46.57\%$ &$42.05\%$ &2.628 \\\cline{1-6}
            $(0.40, 0.25)$ &$2.78\%$ &$8.48\%$ &$46.44\%$ &$42.30\%$ &2.335 \\\cline{1-6} 
            $(0.40, 0.15)$ &$2.81\%$ &$8.80\%$ &$48.16\%$ &$40.24\%$ &2.329 \\\cline{1-6}
            $(0.40, 0.05)$ &$2.99\%$ &$9.23\%$ &$46.67\%$ &$41.10\%$ &2.285 \\\cline{1-6}
            $(0.50, 0.25)$ &$3.15\%$ &$10.04\%$ &$46.98\%$ &$39.83\%$ &2.151 \\\cline{1-6}
            $(0.50, 0.15)$ &$2.98\%$ &$9.79\%$ &$47.48\%$ &$39.74\%$ &2.148 \\\cline{1-6}
            $(0.50, 0.05)$ &$4.04\%$ &$10.66\%$ &$47.30\%$ &$38.00\%$ &2.134 \\\hline
		\end{tabular}
    }
    \label{tab: effect-quality}
\end{table}

\begin{figure}[H]
		\centering
		\subfigure[$(0.30, 0.25)$]{
		\begin{minipage}[c]{0.29\textwidth}
		  \centering
            \includegraphics[width=0.95\textwidth]{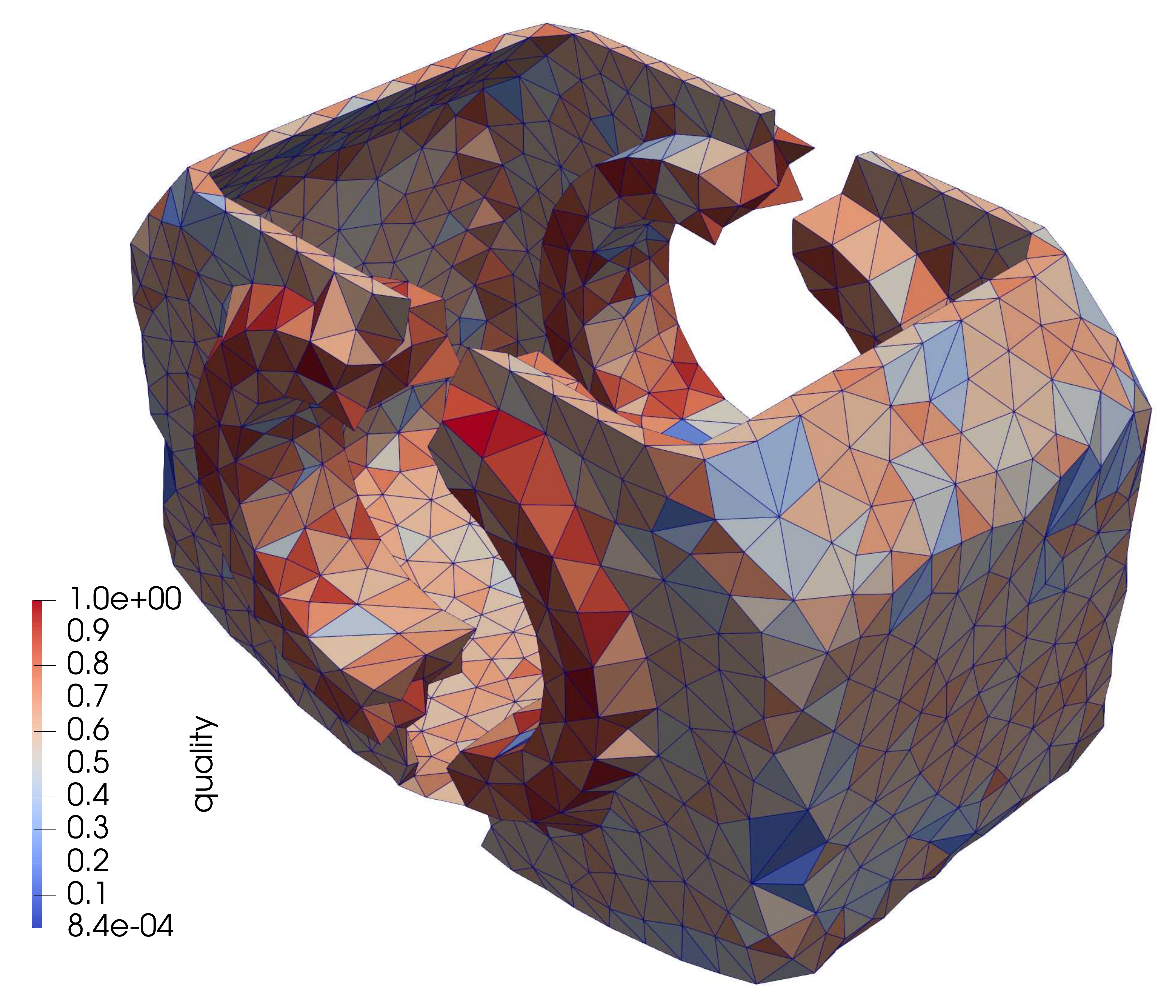}
		\end{minipage}
		} \subfigure[$(0.30, 0.15)$]{
		\begin{minipage}[c]{0.29\textwidth}
		  \centering
            \includegraphics[width=0.95\textwidth]{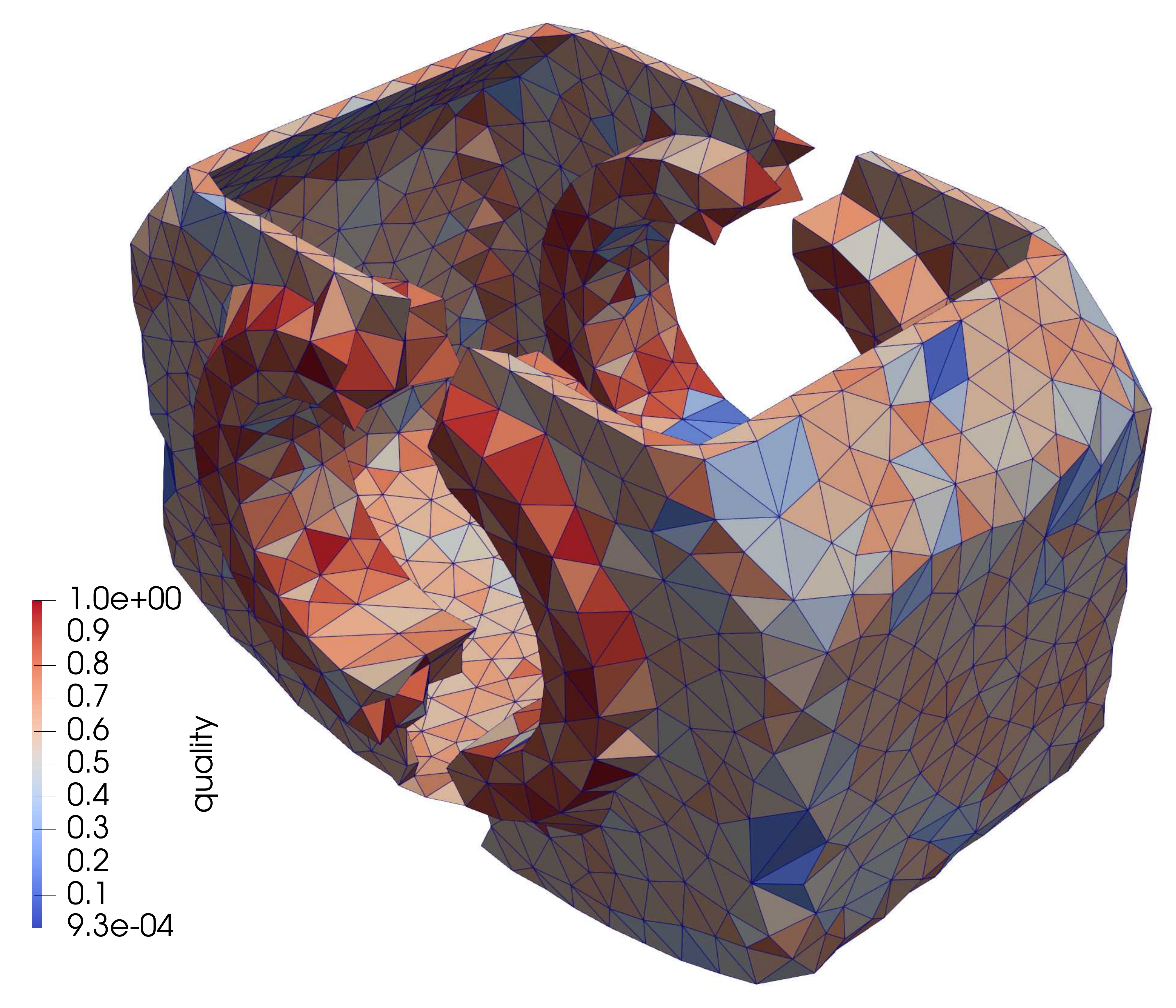}
		\end{minipage}
		} \subfigure[$(0.30, 0.05)$]{
		\begin{minipage}[c]{0.29\textwidth}
		  \centering
            \includegraphics[width=0.95\textwidth]{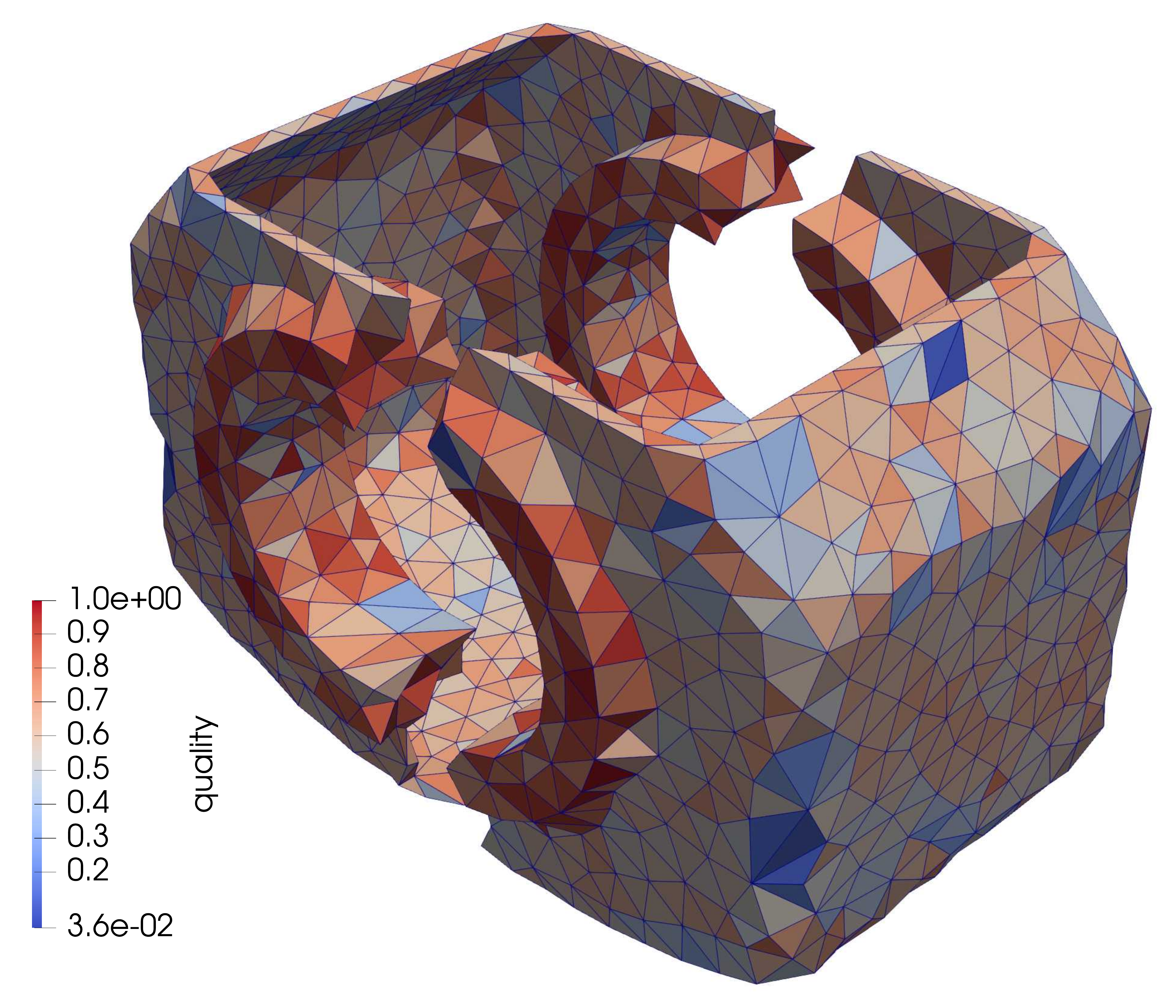}
		\end{minipage}
		} 
  
        \subfigure[$(0.40, 0.25)$]{
		\begin{minipage}[c]{0.29\textwidth}
		  \centering
            \includegraphics[width=0.95\textwidth]{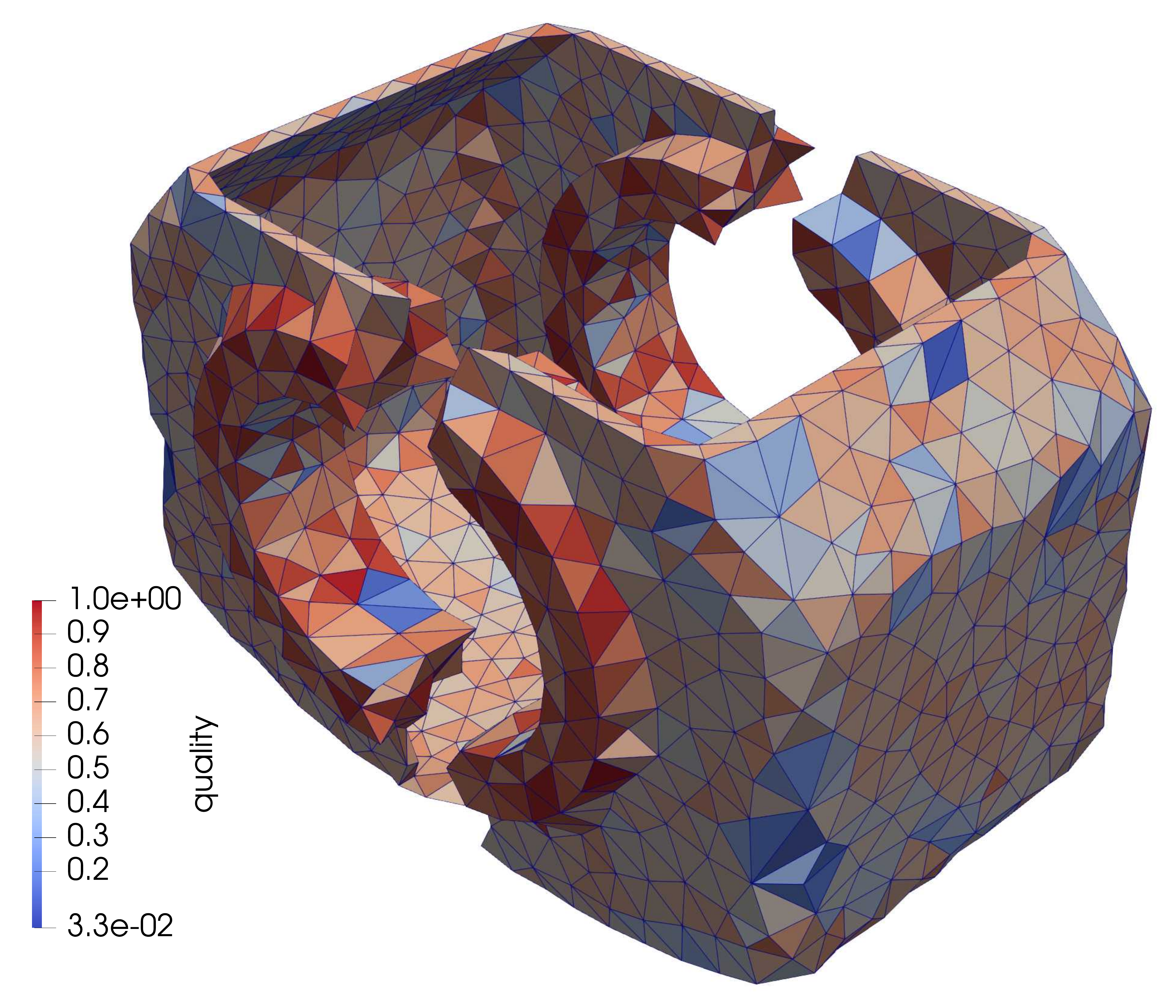}
		\end{minipage}
		} \subfigure[$(0.40, 0.15)$]{
		\begin{minipage}[c]{0.29\textwidth}
		  \centering
            \includegraphics[width=0.95\textwidth]{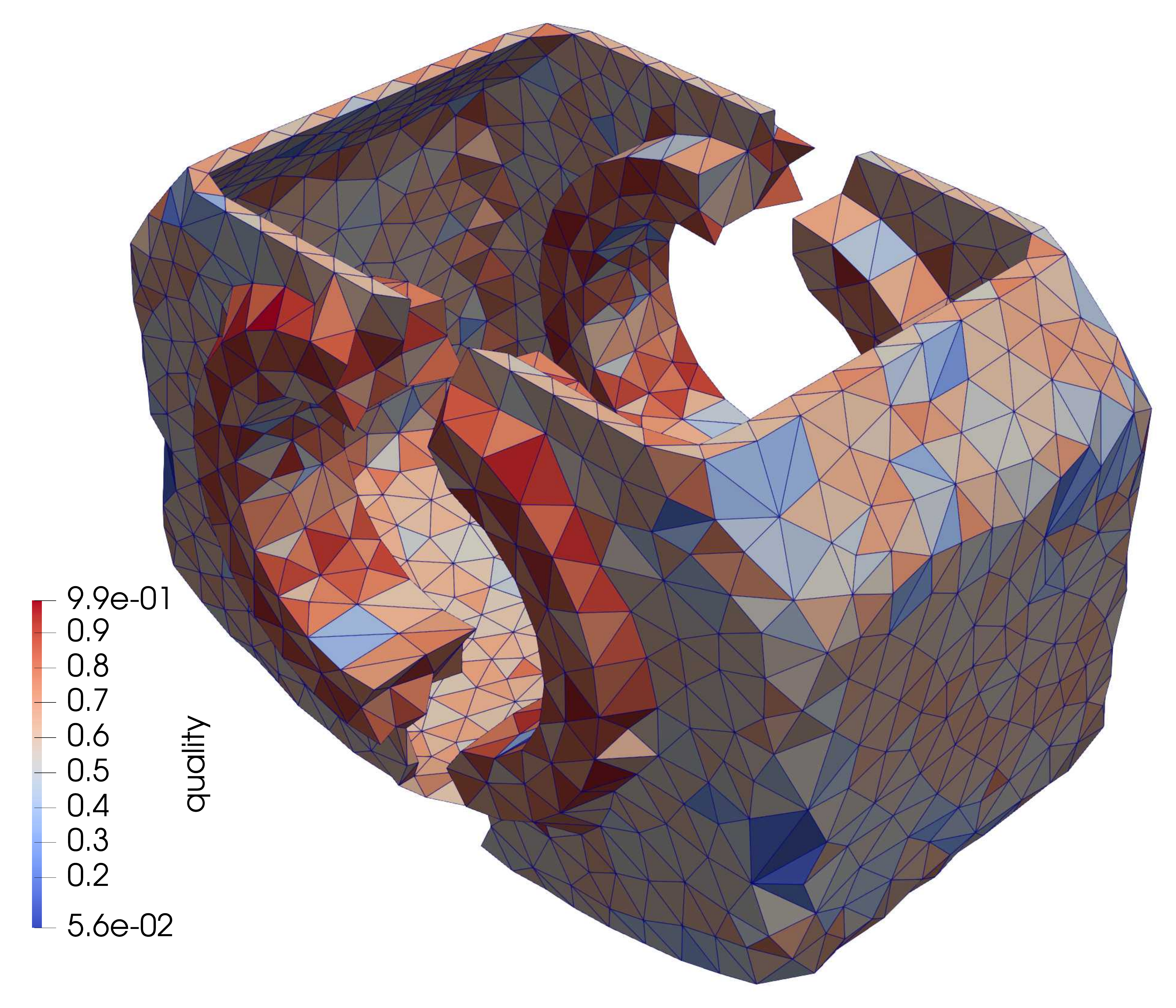}
		\end{minipage}
		} \subfigure[$(0.40, 0.05)$]{
		\begin{minipage}[c]{0.29\textwidth}
		  \centering
            \includegraphics[width=0.95\textwidth]{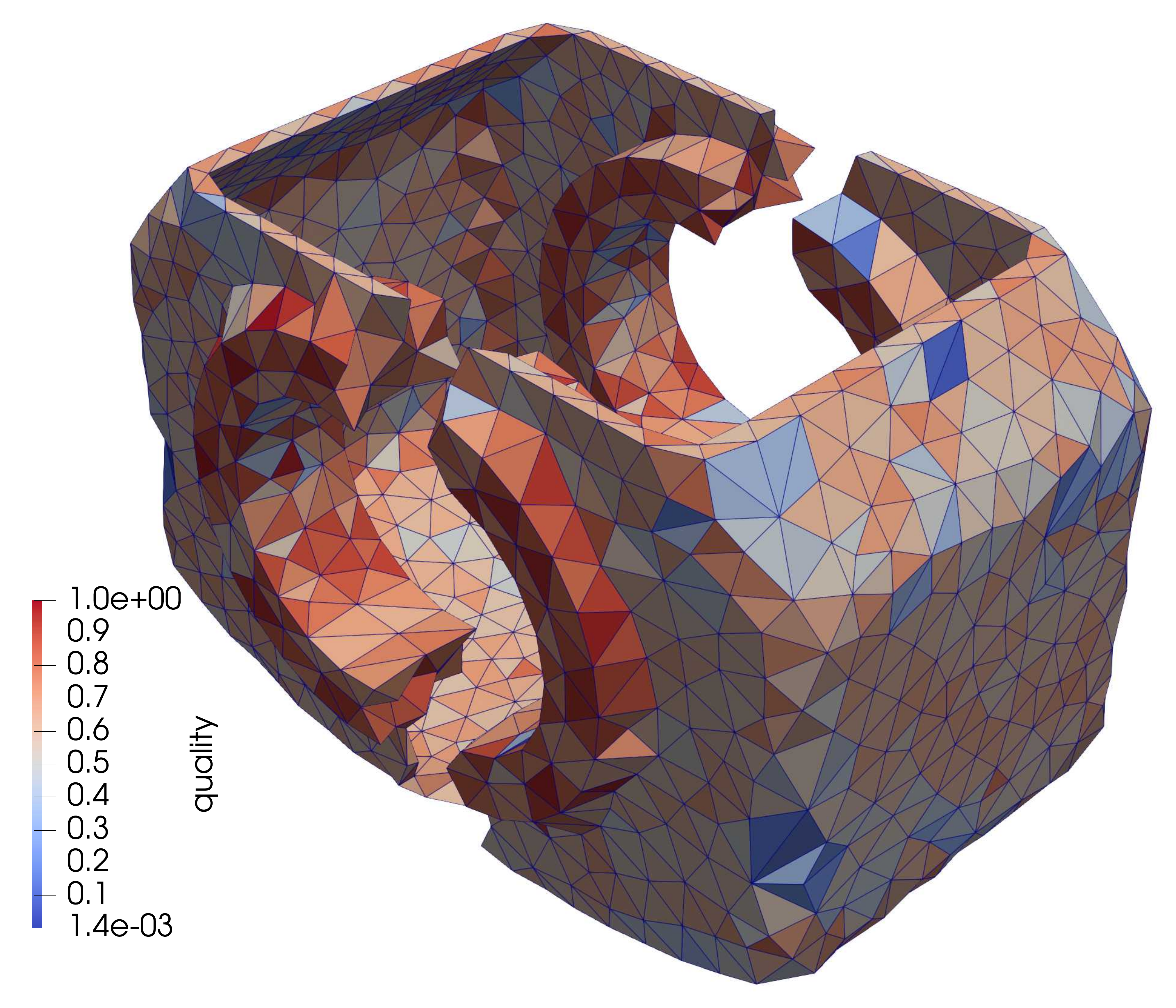}
		\end{minipage}
		} 
  
        \subfigure[$(0.50, 0.25)$]{
		\begin{minipage}[c]{0.29\textwidth}
		  \centering
            \includegraphics[width=0.95\textwidth]{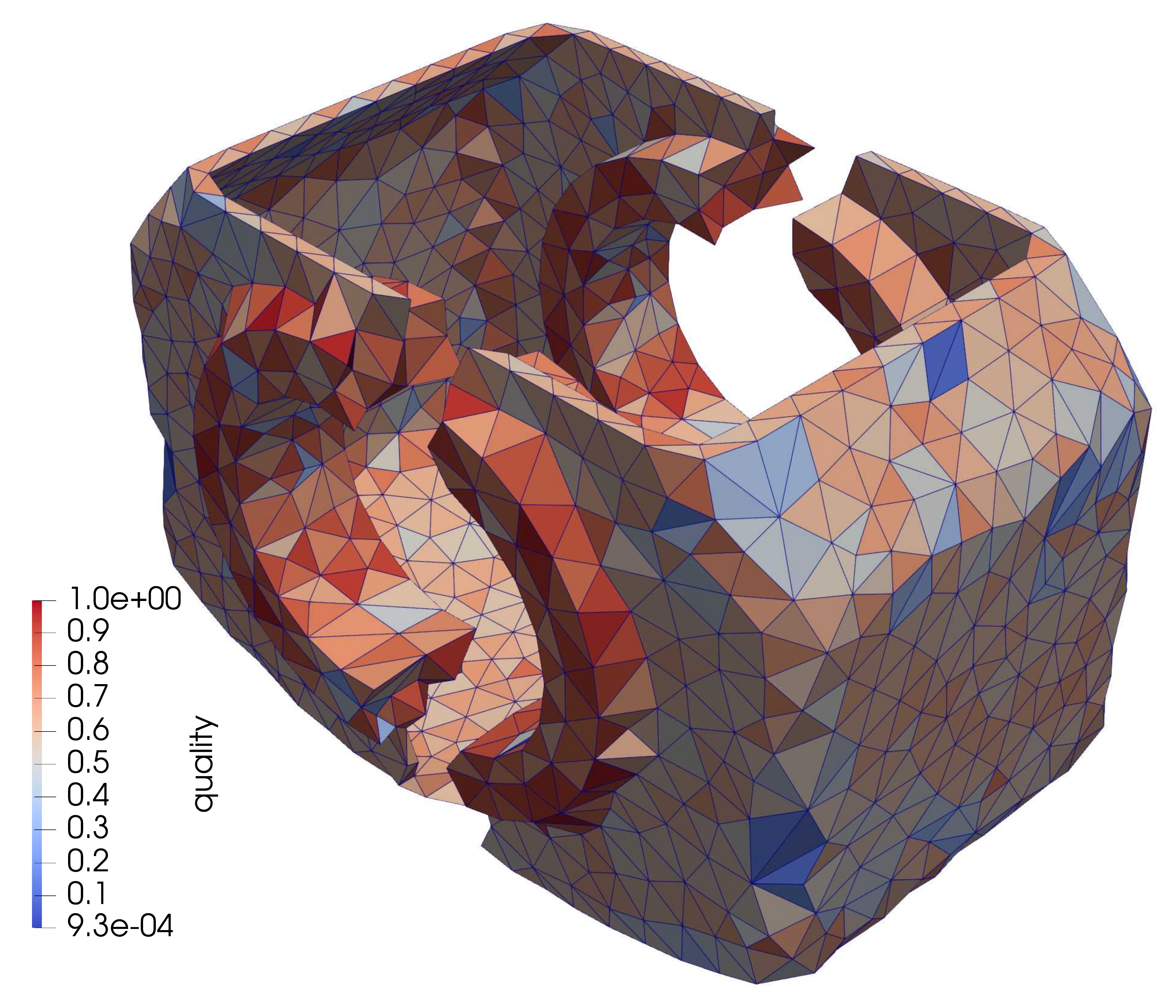}
		\end{minipage}
		} \subfigure[$(0.50, 0.15)$]{
		\begin{minipage}[c]{0.29\textwidth}
		  \centering
            \includegraphics[width=0.95\textwidth]{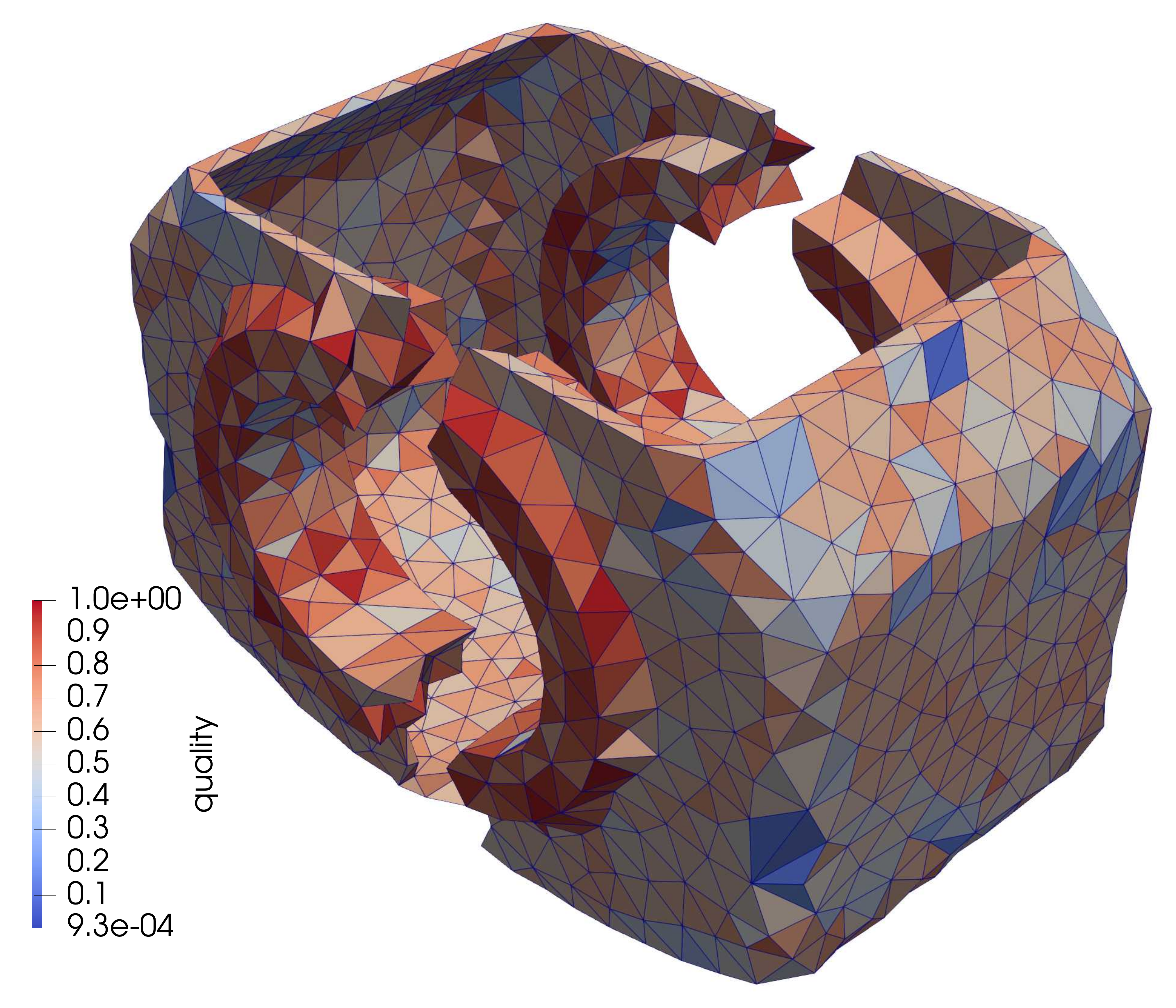}
		\end{minipage}
		} \subfigure[$(0.50, 0.05)$]{
		\begin{minipage}[c]{0.29\textwidth}
		  \centering
            \includegraphics[width=0.95\textwidth]{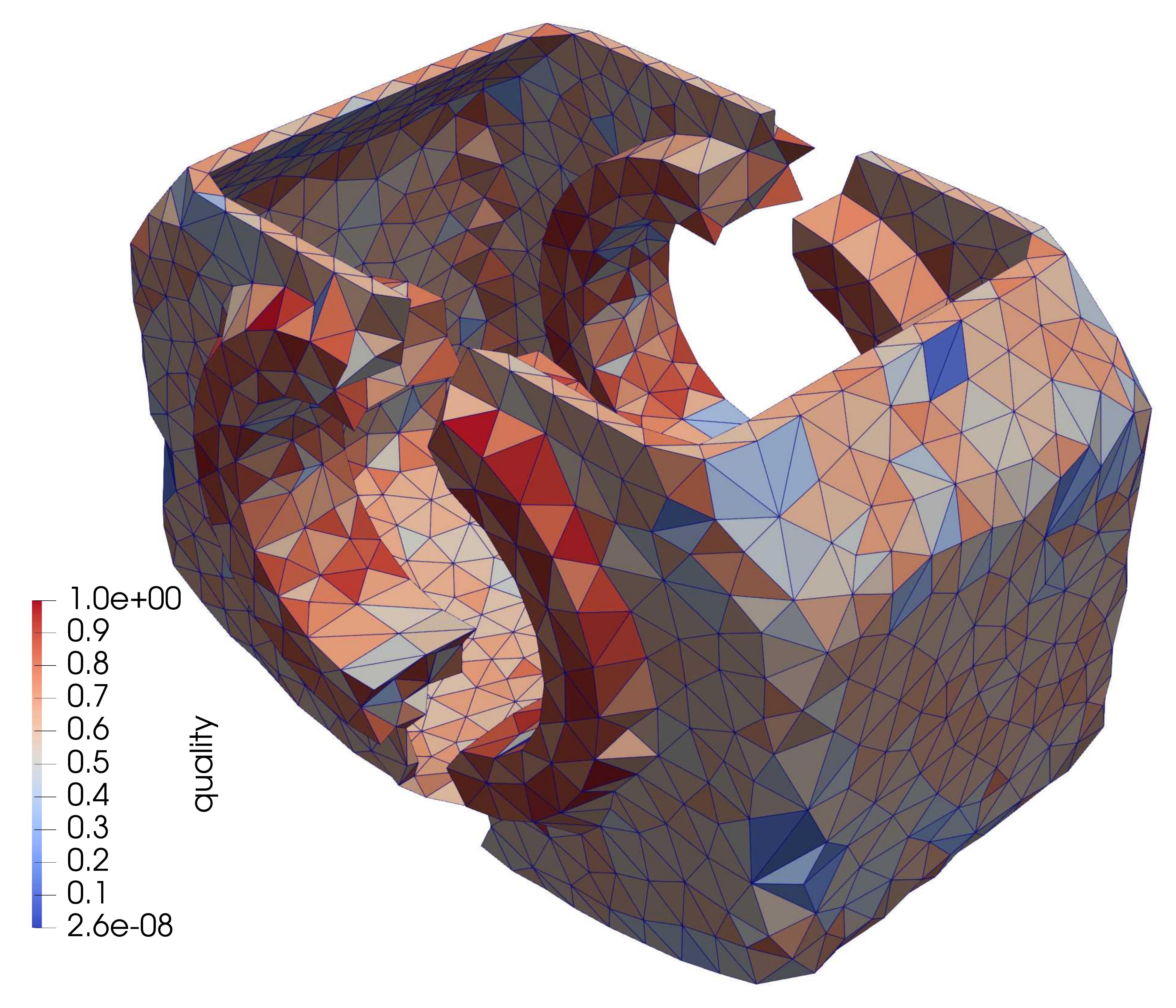}
		\end{minipage}
		} 
		\caption{\color{blue} (a)-(i) Mesh qualities of demo t13 meshes after Laplacian smoothing and local remeshing with different parameters $(\beta_1, \beta_2)$. To show the narrow features in the model, we decompose the mesh into two parts. 
  }
		\label{fig: t13-para-quality}
\end{figure}

Next, we set $(\beta_1,\beta_2)=(0.40,0.15)$ to study the effect of parameter $\cal L$. Three simulations with $\mathcal{L}=5,6,7$ are conducted. 
Table~\ref{tab: effect-quality-L} shows the quality of the resulting meshes for each $\mathcal{L}$, illustrating that the mesh quality is largely insensitive to changes in $\mathcal{L}$. For each $\mathcal{L}$, we run simulations with number of processors varying from 1, 2, 4, to 8. The computing time and parallel efficiency for all simulations are presented in Table~\ref{tab: effect-runtime}. The results indicate a slight increase in computing time as the level $\mathcal{L}$ of SFC increases. On the other hand, increasing $\mathcal{L}$ enhances parallel efficiency, which could offer performance benefits when using a larger number of processors.
\begin{table}[H]
    \centering
    \caption{\color{blue}Effects of parameter $\mathcal{L}$ on mesh quality. Here, $(\beta_1, \beta_2) = (0.40, 0.15)$ and the values in this table represent the percentage of elements within the specified range of $\alpha$.  }
    \setlength{\tabcolsep}{9.2mm}{
        \center
        \begin{tabular} {|c|c|c|c|c|}
		  \hline
             \diagbox{$\mathcal{L}$}{$\alpha$} &$(0, 0.3)$ &$[0.3, 0.5)$ &$[0.5, 0.7)$ &$[0.7, 1.0]$ \\\cline{1-5}
            $\mathcal{L}=5$ &$2.86\%$ &$9.74\%$ &$46.81\%$ &$40.58\%$ \\\cline{1-5}
            $\mathcal{L}=6$ &$2.81\%$ &$8.80\%$ &$48.16\%$ &$40.24\%$ \\\cline{1-5}
            $\mathcal{L}=7$ &$2.88\%$ &$9.34\%$ &$49.05\%$ &$38.73\%$ \\\hline
		\end{tabular}
    }
    \label{tab: effect-quality-L}
\end{table}


\begin{table}[H]
	\centering
	\caption{\color{blue} Effects of parameter $\mathcal{L}$ on total computing time (``Time(s)") and parallel efficiency (``Eff($\%$)"). Here $(\beta_1, \beta_2) = (0.40, 0.15)$ and ``Eff ($\%$)$=T_1/(nT_n)$" denotes the parallel efficiency using $n$ processors. }
        \setlength{\tabcolsep}{8.4mm}{
        \center
        \begin{tabular} {|c|c|c|c|c|c|}
		  \hline
            &$np$ &$1$ &$2$ &$4$ &$8$  \\\cline{1-6}
            \multirow{2}{*}{$\mathcal{L} = 5$} &Time &13.524 &7.138 &3.965 &2.357 \\\cline{2-6}
            &Eff($\%$) &$100\%$ &$94.73\%$ &$85.27\%$ &$71.72\%$ \\\cline{1-6}
            \multirow{2}{*}{$\mathcal{L} = 6$} &Time &14.380 &7.477 &4.166 &2.329 \\\cline{2-6}
            &Eff($\%$) &$100\%$ &$96.16\%$ &$86.29\%$ &$77.18\%$ \\\cline{1-6}
            \multirow{2}{*}{$\mathcal{L} = 7$} &Time &18.204 &9.117 &5.217 &2.936 \\\cline{2-6}
            &Eff($\%$) &$100\%$ &$99.84\%$ &$87.23\%$ &$77.50\%$ \\\cline{1-6}
		\end{tabular}
        }
	\label{tab: effect-runtime}
\end{table}

}

\ 

\subsection{Mesh quality of the CPAFT algorithm}
\label{meshquality}

{\color{red}
We demonstrate the mesh quality of the meshes generated by the CPAFT algorithm for several 3D models and compare it with three open-source software tools including Gmsh \citep{geuzaine2009gmsh}, NETGEN \citep{schoberl1997netgen}, and TetGen \citep{hang2015tetgen}. 
First, we carry out experiments on two multiply-connected domains for tetrahedral mesh generation, which are demo t13 and demo t20 from Gmsh \citep{geuzaine2009gmsh}. For each demo, we run three simulations using 1, 2, and 4 processors, respectively.  
The geometries of the two demos, along with the corresponding unstructured meshes generated by the CPAFT algorithm with Laplacian smoothing, are displayed in Figure~\ref{fig: 3D model-op}, indicating both the parallel consistency of the CPAFT algorithm and the high quality of the generated meshes. 
For these two demos, the mesh quality indicators $\alpha$ for the output meshes produced by the original CPAFT algorithm, the CPAFT algorithm with Laplacian smoothing, and the CPAFT algorithm with both Laplacian smoothing and local remeshing are presented in Figure~\ref{fig: quality3D-op}-(a). 
The results clearly demonstrate that the Laplacian smoothing and local remeshing approaches can significantly improve mesh quality.
Given the favorable geometry of both demos, all three versions of the CPAFT algorithm are able to generate tetrahedral meshes with satisfactory quality, comparable to that of the reference open-source software tools, as shown in Figure~\ref{fig: quality3D-op}-(b).

\begin{figure}[H]
    \centering
		\subfigure[Geometry (.stl)]{
		\begin{minipage}[c]{0.465\textwidth}
		  \centering
            \includegraphics[height=0.45\textwidth]{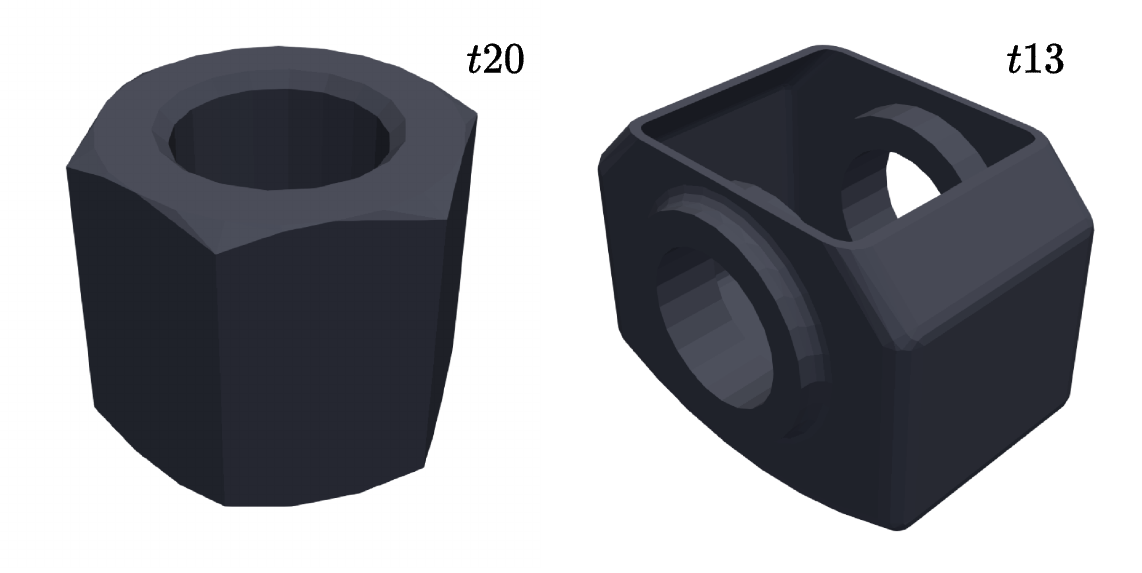}
		\end{minipage}
		} \subfigure[$np = $ 1]{
		\begin{minipage}[c]{0.465\textwidth}
		  \centering
            \includegraphics[height=0.45\textwidth]{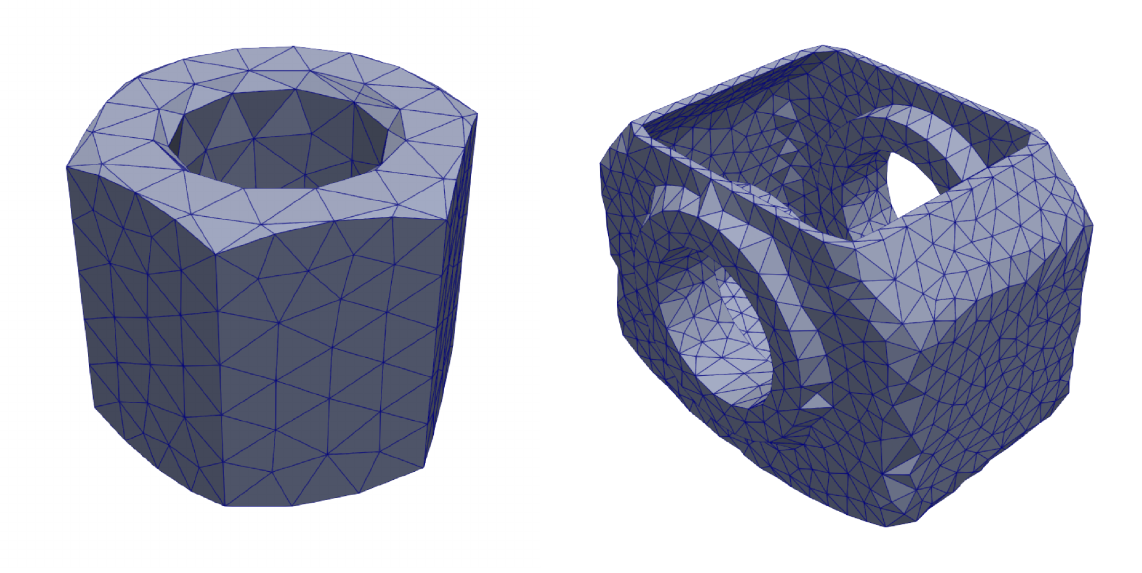}
		\end{minipage}
		} 
  
        \subfigure[$np = $ 2]{
		\begin{minipage}[c]{0.465\textwidth}
		  \centering
            \includegraphics[height=0.45\textwidth]{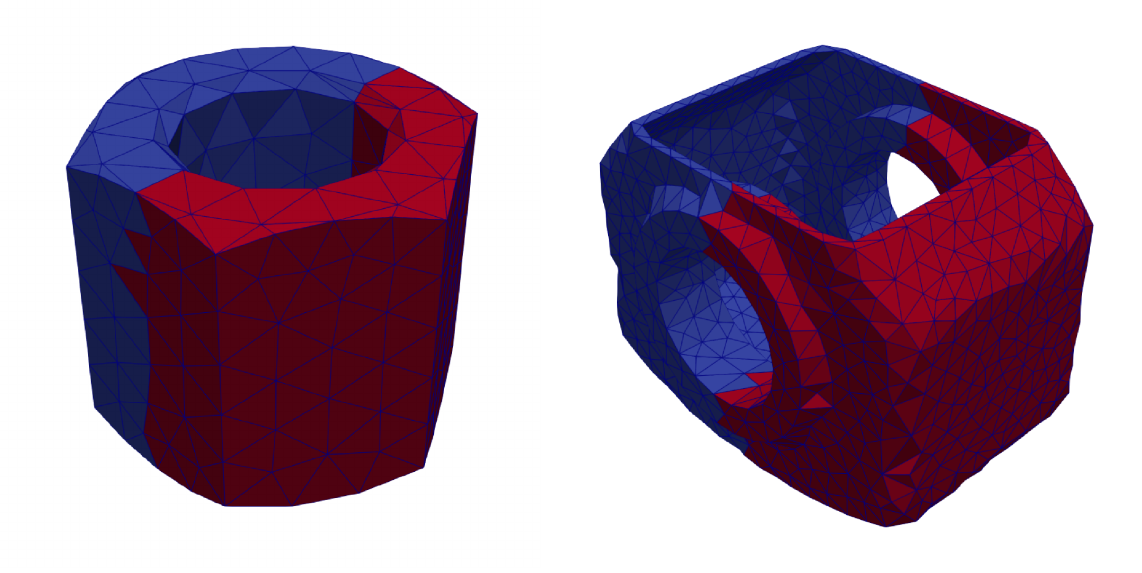}
		\end{minipage}
		} \subfigure[$np = $ 4]{
		\begin{minipage}[c]{0.465\textwidth}
		  \centering
            \includegraphics[height=0.45\textwidth]{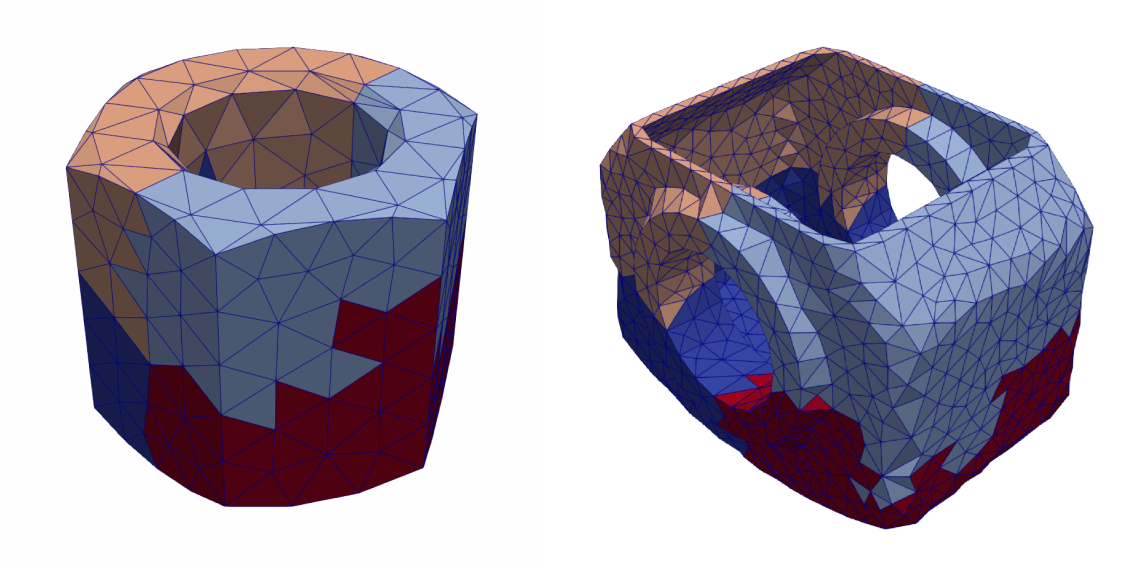}
		\end{minipage}
		}
    \caption{(a) The geometries for two demos t13 and t20.
    The meshes generated by the CPAFT with Laplacian smoothing using  (b) $np=1$, (c) $np=2$, and (d) $np=4$ processors, respectively. All simulations lead to 6,035 elements for t13 and 1,443 elements for t20, respectively.}
    \label{fig: 3D model-op}
\end{figure}


\ 

\begin{figure}[H]
    \centering
    \subfigure[The three versions of CPAFT algorithm]{
        \begin{minipage}[c]{0.5\textwidth}
		  \centering
            \includegraphics[width=1.0\textwidth]{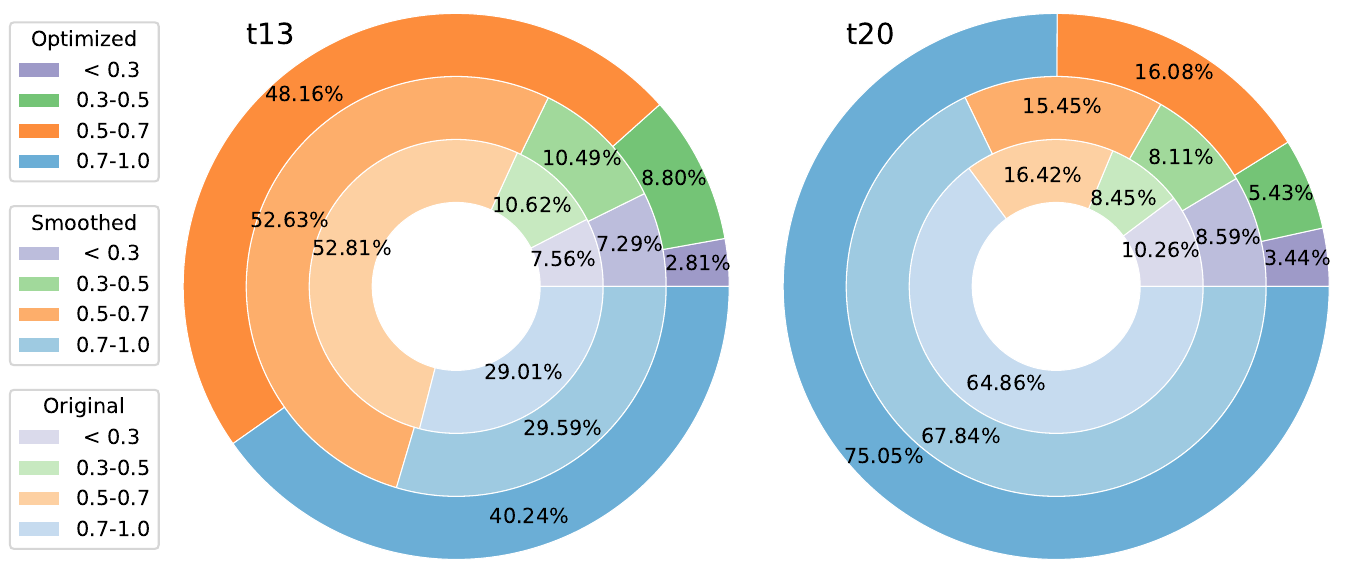}
        \end{minipage}
    }\subfigure[Reference software tools]{
		\begin{minipage}[c]{0.5\textwidth}
		  \centering
            \includegraphics[width=1.0\textwidth]{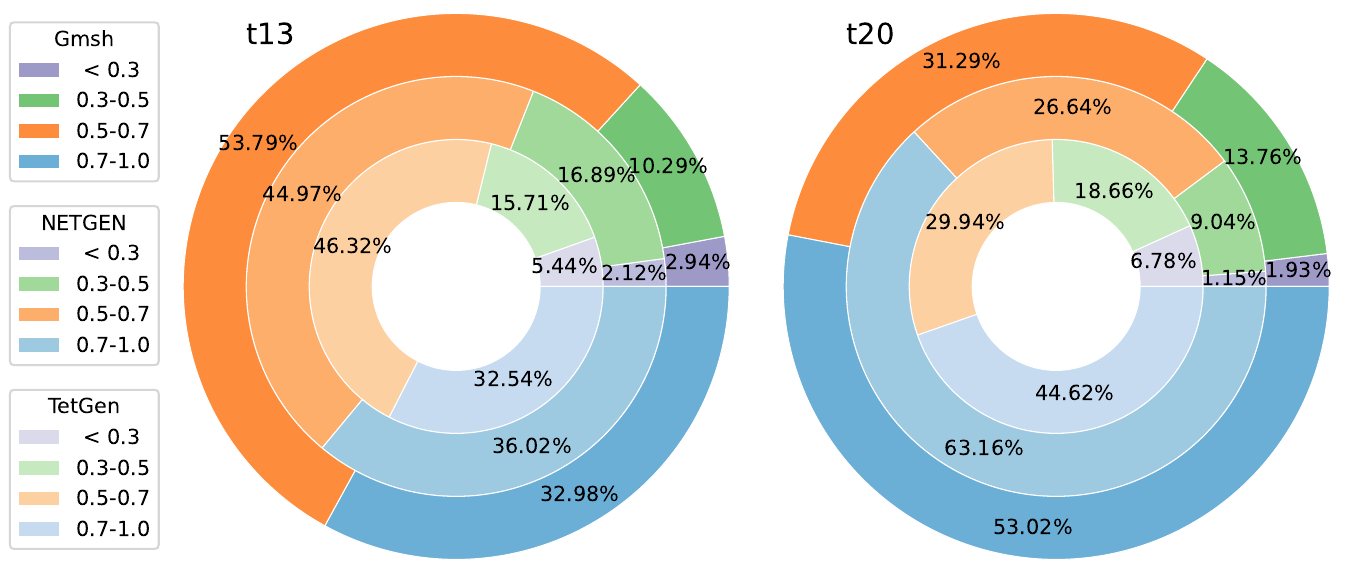}
        \end{minipage}
    }
    \caption{\color{red} Breakdown of mesh quality indicators $\alpha$ for demos t13 and t20: (a) the original CPAFT algorithm (denoted as `Original'), the CPAFT algorithm with Laplacian smoothing (denoted as `Smoothed'), and the CPAFT algorithm with both Laplacian smoothing and local remeshing (denoted as `Optimized'), respectively;  (b) TetGen, NETGEN, and Gmsh, respectively.}
    \label{fig: quality3D-op}
\end{figure}

\ 

}

{\color{blue}

We then consider two challenging models with more complex geometries: a space shuttle model and a real artery vessel model. For each model, we run two simulations using 1 and 8 processors, respectively. 
The geometries of the two models, along with the corresponding unstructured meshes generated by the CPAFT algorithm with Laplacian smoothing, are displayed in Figure~\ref{fig: space shuttle} and \ref{fig: artery}. These results demonstrate the parallel consistency of the CPAFT algorithm and the high quality of the generated meshes, even for these challenging models with relatively flat and narrow features. The mesh quality indicators $\alpha$ for the output meshes produced by the three versions of the CPAFT algorithm are presented in Figure~\ref{fig: quality-space-artery-op}.
Due to the complex geometry of these challenging models, meshes generated by the original CPAFT algorithm include more than $ 20\%$ elements with an indicator $\alpha<0.3$. While the mesh quality of the original CPAFT algorithm surpasses that of TetGen, it falls short when compared to Gmsh and NETGEN. Although the Laplacian smoothing approach improves mesh quality, the meshes generated by the CPAFT algorithm with Laplacian smoothing still contain more than than $ 15\%$ elements with an indicator $\alpha<0.3$. To further enhance mesh quality, we combine CPAFT with Laplacian smoothing and a parallel local remeshing approach, resulting in an optimized version of the CPAFT algorithm. The mesh quality indicators $\alpha$ for the optimized version are shown in Figure~\ref{fig: quality-space-artery-op}-(a), which indicates that the optimized version outperforms all baselines in terms of mesh quality.

\ 

\begin{figure}[H]
    \centering
    \subfigure[Geometry (.stl)]{
		\begin{minipage}[c]{0.31\textwidth}
		  \centering
            \begin{tikzpicture}
                \node[anchor=south west,inner sep=0] (image) at (0, 0) {\includegraphics[width=1.1\textwidth]{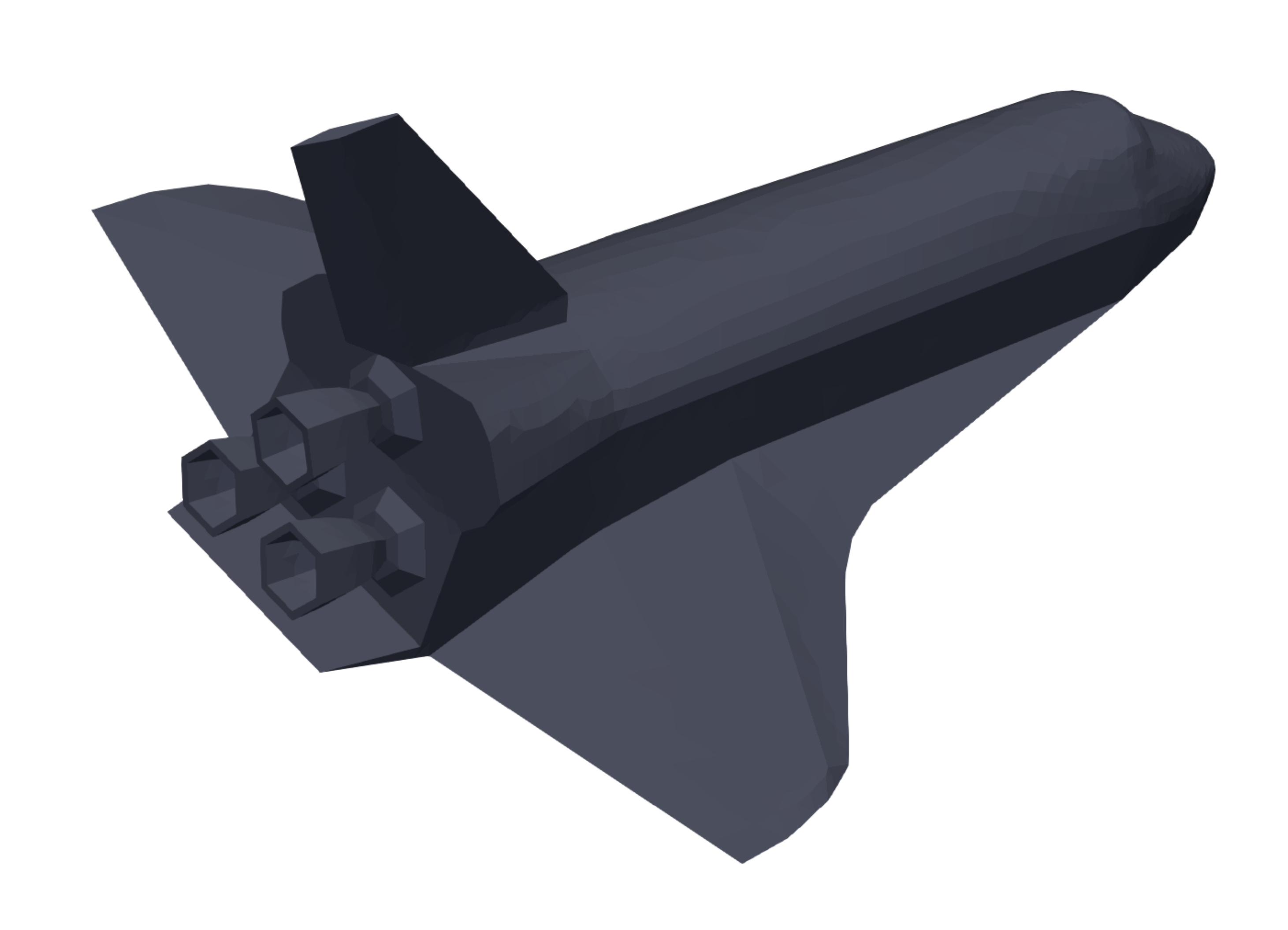}};
            \end{tikzpicture}
		\end{minipage}
    } \subfigure[$np = $ 1]{
		\begin{minipage}[c]{0.31\textwidth}
		  \centering
            \begin{tikzpicture}
                \node[anchor=south west,inner sep=0] (image) at (0, 0) {\includegraphics[width=1.1\textwidth]{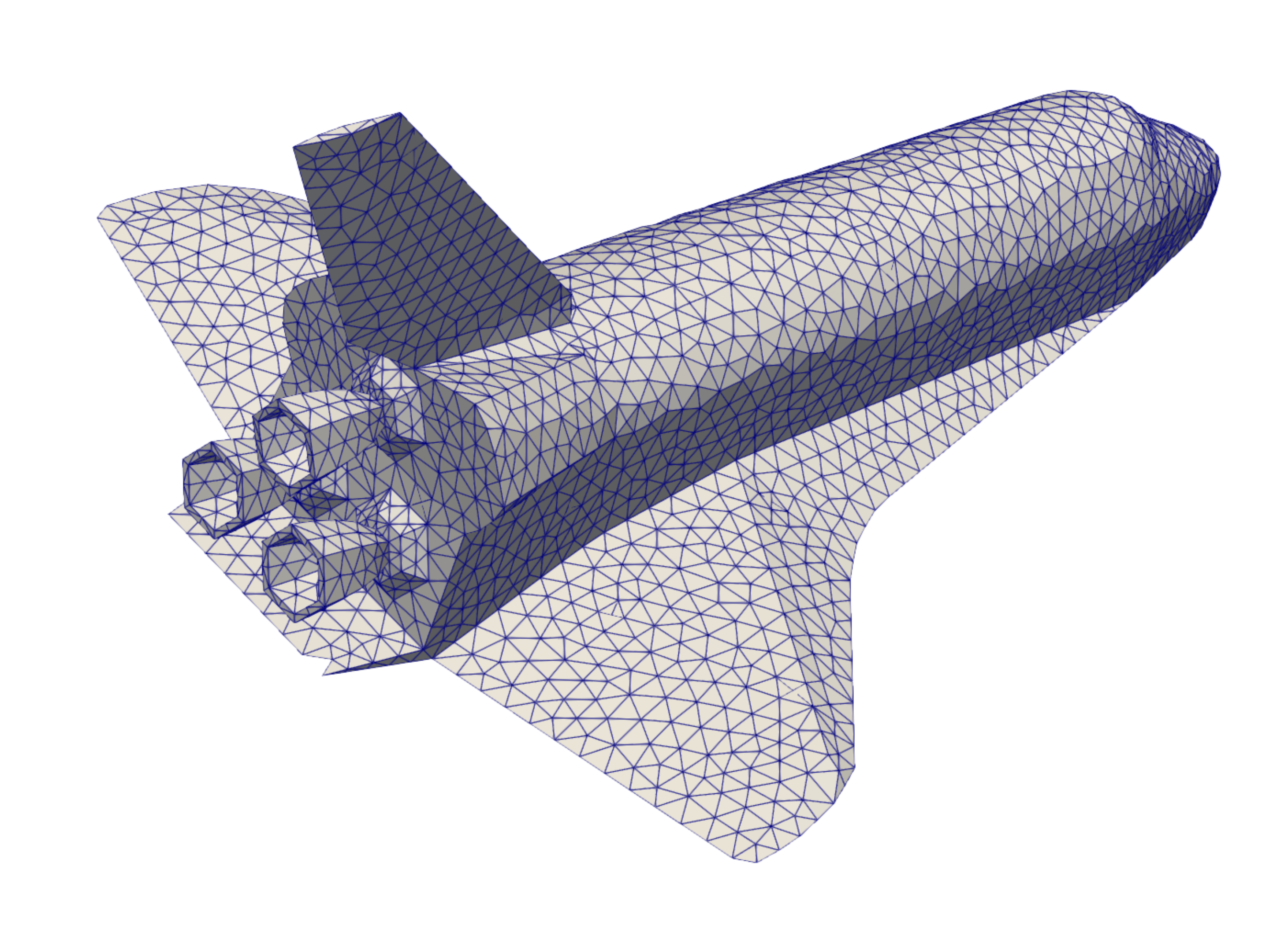}};
                \node[anchor=north east,inner sep=0] (image) at (2.0,1.5) {\includegraphics[width=0.42\textwidth]{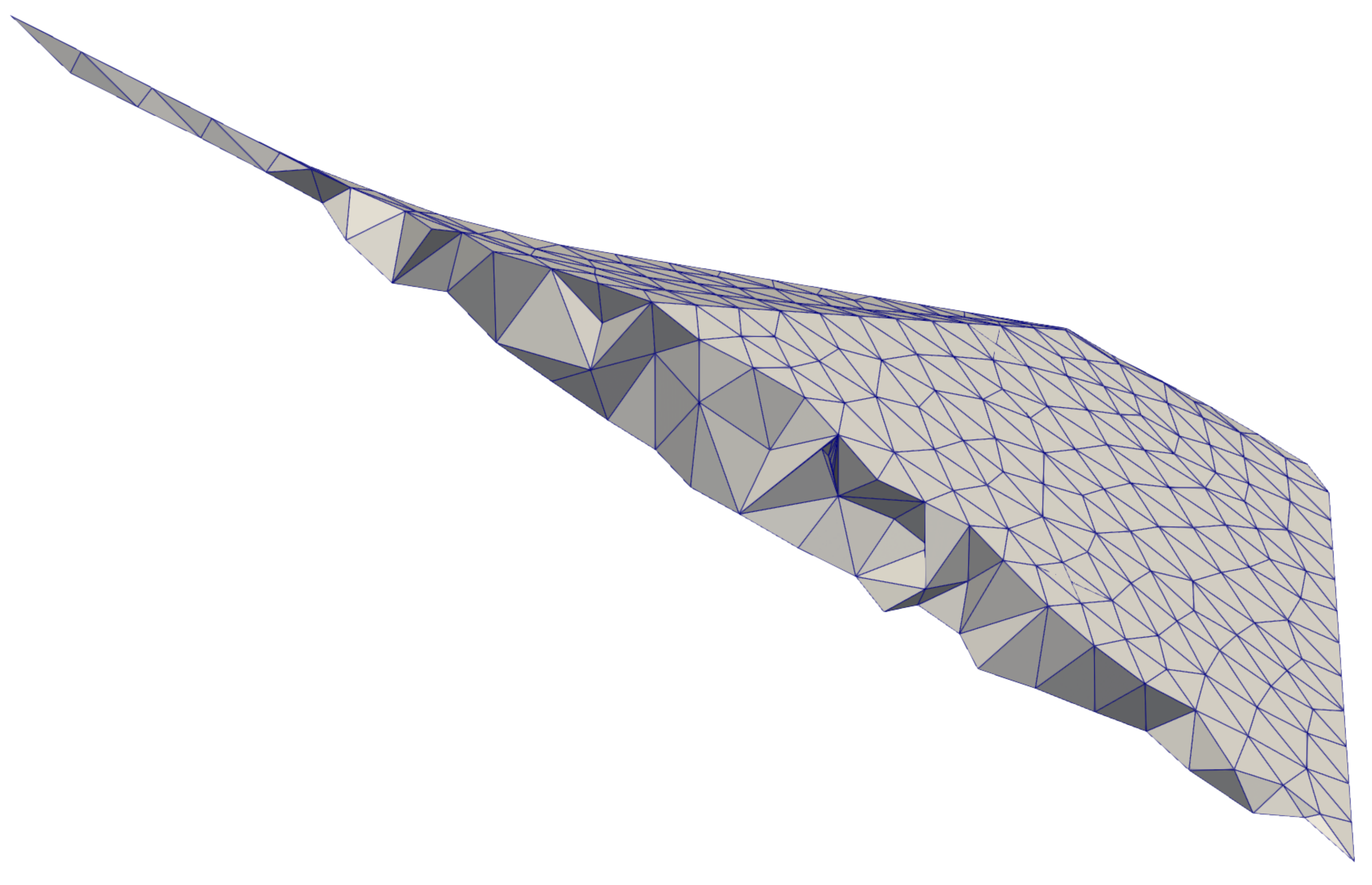}};
                \draw[dashed, black] (image.south west) rectangle (image.north east);
                \draw[dashed, black, line width=0.5pt] (2.0,1.0) -- (4.1, 2.0);
                \draw[thick, black, ->] (2.5,1.1) arc (35:160:0.5cm and 0.5cm);
            \end{tikzpicture}
		\end{minipage}
    } \subfigure[$np = $ 8]{
		\begin{minipage}[c]{0.31\textwidth}
		  \centering
            \begin{tikzpicture}
                \node[anchor=south west,inner sep=0] (image) at (0, 0) {\includegraphics[width=1.1\textwidth]{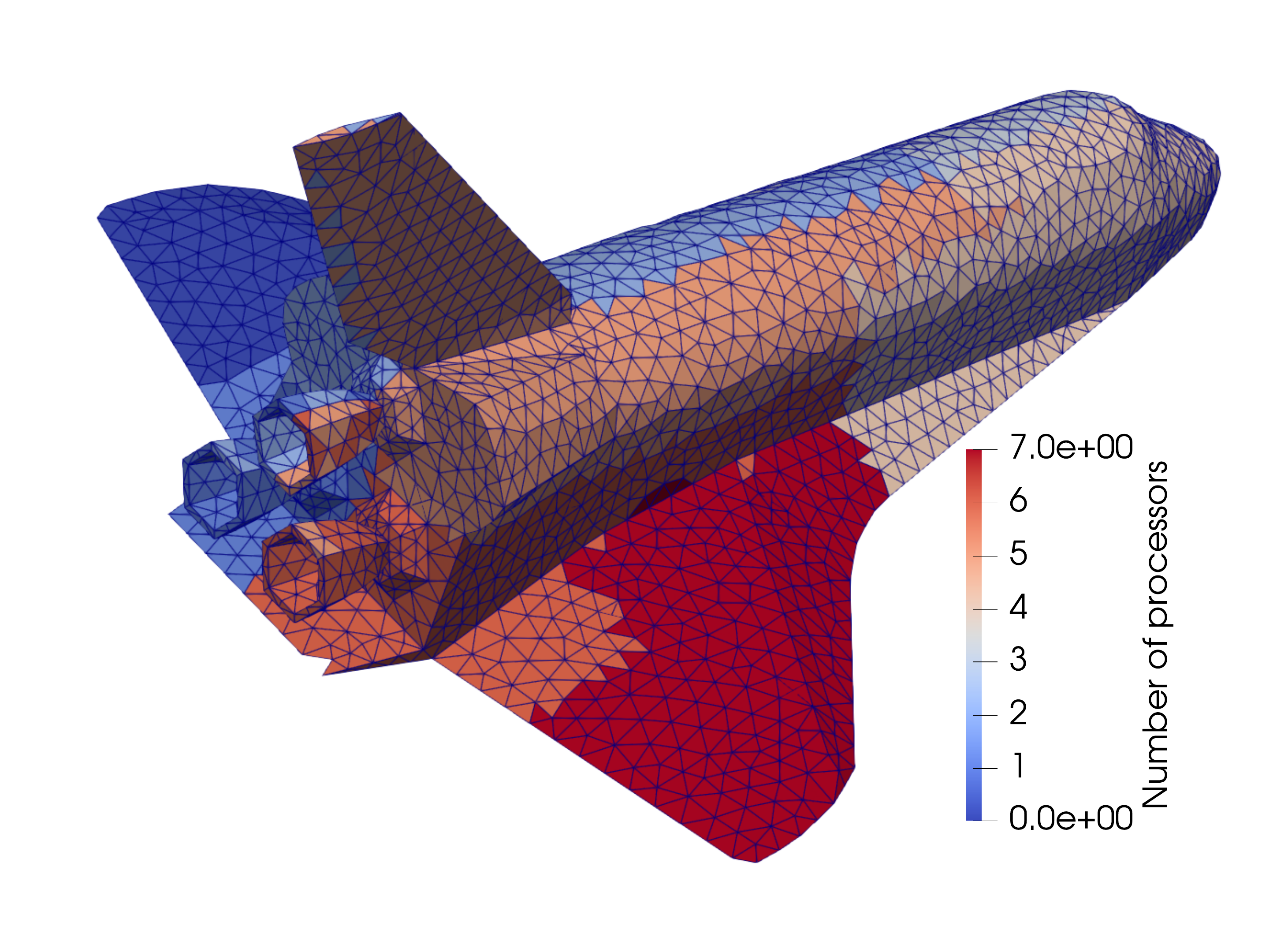}};
                \node[anchor=north east,inner sep=0] (image) at (2.0,1.5) {\includegraphics[width=0.42\textwidth]{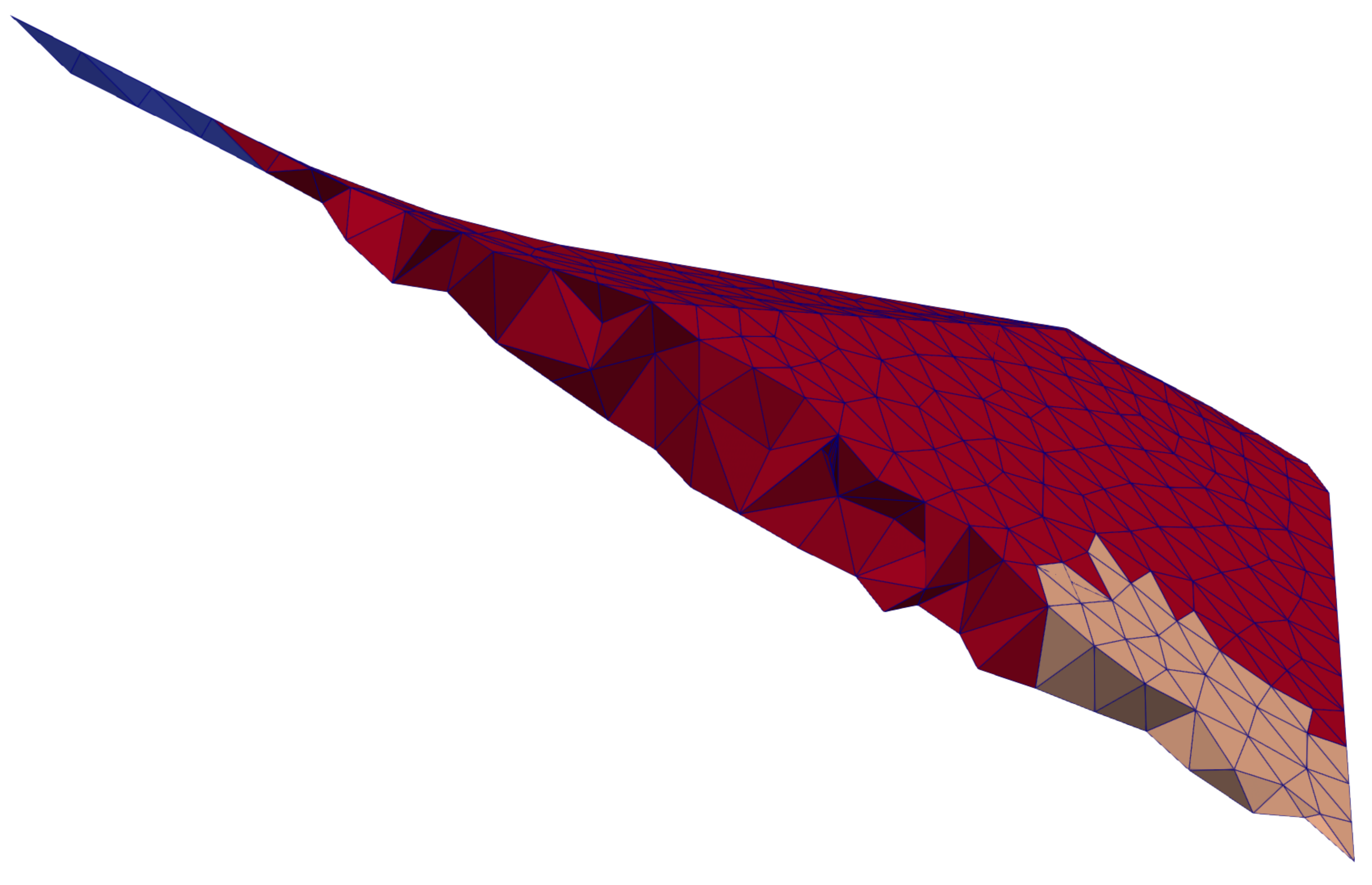}};
                \draw[dashed, black] (image.south west) rectangle (image.north east);
                \draw[dashed, black, line width=0.5pt] (2.0,1.0) -- (4.1, 2.0);
                \draw[thick, black, ->] (2.5,1.1) arc (35:160:0.5cm and 0.5cm);
            \end{tikzpicture}
		\end{minipage}
    }
    \caption{\color{blue} The space shuttle model: (a) geometry, (b) mesh generated by the CPAFT with Laplacian smoothing using (b) $np=1$, and (c) $np=8$. Both simulations lead to a same mesh with 19,973 elements.}
    \label{fig: space shuttle}
\end{figure}

\ 

\ 

\begin{figure}[H]
    \centering
		\subfigure[Geometry (.stl)]{
		\begin{minipage}[c]{0.31\textwidth}
		  \centering
		  \begin{tikzpicture}
                \node[anchor=south west,inner sep=0] (image) at (0, 0) {\includegraphics[height=1.35\textwidth]{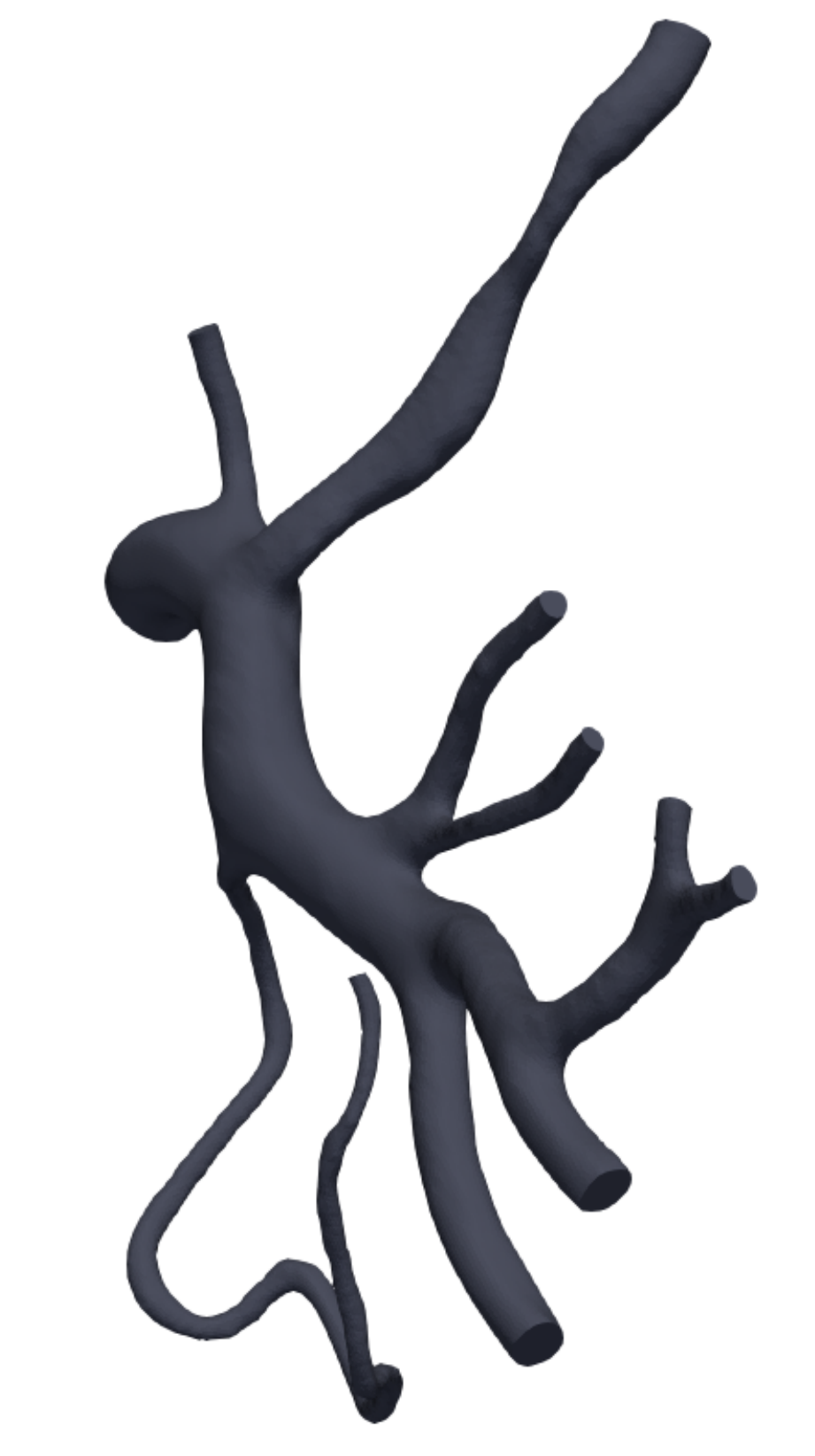}};
            \end{tikzpicture}
        \end{minipage}
		} \subfigure[$np = $ 1]{
		\begin{minipage}[c]{0.31\textwidth}
		  \centering
            \begin{tikzpicture}
                \node[anchor=south west,inner sep=0] (image) at (0, 0) {\includegraphics[height=1.35\textwidth]{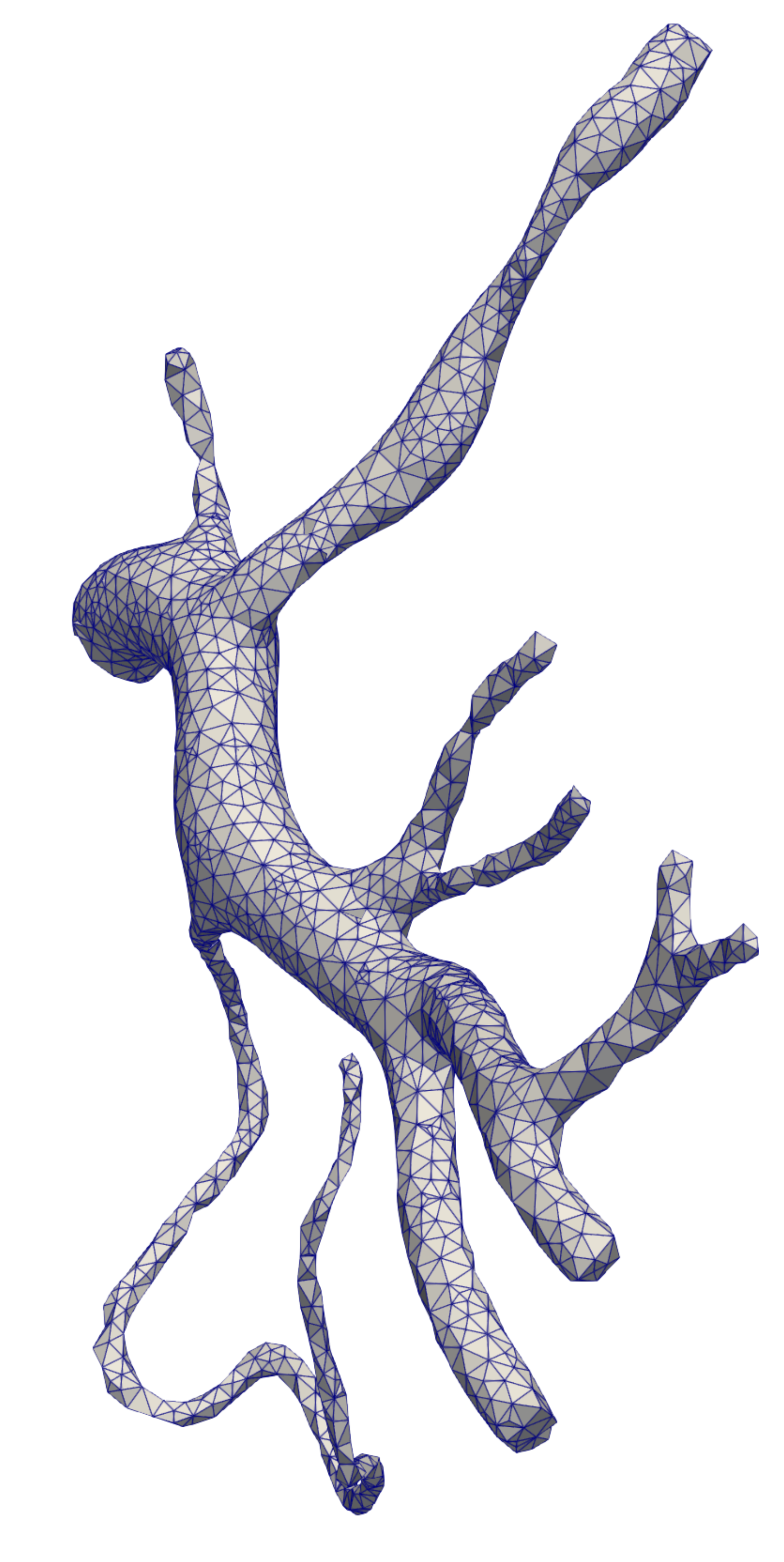}};
                \node[anchor=north east,inner sep=0] (image) at (4.5,6.9) {\includegraphics[height=0.7\textwidth]{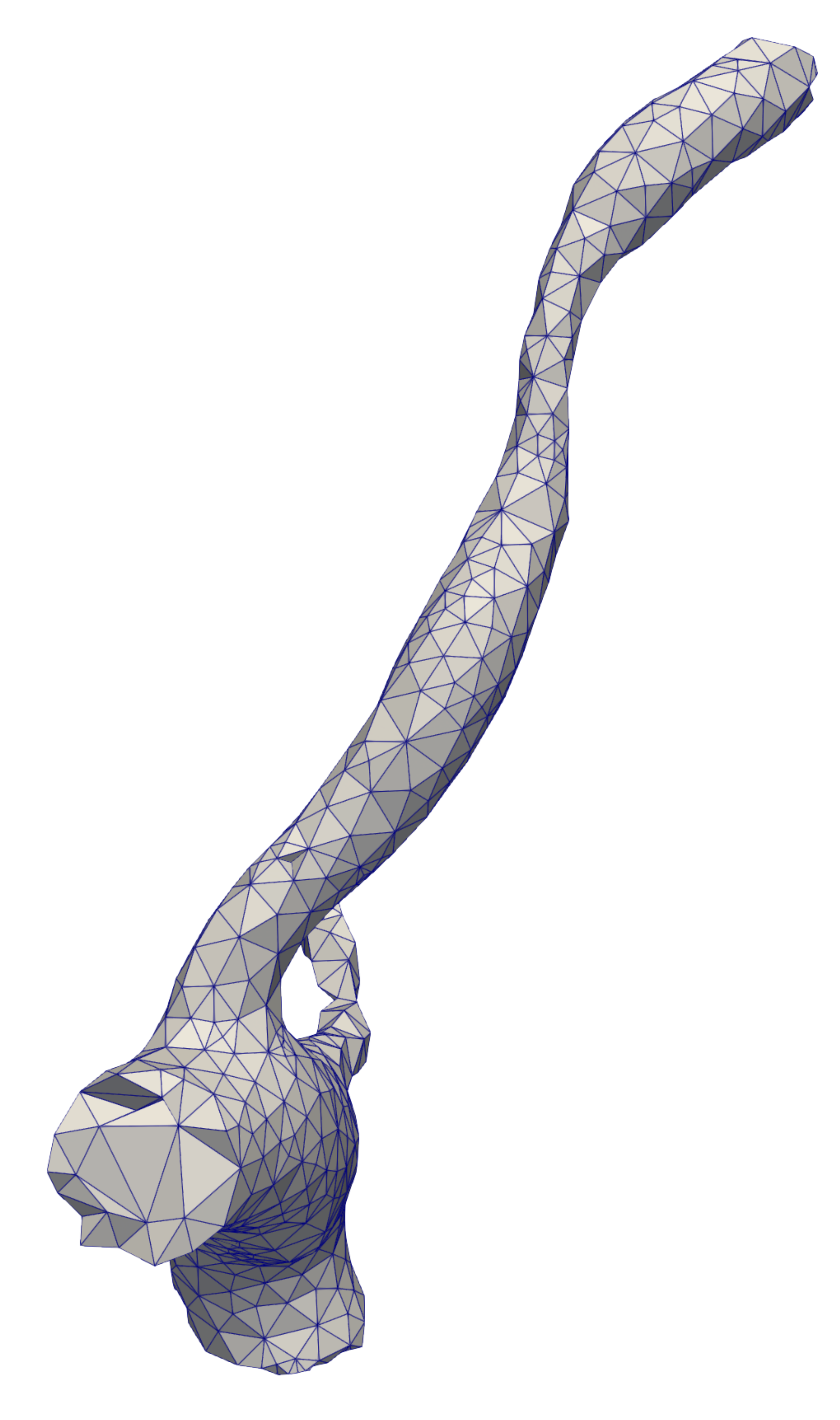}};
                \draw[dashed, black] (image.south west) rectangle (image.north east);
                \draw[dashed, black, line width=0.5pt] (0.35, 3.2+0.2)-- (0.4, 3.5+0.2) -- (1.6, 3.8+0.2) -- (1.55, 3.5+0.2);
                \draw[dashed, black, line width=0.5pt] (0.35, 3.2+0.2)--(0.75, 3.3+0.2);
                \draw[dashed, black, line width=0.5pt] (1.35, 3.45+0.2)--(1.55, 3.5+0.2);
                \draw[thick, black, ->] (1.7,4.05) arc (180:50:0.56cm and 0.3cm);
            \end{tikzpicture}
		\end{minipage}
		} \subfigure[$np = $ 8]{
		\begin{minipage}[c]{0.31\textwidth}
		  \centering
            \begin{tikzpicture}
                \node[anchor=south west,inner sep=0] (image) at (0, 0) {\includegraphics[height=1.35\textwidth]{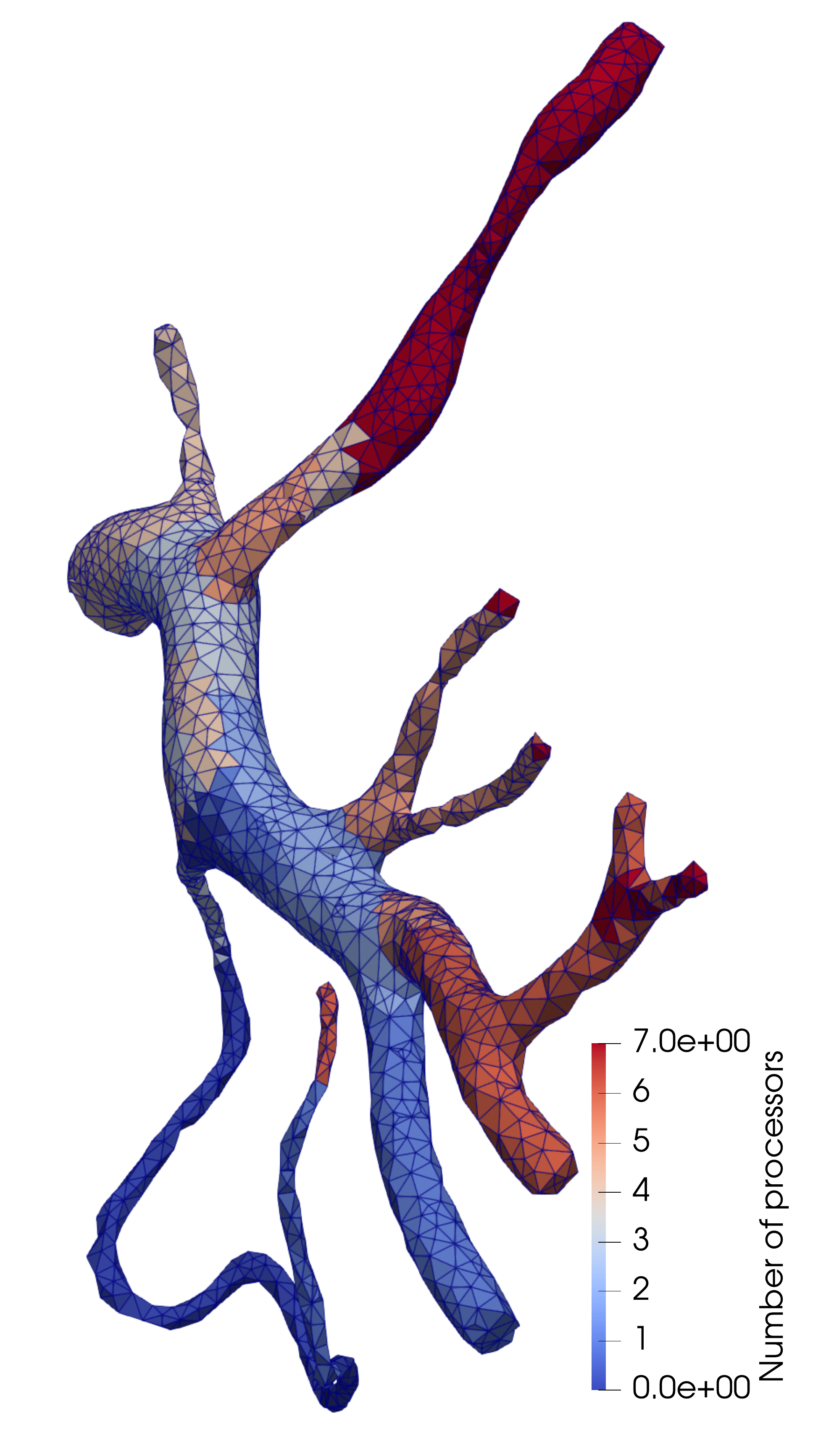}};
                \node[anchor=north east,inner sep=0] (image) at (4.5,6.9) {\includegraphics[height=0.7\textwidth]{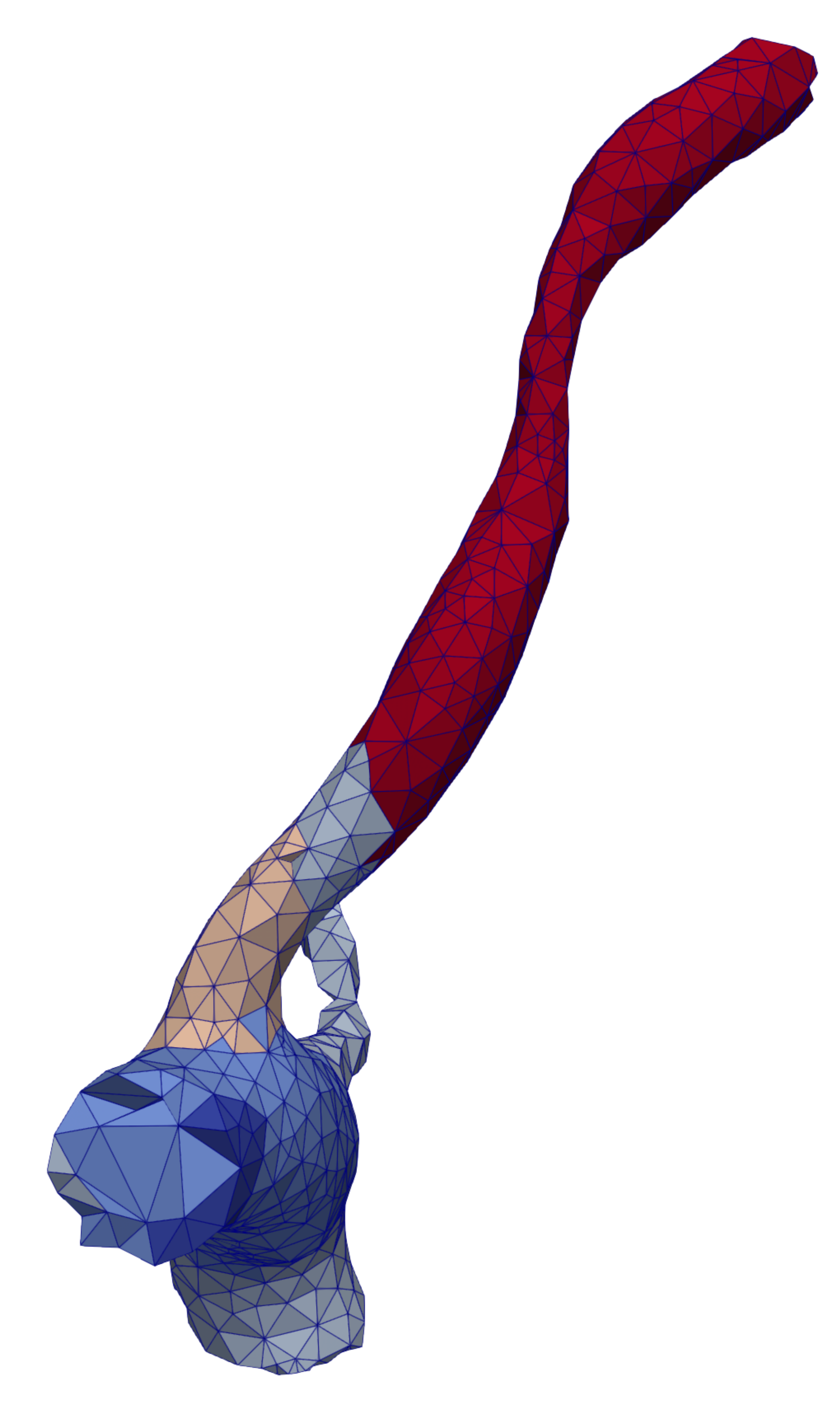}};
                \draw[dashed, black] (image.south west) rectangle (image.north east);
                \draw[dashed, black, line width=0.5pt] (0.35, 3.2+0.2)-- (0.4, 3.5+0.2) -- (1.6, 3.8+0.2) -- (1.55, 3.5+0.2);
                \draw[dashed, black, line width=0.5pt] (0.35, 3.2+0.2)--(0.75, 3.3+0.2);
                \draw[dashed, black, line width=0.5pt] (1.35, 3.45+0.2)--(1.55, 3.5+0.2);
                \draw[thick, black, ->] (1.7,4.05) arc (180:50:0.56cm and 0.3cm);
            \end{tikzpicture}
		\end{minipage}
		}
    \caption{\color{blue} The real artery vessel model: (a) geometry, (b) mesh generated by the CPAFT with Laplacian smoothing using (b) $np=1$, and (c) $np=8$. Both simulations lead to a same mesh with 9,846 elements.
}
    \label{fig: artery}
\end{figure}

\ 

\begin{figure}[H]
    \centering
    \subfigure[The three versions of CPAFT algorithm]{
        \begin{minipage}[c]{0.5\textwidth}
		  \centering
            \includegraphics[width=1.0\textwidth]{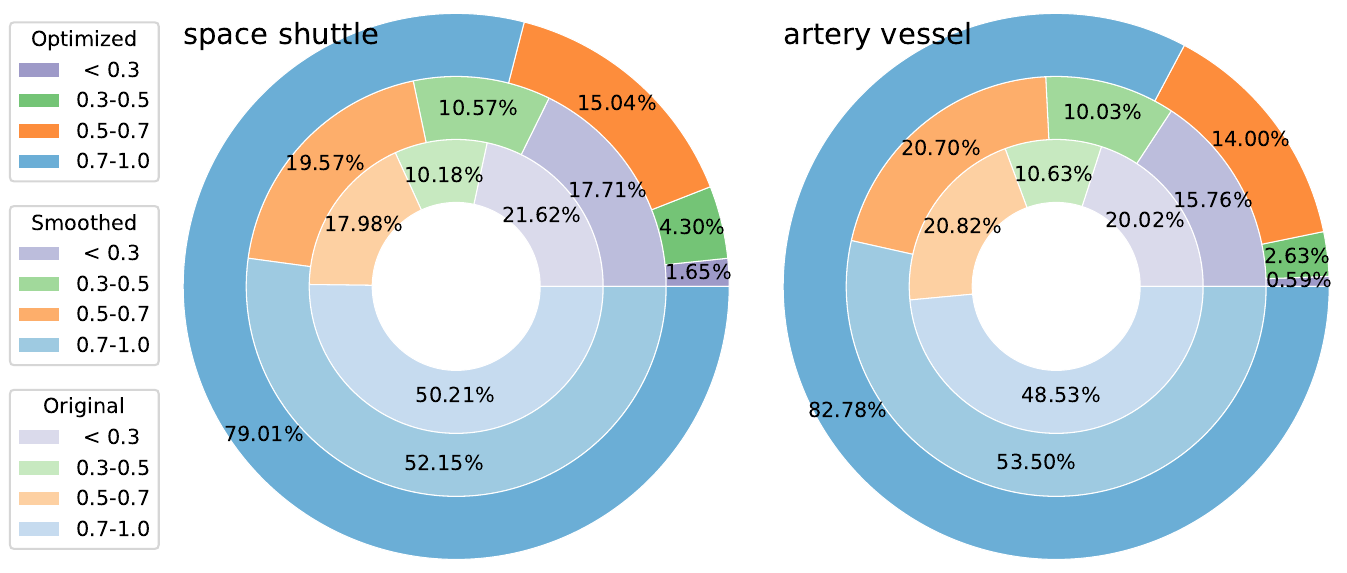}
        \end{minipage}
    }\subfigure[Reference software tools]{
		\begin{minipage}[c]{0.5\textwidth}
		  \centering
            \includegraphics[width=1.0\textwidth]{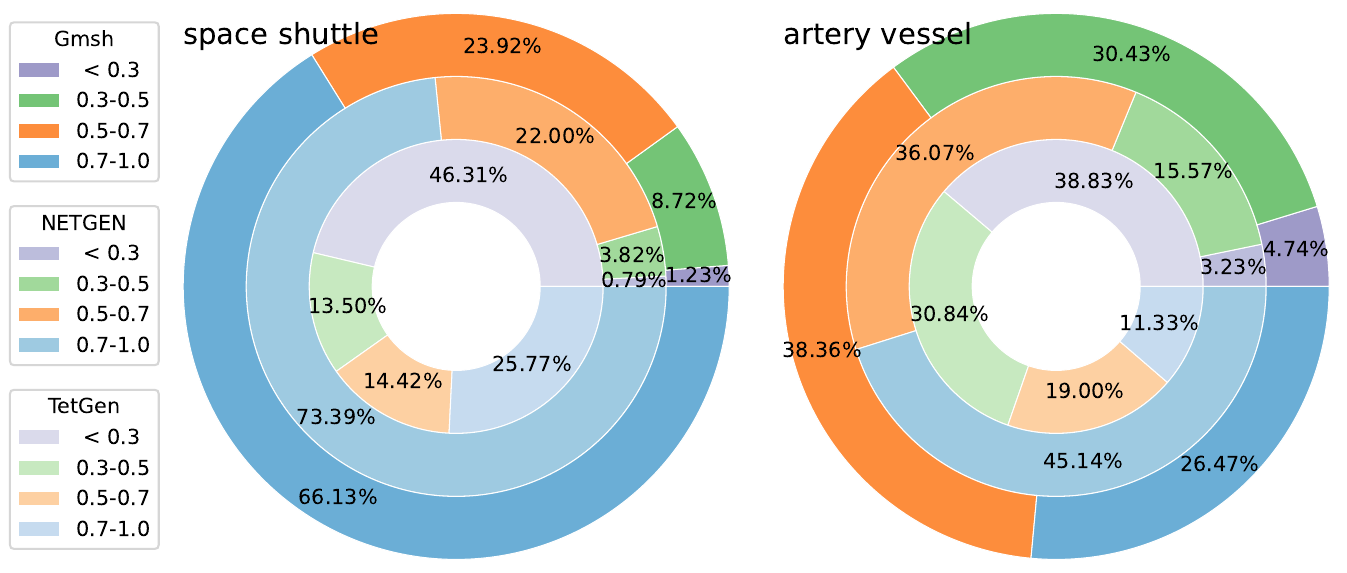}
        \end{minipage}
    }
    \caption{\color{blue} Breakdown of mesh quality indicators $\alpha$ for the space shuttle and real artery vessel model: (a) the original CPAFT algorithm (denoted as `Original'), the CPAFT algorithm with Laplacian smoothing (denoted as `Smoothed'), and the CPAFT algorithm with both Laplacian smoothing and local remeshing (denoted as `Optimized'), respectively;  (b) TetGen, NETGEN, and Gmsh, respectively.}
    \label{fig: quality-space-artery-op}
\end{figure}
}

\ 

\subsection{Parallel scalability}
\label{subsection: scalability}
{\color{red}
To study the parallel efficiency of the CPAFT algorithm, we use the artery vessel model and the space shuttle model as the benchmark test cases. In the first test case, the surface of the artery vessel model initially contains 9,982 triangles. Table~\ref{tab: efficiency comparison} presents the computing time, generation rate, and parallel efficiency for the CPAFT algorithm with both Laplacian smoothing and local remeshing, as well as for Gmsh and NETGEN. The CPAFT algorithm achieves a parallel efficiency of $47.50\%$ (i.e. $15.2$x speedup), while the reference software tools demonstrate less than $10\%$ parallel efficiency with 32 processors, highlighting the superior parallel efficiency of the CPAFT algorithm. Furthermore, the CPAFT algorithm exhibits a significant advantage in term of generation rate at 32 processors.

\begin{table}[H]
	\centering
	\caption{\color{red} The computing time, generation rate, and parallel efficiency for the CPAFT algorithm with both Laplacian smoothing and local remeshing, as well as for Gmsh and NETGEN. ``Time\_S (s)" and ``Time\_R (s)" represent the computing time for Laplacian smoothing and local remeshing in CPAFT algorithm, respectively. Here ``Eff ($\%$)$=T_1/(nT_n)$" denotes the parallel efficiency using $n$ processors. ``Time (s)" and ``Rate (k/s)" represent the total computing time in seconds and the elements generated per second, respectively. }
        \setlength{\tabcolsep}{4.5mm}{
        \center
        \begin{tabular} {|c|c|c|c|c|c|c|c|}
		  \hline
            &$np$ &$1$ &$2$ &$4$ &$8$ &$16$ &$32$ \\\cline{1-8}
            \multirow{5}{*}{CPAFT
            } 
            &Time\_S &0.0839 &0.0495 &0.0363 &0.0224 &0.0135 &0.0071 \\\cline{2-8}
            &Time\_R &1.2892 &0.7690 &0.5777 &0.2281 &0.1307 &0.1250 \\\cline{2-8}
            &Time &28.013 &14.697 &8.198 &4.764 &2.849 &1.843 \\\cline{2-8}
            &Rate &0.808 &1.539 &2.760 &4.749 &7.942 &12.277  \\\cline{2-8}
            &Eff($\%$) &$100\%$ &$95.30\%$ &$85.43\%$ &$73.50\%$ &$61.45\%$ &$47.50\%$ \\\cline{1-8}
            \multirow{3}{*}{Gmsh
            } &Time &2.445 &2.203 &1.999 &2.030 &2.158 &2.660 \\\cline{2-8}
            &Rate &8.629 &9.577 &10.555 &10.394 &9.777 &7.932 \\\cline{2-8}
            &Eff($\%$) &$100\%$ &$55.49\%$ &$30.58\%$ &$15.06\%$ &$7.08\%$ &$2.87\%$ \\\cline{1-8}
            \multirow{3}{*}{NETGEN
            } &Time &7.563 &7.433 &6.127 &5.439 &5.087 &5.084  \\\cline{2-8}
            &Rate &2.809 &2.858 &3.467 &3.905 &4.176 &4.178  \\\cline{2-8}
            &Eff($\%$) &$100\%$ &$50.87\%$ &$30.86\%$ &$17.38\%$ &$9.29\%$ &$4.65\%$\\\cline{1-8}
		\end{tabular}
        }
	\label{tab: efficiency comparison}
\end{table}
}

{\color{blue}


We then take the space shuttle model as the second benchmark test case and employ up to 2,048 processors to evaluate the parallel scalability. The CPAFT algorithm, with the default settings $(\beta_1, \beta_2)=(0.40, 0.15)$, is used to generate tetrahedral mesh. Five initial front sets, denoted as ${\cal F}_{\mathcal{L}_{i}}$ with $i=1, \cdots, 5$, are constructed. The front set ${\cal F}_{\mathcal{L}_{1}}$ corresponds to the surface mesh of the space shuttle model used in subsection~\ref{meshquality}. By dividing each triangular element of this surface mesh into four smaller triangles, we obtain a refined surface mesh, denoted as ${\cal F}_{\mathcal{L}_{2}}$. The other front sets are constructed in a similar manner. The total numbers of initial fronts in ${\cal F}_{\mathcal{L}_{1}}$, $\cdots$, ${\cal F}_{\mathcal{L}_{5}}$ are 0.222, 0.888, 3.551, 14.205, and 56.820 million, respectively. 
\begin{figure}[H]
    \centering
    \includegraphics[width=0.85\textwidth]{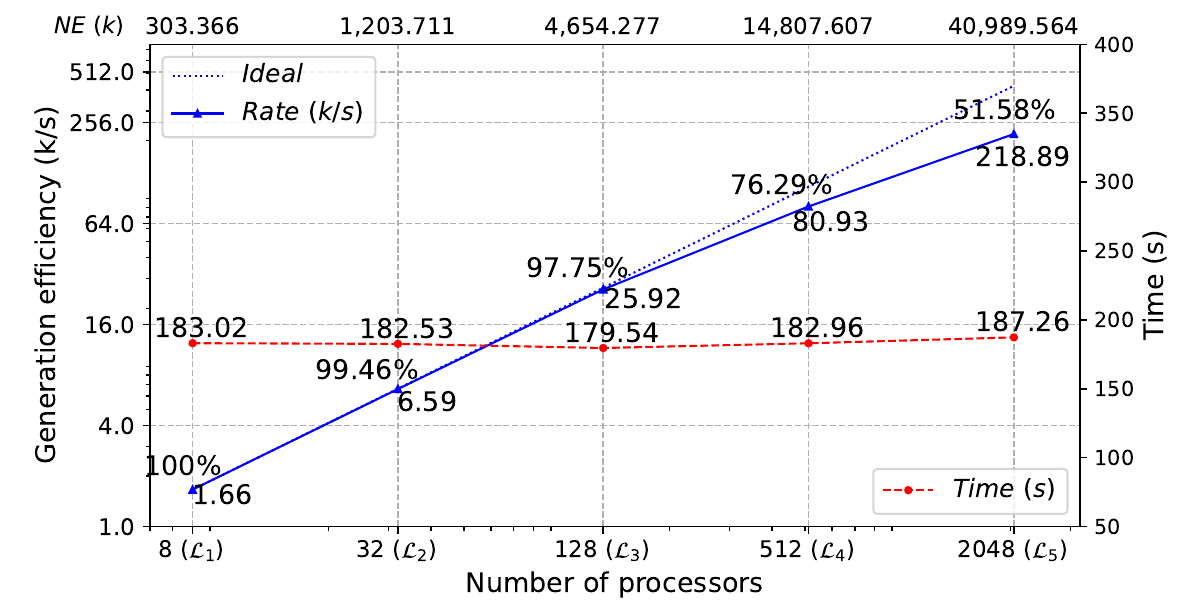}
    \caption{\color{blue} Weak scaling results with $np=$ 8, 32, 128, 512, and 2,048 processors, respectively. Here ``$NE (k)$" in the top of the figure denotes the total number of elements generated in 10 iterations. ``Time (s)" and ``Rate (k/s)" represent the total computing time in seconds and the elements generated per second, respectively.}
    \label{fig: fig-weak-scaling}
\end{figure}

To investigate the weak scalability, we apply the CPAFT algorithm to generate tetrahedral meshes using the initial front sets ${\cal F}_{\mathcal{L}_{1}}$, $\cdots$, and ${\cal F}_{\mathcal{L}_{5}}$. The corresponding numbers of processors used are 8, 32, 128, 512, and 2,048, respectively.  To match each initial front set, the level of SFC varies from 4 to 8.
In all simulations, the number of elements per processor is around 37k. The total computing time increases from 183.02 to 187.26, as the number of processors changes from 8 to 2,048, which shows the good weak scaling performance of the CPAFT algorithm. The mesh generation rate, as shown in Figure~\ref{fig: fig-weak-scaling}, indicates that the CPAFT algorithm achieves a generation efficiency of $51.58\%$ as the number of processors increases from 8 to 2,048.

We then use the initial front set ${\cal F}_{\mathcal{L}_{4}}$ to test the strong scalability of the newly proposed CPAFT algorithm. The numbers of processors used are 64, 128, 256, 512, and 1,024 processors. The numerical results, presented in Figure~\ref{fig: fig-strong-scaling}, demonstrate that the CPAFT algorithm obtains stable and robust performance with an appropriately large level of SFC. In particular, the algorithm scales efficiently from 64 to 1,024 processors, attaining a parallel efficiency of $60.33\%$ when the level of SFC is set to $7$. The parallel efficiency further increases to $64.39\%$ when the level of SFC is raised to $8$.

\begin{figure}[H]
    \centering
    \includegraphics[width=0.85\textwidth]{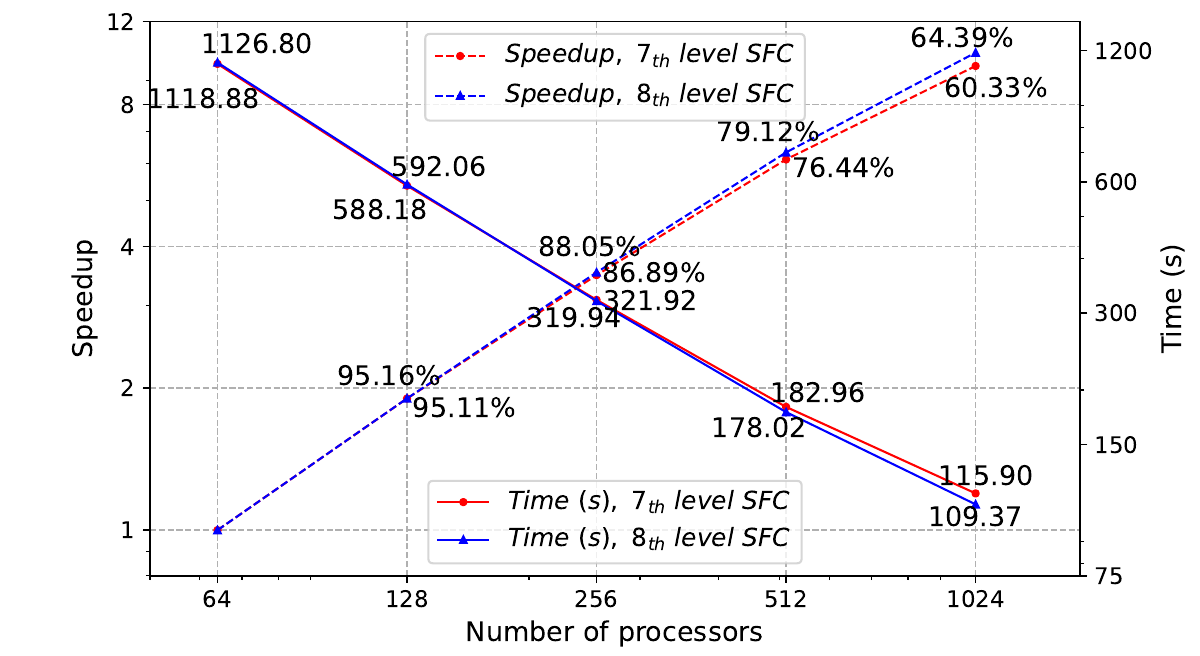}
    \caption{ \color{blue} Computing time and speedup for the CPAFT with the initial front set $\mathcal{F}_{\mathcal{L}_4}$ are evaluated using $np = $ 64, 128, 256, 512, and 1,024 processors. Two simulations are performed with $7_{th}$ and $8_{th}$ level of SFC. ``Time (s)" represents the computing time in seconds, while ``Speedup" indicates the strong scaling speedup achieved. 
    }
    \label{fig: fig-strong-scaling}
\end{figure}

}

\section{Conclusions}
\label{sec: conclusions}
The CPAFT algorithm was designed to efficiently generate large-scale unstructured triangular/tetrahedral meshes on multi-core CPU supercomputers, addressing the mesh intersection problem with parallel consistency and constructing the same unstructured mesh as the sequential AFT. To distribute the task of mesh generation to multiple processors, a non-overlapping domain decomposition was constructed based on the SFC Cartesian background meshes, resulting in a geometrically invariant global index for each front. By extending each subdomain with $\delta$ layers of background, an overlapping domain decomposition was generated to facilitate communication. To prevent newly generated elements from intersecting with existing ones, a distributed forest-of-overlapping-octrees (or quadtrees) was employed. Mutual mesh intersection judgment between newly generated elements was handled by a new consistent parallel MIS algorithm. Theoretical analysis substantiated the termination of the CPAFT algorithm after a finite number of iterations and verified its consistency with the sequential version mathematically. Several numerical simulations demonstrated that the CPAFT algorithm exhibits excellent performance in concave and non-uniform scenarios, enabling the generation of high-quality meshes and scaling effectively up to two thousand processors. 
In the future, we plan to further improve the performance of CPAFT and extend it on  heterogeneous architectures, such as multi-GPU systems. {\color{blue} Additionally, we will consider integrating rational coordinates into the algorithm to enhance its robustness, as mentioned in refs. \citep{richard1997adaptive, shewchuk1996triangle}. }

\section*{Acknowledgments}
This work was supported in part by the National Natural Science Foundation of China
(No. 12131002) and the Changsha Science and Technology Bureau (No. KH2301001).



\bibliographystyle{elsarticle-num-names} 
\bibliography{elsarticle-template}





\end{document}